\documentclass[11pt]{article}

\usepackage[latin1]{inputenc}
\usepackage{amsthm,amsmath,amssymb,bm}
\usepackage{graphicx}
\usepackage{color}
\usepackage{verbatim}
\usepackage{float}
\usepackage{mathrsfs}
\usepackage{hyperref}

\usepackage{array}   
\newcolumntype{C}{>{$}c<{$}}
\newcolumntype{L}{>{$}l<{$}}

\usepackage[short labels]{enumitem}
\setlist[enumerate]{topsep=0pt,itemsep=-1ex,partopsep=1ex,parsep=1ex}
\setlist[itemize]{topsep=0pt,itemsep=-1ex,partopsep=1ex,parsep=1ex}

\usepackage{bbm}

\theoremstyle{plain}
\newtheorem{theo}{Theorem}[section]
\newtheorem{prop}[theo]{Proposition}
\newtheorem{lemma}[theo]{Lemma}
\newtheorem{cor}[theo]{Corollary}

\newtheorem{prob}[theo]{Problem}

\theoremstyle{definition}
\newtheorem{defn}[theo]{Definition}

\newtheorem{eg}[theo]{Example}

\newcommand{\mc}[1]{\mathcal{#1}}
\newcommand{\mb}[1]{\mathbb{#1}}

\newcommand{\ms}[1]{\mathscr{#1}}

\newcommand{\sm}{\setminus}
\newcommand{\ov}{\overline}

\newcommand{\wt}{\widetilde}
\newcommand{\eps}{\varepsilon}

\newcommand{\aA}{\alpha}
\newcommand{\bB}{\beta}
\newcommand{\gG}{\gamma}
\newcommand{\dD}{\delta}

\newcommand{\zZ}{\zeta}
\newcommand{\lL}{\lambda}

\newcommand{\sS}{\sigma}

\newcommand{\DD}{\Delta}
\newcommand{\OO}{\Omega}

\newcommand{\Cols}{\mathscr{C}}
\newcommand{\con}{\mathrm{con}}
\newcommand{\abs}{\mathrm{abs}}
\newcommand{\app}{\mathrm{app}}
\newcommand{\col}{\mathrm{col}}
\newcommand{\vx}{\mathrm{vx}}
\newcommand{\colvx}{\mathrm{colvx}}

\usepackage[dvipsnames]{xcolor}

\title{Transversals via regularity}
\author{Yangyang Cheng\thanks{Mathematical Institute,
University of Oxford, Oxford, UK. Email: yangyang.cheng@maths.ox.ac.uk.}
\thanks{Yangyang Cheng was supported by a PhD studentship of ERC Advanced Grant 883810.}
\and Katherine Staden\thanks{School of Mathematics and Statistics, The Open University, Walton Hall, Milton Keynes, UK. Email: katherine.staden@open.ac.uk.}
\thanks{Katherine Staden was supported by EPSRC Fellowship EP/V025953/1.}}

\begin{document}

\maketitle

\begin{abstract}
Given graphs $G_1,\ldots,G_s$ all on the same vertex set and a graph $H$
with $e(H) \leq s$,
a copy of $H$ is \emph{transversal} or \emph{rainbow} if it contains
at most one edge from each $G_c$.
We study the case when $H$ is spanning and explore how the regularity
blow-up method, that has been so successful in the uncoloured setting,
can be used to find transversals.
We provide the analogues of the tools required
to apply this method in the transversal setting. Our main result is a blow-up lemma for
transversals that applies to separable bounded degree graphs $H$.

Our proofs use weak regularity in the $3$-uniform hypergraph whose edges are those $xyc$
where $xy$ is an edge in the graph $G_c$. 
We apply our lemma to give a large class of spanning $3$-uniform linear hypergraphs $H$ such that
any sufficiently large uniformly dense $n$-vertex $3$-uniform hypergraph with minimum vertex degree $\Omega(n^2)$ contains $H$ as a subhypergraph. 
This extends work of Lenz, Mubayi and Mycroft. 
\end{abstract}

\section{Introduction}

\subsection{Transversal embedding}

The problem of deciding whether an $n$-vertex graph $G$ contains a given subgraph $H$ is a central topic in graph theory.
Since this problem is NP-complete, much of the research on this topic has focused on finding sufficient
conditions on $G$ that guarantee the presence of $H$.
Given a graph parameter $\pi$, we seek the \emph{best possible} bound $\pi_{H,n}$ such that if $\pi(G) > \pi_{H,n}$, then $G$ contains a
copy of $H$, whereas there are graphs $G$ with $\pi(G)\leq\pi_{H,n}$ which are $H$-free.

In this paper, we investigate the generalisation of this problem to graph collections, also known as a multilayer graph.
Here, we are given graphs $G_1,\ldots,G_s$ on the same vertex set, where $s \geq e(H)$,
and we seek a \emph{transversal} copy of $H$, which is a copy of $H$ containing at most one
edge from each of the graphs $G_1,\ldots,G_s$.
We often think of each $G_c$ having the colour $c$, so a transversal copy of $H$ is
also called a \emph{rainbow} copy.

Again, we seek a best possible condition on $\pi$ which guarantees
the existence of $H$. That is, we seek the minimum $\pi_{H,n}^s$ such that if $G_1,\ldots,G_s$ are any graphs on the same vertex set of size $n$ and $\pi(G_c) > \pi_{H,n}^s$ for all $1 \leq c \leq s$,
then there is a transversal copy of $H$.
Observe that we recover the original `uncoloured' problem when $G_1=\ldots=G_s$,
so the transversal embedding problem is indeed a generalisation.
Ideally, we take $s=e(H)$ graphs in the collection and are most interested in determining $\pi_{H,n}^{\rm col} := \pi_{H,n}^{e(H)}$.

There has been a lot of recent progress in this area. The central question about transversal embeddings, posed by Joos and Kim~\cite{JK}, is whether the addition of colours changes the answer to the problem. That is, given some graph parameter $\pi$, when do we have $\pi_{H,n} = \pi_{H,n}^{\rm col}$?
When equality does hold, we say that the embedding problem is \emph{colour-blind}.

We already have a non-example of colour-blindness in the case of the triangle and the size parameter. Mantel's theorem from 1907 states that $e(G)>\lfloor n^2/4 \rfloor$ guarantees that an $n$-vertex graph $G$ contains a triangle,
whereas the complete balanced bipartite graph shows that this is best possible. So $e_{K_3,n}=\lfloor n^2/4\rfloor$.
However, Aharoni, DeVos, de la Maza, Montejano and \v{S}\'amal \cite{Aharoni} proved that whenever graphs $G_1,G_2,G_3$ on the same set of $n$ vertices satisfy $e(G_c)>an^2$ for all $c=1,2,3$, where $a:=\frac{26-2\sqrt{7}}{81} \approx 0.2557>1/4$, then there is a rainbow triangle,
while there is an example which shows $a$ cannot be decreased.
That is, $e_{K_3,n}^{\rm col} = \lfloor an^2 \rfloor$.
The transversal embedding problem for larger cliques is still open.

When one is interested in embedding a spanning graph, it makes sense to consider minimum degree
conditions rather than size conditions, as, for example, a graph could be very dense but contain an isolated vertex, and therefore not contain a copy of any given spanning graph with minimum degree at least $1$. Let $M_n$ and $C_n$ be the perfect matching and Hamilton cycle on $n$ vertices respectively.
Joos and Kim \cite{JK} proved that $\dD_{C_n,n}^{\rm col}=\dD_{C_n,n}$ and $\dD_{M_n,n}^{\rm col}=\dD_{M_n,n}$ (which both have the common value $=\lceil n/2\rceil-1$);
that is, Hamilton cycle embedding and matching embedding are both colour-blind with respect to minimum degree.

An easier question is to ask for approximately best possible conditions. For minimum degree, this means
we would like to find $\wt{\dD}_H$ and $\wt{\dD}_H^{\rm col}$, where
\begin{itemize}
\item $\wt{\dD}_H$ is the minimum $\dD$ such that for all $\aA>0$, any sufficiently large $n$-vertex graph $G$ with $\dD(G) \geq (\dD+\aA) n$ contains a copy of $H$; and
\item $\wt{\dD}_H^{\rm col}$ is the minimum $\dD$ such that for all $\aA>0$, any collection $G_1,\ldots,G_{e(H)}$ of sufficiently large $n$-vertex graphs with $\dD(G_c) \geq (\dD+\aA) n$ for all $1 \leq c \leq e(H)$ contains a transversal copy of $H$.
\end{itemize}
So the results of~\cite{JK} imply that $\wt{\dD}^{\rm col}_{C_n}=\wt{\dD}^{\rm col}_{M_n}=\frac{1}{2}$; this approximate
version was earlier proved by Cheng, Wang and Zhao~\cite{Cheng1}.
If $\wt{\dD}_H=\wt{\dD}_H^{\rm col}$, then we say that $H$ is \emph{approximately colour-blind}.

Montgomery, M\"uyesser and Pehova~\cite{MMP}
determined which $F$-factors are approximately colour-blind for any given small graph $F$,
and showed that spanning trees with maximum degree $o(\frac{n}{\log n})$ are approximately colour-blind.
They observed that some spanning graphs are very far from being colour-blind: taking $H$ to be the disjoint union of copies of $K_{2,3} \cup C_4$, it holds that $\wt{\dD}_H \leq \frac{4}{9}$ by a result of K\"uhn and Osthus~\cite{kuhn2009minimum}, whereas $\wt{\dD}_H^{\rm col} \geq \frac{1}{2}$.
Gupta, Hamann, M\"uyesser, Parczyk and Sgueglia~\cite{GHMPS} showed that powers of Hamilton cycles are approximately colour-blind.

They also showed, improving results of Cheng, Han, Wang, Wang and Yang~\cite{Cheng3}, that Hamilton $\ell$-cycles in $k$-uniform hypergraphs are approximately `$d$-colour-blind' for some range of $\ell$, $k$ and $d$,
which means with respect to minimum $d$-degree, which we do not define here.

The aim of~\cite{MMP} was to provide a fairly general approach for transversal embedding problems that used the corresponding uncoloured embedding result as a black box.
Roughly speaking, their approach is designed to work when the graph $H$ to be embedded is made up of small almost-disconnected blocks: they prove results for $F$-factors, which are made up of vertex-disjoint copies of $F$, and for trees which are made up of an ordered sequence of small subtrees each sharing one vertex with a previous tree.
Similarly,~\cite{GHMPS} gave a widely applicable sufficient condition for approximate colour-blindness, which applies to graphs obtained by `cyclically gluing' copies of a smaller graph (e.g.~a Hamilton cycle is obtained by `cyclically gluing' copies of an edge). Nevertheless, the $H$ to which these two results apply are fairly specific, and in this paper we build on~\cite{MMP} to develop a method for $H$ which are `more connected' than the above.

Very recently, during the preparation of this paper, Chakraborti, Im, Kim and Liu~\cite{CIKL} made important progress in this direction by proving a `transversal bandwidth theorem'.
A graph has \emph{bandwidth} $b$ if there is an ordering $v_1,\ldots,v_n$ of its vertices such that $|i-j| \leq b$ whenever $v_iv_j$ is an edge.
Their result states that if a bounded degree graph $H$ of sublinear bandwidth has chromatic number $k$, then $\wt{\dD}^{\rm col}_H\leq 1/k$.
This extends several of the above results, and is asymptotically best possible for many graphs $H$, as well as generalising the original bandwidth theorem of B\"ottcher, Schacht and Taraz~\cite{bottcher2009proof} for single graphs.
The proof uses similar ideas to the proof of the main result of our paper, a transversal blow-up lemma, which we introduce in the next section.

\subsection{The regularity-blow-up method for transversals}

The so-called `regularity-blow-up method' has been employed to prove many results concerning the embedding of a given spanning subgraph $H$ in a large graph $G$. Such proofs typically run along the following lines. Apply Szemer\'edi's regularity lemma~\cite{Szemeredi} to obtain a constant-size `reduced graph' $R$ of $G$, which approximates the structure of $G$: vertices of $R$ correspond to disjoint vertex clusters in $G$, and edges of $R$ correspond to \emph{regular} pairs of vertex clusters, which informally means that the edges between them are randomlike and therefore easy to embed into. There may be a small number of `exceptional' vertices which are not part of this structure. Next, find a suitable subgraph $H'$ of $R$ which is simpler than $H$, and often consists of many small connected components (whereas $H$ could be connected). Next, embed small pieces of $H$ which connect the components of $H'$. At this stage, some vertices may need to be moved from the structure into the exceptional set to make clusters balanced and to ensure regular pairs have sufficiently large minimum degree (they are \emph{superregular}). Then incorporate the exceptional vertices: for each such vertex $v$ find a component of $H'$ where $v$ has many suitable neighbours and a few vertices inside the $H'$-structure which can be used along with $v$ to embed a small part of $H$. Finally apply the \emph{blow-up lemma} of Koml\'os, S\"ark\'ozy and Szemer\'edi~\cite{komlos1997blow} to embed most of $H$ into $H'$ allowing for the restrictions imposed by the initial embedding. The blow-up lemma states that, for the purposes of embedding a possibly spanning bounded degree graph, a regular pair behaves the same as a complete bipartite graph.

The primary goal of this paper is to provide the tools needed to apply the regularity-blow-up method
to obtain transversal embeddings, following the same steps as above. 
The basic idea is that one can think of a graph collection $\bm{G}=(G_c: c \in \Cols)$ with common vertex set $V$ as the $3$-uniform hypergraph $G^{(3)}$ with vertex and edge sets
$$
V(G^{(3)}) = V \cup \Cols\quad\text{and}\quad E(G^{(3)}) = \{xyc: xy \in G_c\}. 
$$
This natural idea has been noticed and used before, for example in~\cite{JK}, where it was noted that a result of Aharoni, Georgakoupoulos and Spr\"ussel~\cite{aharoni2009perfect} on matchings in $k$-partite $k$-uniform hypergraphs can be used to find a transversal matching in a bipartite graph collection; and in \cite{Cheng3}, where the weak regularity lemma was applied to $G^{(3)}$ to find an almost spanning clique factor.  
In this paper, we take the idea further by using this perspective to provide the tools one would require to use the regularity-blow-up method for graph collections.
To this end, we
\begin{itemize}
\item define regularity, superregularity, reduced graphs and provide a regularity lemma for graph collections;
\item state some of the standard tools for regularity arguments, such as a `slicing lemma' which states that regularity is inherited in subpairs of regular pairs, degree inheritance in the reduced graph, and various embedding lemmas with target sets, candidate sets and prescribed colours;
\item prove a blow-up lemma for `separable graphs', which we define shortly.
\end{itemize}
The first item is mainly a convenient reformulation of weak regularity for $3$-uniform hypergraphs
and does not require any new ideas; nonetheless the statements are useful to have.
The third (which builds upon the tools developed in the second item) is our main contribution, and the proof uses the original blow-up lemma combined with colour absorption ideas from~\cite{MMP}.

We envisage that our blow-up lemma will be useful in solving various transversal embedding problems, making it a viable alternative to the commonly used absorption technique. Furthermore, there are cases where the blow-up lemma is a more appropriate tool. For instance, when attempting to embed a given transversal $2$-factor (a spanning $2$-regular subgraph) or, more generally, a graph with low bandwidth, the absorption technique is not as effective. In such situations, the blow-up lemma can still be employed for successful embedding. For example, our blow-up lemma can be used to embed powers of Hamilton cycles, providing an alternative proof of a result in~\cite{GHMPS}. 
Abbasi~\cite{abbasi1998spanning}, proving a conjecture of El-Zahar~\cite{el1984circuits}, used the original blow-up lemma to obtain the best possible minimum degree bound in a large graph containing a given $2$-factor; our blow-up lemma could be useful in proving a transversal version of this result.
Another notable advantage of the blow-up lemma is its ability to aid in characterising extremal constructions. This is helpful for determining exact bounds in these types of problems, by utilising the stability method commonly employed in non-transversal graph embeddings. This natural problem was recently raised in \cite{GHMPS}. 
We discuss applications to transversal embedding problems further in Section~\ref{sec:conclude}.

The secondary goal of the paper is to apply our transversal blow-up lemma to embeddings in uniformly dense $3$-uniform hypergraphs.
We introduce these `randomlike' hypergraphs, and our results in this direction, in Section~\ref{sec:quasi}.
Indeed, our blow-up lemma can be formulated as a blow-up lemma for $3$-uniform hypergraphs. Consequently, there is also potential to employ this lemma in generalising results such as those of K\"uhn and Osthus~\cite{kuhn2006loose} who studied the problem of embedding loose Hamiltonian cycles under vertex degree conditions.

All our results apply to `separable graphs' with bounded maximum degree.
An $n$-vertex graph $H$ is \emph{$\mu$-separable} if there is $X \subseteq V(H)$ of size at most $\mu n$ such that $H-X$ consists of components of size at most $\mu n$.
Separable graphs (with suitable small $\mu$) include $F$-factors for fixed $F$, trees, $2$-regular graphs, powers of a Hamilton cycle, and graphs of small bandwidth.



\subsection{Transversals in uniformly dense graph collections}\label{sec:quasicol}

Our first main result concerns transversal embeddings of spanning graphs inside a quasirandom graph collection. A `quasirandom' condition is one possessed by a random graph of a similar density. Our condition is that for every linear subset of vertices and linear subset of colours, the total number of edges in these colours spanned by the subset is large.
A graph collection satisfying this condition is said to be \emph{uniformly dense}.
We also require that the number of edges of each colour and the total degree of each vertex is $\Omega(n^2)$, where the graph collection has a common vertex set of size $n$. This condition cannot be completely removed, since if a colour is not present or if a vertex is isolated, we certainly cannot find a transversal copy of a spanning graph $H$.
A graph collection satisfying both conditions is \emph{super uniformly dense}, in analogy with regular and superregular.
We are interested in the following question: for which (spanning) graphs $H$ must any super uniformly dense graph collection contain a transversal copy of $H$?

\begin{defn}[Uniformly dense graph collection]
Given $d,\eta>0$, we say that a graph collection $\bm{G}=(G_c: c \in \Cols)$ with vertex set $V$ of size $n$ is \emph{$(d,\eta)$-dense} if
for all $A,B \subseteq V$ and $\Cols' \subseteq \Cols$, we have
    $$
    \sum_{c \in \Cols'}e(G_c[A,B]) \geq d|\Cols'||A||B| - \eta n^3,
    $$
where $e(G_c[A,B])$ denotes the number of pairs $(x,y) \in A \times B$
for which $\{x,y\}$ is an edge of $G_c$.
We informally refer to a graph collection $\bm{G}$ on $n$ vertices that is $(d,\eta)$-dense for parameters $0<1/n \ll \eta \ll d \leq 1$ as \emph{uniformly dense}, and if additionally there is $\aA \gg \eta$ for which $\sum_{c \in \Cols}d_{G_c}(v) \geq \aA|\Cols| n$ for all $v \in V$ and $e(G_c) \geq \aA n^2$ for all $c \in \Cols$, then $\bm{G}$ is \emph{super uniformly dense}.
\end{defn}

Note that the condition is vacuous unless $A,B$ and $\Cols'$ are of linear size (in the number $n$ of vertices).
We prove that a transversal copy of a separable graph can be found in a super uniformly dense graph collection.

\begin{theo}\label{th:quasi}
For all $\DD,d,\dD,\aA>0$, there are $\eta,\mu>0$ and $n_0 \in \mb{N}$ such that the following holds for all integers $n \geq n_0$.
Let $\bm{G}=(G_c: c \in \Cols)$ be a $(d,\eta)$-dense graph collection on a common vertex set $V$ of size $n$, where $|\Cols| \geq \dD n$, and suppose that
for all $v \in V$ we have $\sum_{c \in \Cols}d_{G_c}(v) \geq \aA|\Cols|n$,
and for all $c \in \Cols$ we have $e(G_c) \geq \aA n^2$.
Then $\bm{G}$ contains a transversal copy of any given $\mu$-separable graph $H$ on $n$ vertices with $\DD(H) \leq \DD$ and $e(H) = |\Cols|$.
\end{theo}

This result is a consequence of a more general transversal blow-up lemma,
which we state in the next section.
First, we discuss possible extensions to other colour patterns. 
One can ask whether a graph collection contains a copy of $H$ with a given colour pattern, extending the monochromatic case (a copy in a single graph $G_c$) and rainbow case (a copy with one edge from each $G_c$).
For example, the following was shown in~\cite{MMP}. Not only does the same minimum degree $(\frac{2}{3}+o(1))3n$ which guarantees a triangle factor in a large graph with a vertex set of size $3n$ in fact guarantee a transversal copy when $|\Cols|=3n$,
it also guarantees a triangle factor where the $c$-th triangle lies inside $G_c$ (has colour $c$), when $|\Cols|=n$.

Given $d,\eta>0$, a (single) graph with a vertex set $V$ of size $n$ is \emph{$(d,\eta)$-dense} if for all $A,B \subseteq V$, we have $e(G[A,B]) \geq d|A||B|-\eta n^2$.
A collection of $(d,\eta)$-dense graphs on the same vertex set is a $(d,\eta)$-dense graph collection, but the converse does not hold. 
It is easy to see that, for any $d>0$, as long as $\eta$ and $1/n$ are sufficiently small, such a $G$ contains $\Omega(n^3)$ triangles: the condition applied to $A=B=V$ implies that there are $\Omega(n)$ vertices $v$ with $d_G(v) \geq dn/2$; each one has $\Omega(n^2)$ edges in its neighbourhood.

However, there are uniformly dense graph collections with $d$ fairly large which do not contain any monochromatic triangles, as the following example shows.
(The example is essentially equivalent to one in the setting of uniformly dense hypergraphs -- introduced in Section~\ref{sec:quasi} -- of Reiher, R\"odl and Schacht~\cite{reiher2018turan}, which itself has its roots in work of Erd\H{o}s and Hajnal~\cite{erdos1972ramsey}.)
Form an auxiliary oriented $2$-graph $J$ by letting $V,\Cols$ be disjoint, where, say, $|V|=|\Cols|=n$ is large, and first adding edges between every pair in $V$, and every pair in $(V,\Cols)$.
Independently for each edge, choose an orientation uniformly at random. Add the pair $xy$ to $G_c$ with $x,y \in V$ and $c \in \Cols$
precisely when $xyc$ is a cyclic triangle.
Then, with probability tending to $1$ as $n \to \infty$, $\bm{G}$ is $(\frac{1}{4},o(1))$-dense.
However, for every triple $x,y,z$ of vertices in $V$ and every colour $c \in \Cols$, there are at most two edges in $G_c[\{x,y,z\}]$.

This raises the question of which colour patterns of which graphs $H$ one can expect to find in a (super) uniformly dense graph collection.
If $d$ is sufficiently large, then any bounded degree $H$ with any colour pattern can be found; for those $H$ which are not present when $d$ is an arbitrary positive constant, how large must $d$ be to guarantee such a copy of $H$?
We will explore a generalisation of this question in Section~\ref{sec:quasi}.

\subsection{A transversal blow-up lemma}

In this section, we state a simplified version of our blow-up lemma for bipartite graphs.
We defer the complete statement to Theorem~\ref{th:blowupintro} at the end of this section.

Our blow-up lemma has the advantages that its proof, sketched at the beginning of Section~\ref{sec:blowup}, is conceptually straightforward,
following from the original blow-up lemma and a colour absorption tool introduced in~\cite{MMP};
and the lemma should be powerful enough for all of the transversal embedding applications we have in mind.
This is since the usual regularity-blow-up method is almost always successfully applied to `non-expanding' separable graphs, since one embeds them by using the blow-up lemma on a series of small pieces. 
The disadvantage is that there does not seem to be a reason why the separability condition should be necessary.

\begin{theo}[Simplified transversal blow-up lemma]\label{th:bip}
Let $0 < 1/n \ll \eps,\mu \ll d,\dD,1/\DD \leq 1$.
Let $\Cols$ be a set of at least $\dD n$ colours and let $\bm{G} = (G_c: c \in \Cols)$ be a collection of bipartite graphs with the same vertex partition $V_1,V_2$,
where $n \leq |V_1|\leq|V_2|\leq n/\dD$, such that
\begin{itemize}
\item for all $V_i' \subseteq V_i$ for $i=1,2$ and $\Cols' \subseteq \Cols$
with $|V_i'|\geq \eps|V_i|$ for $i=1,2$ and $|\Cols'| \geq \eps |\Cols|$, we have that
$$
\sum_{c \in \Cols'}e_{G_c}(V_1',V_2') \geq d|\Cols'||V_1'||V_2'|;
$$
\item for $i=1,2$ and every $v \in V_i$ we have $\sum_{c \in \Cols}d_{G_c}(v) \geq d|\Cols|n$
and for every $c \in \Cols$, we have $e(G_c) \geq dn^2$.
\end{itemize}
Let $H$ be a $\mu$-separable bipartite graph with parts of size $|V_1|,|V_2|$, and $|\Cols|$ edges and maximum degree $\DD$.
Then $G$ contains a transversal copy of $H$.
\end{theo}

The theorem requires that the number of vertices and edges of $H$ are comparable, so it does not apply to very sparse $H$.
If one adds edges to $H$ to obtain a suitable denser graph $H'$, and duplicates colours until there are $e(H')$ colours, the copy of $H'$ produced by the theorem will not necessarily contain a transversal copy of $H$.

The general version of the theorem, Theorem~\ref{th:blowupintro}, applies to a graph collection whose graphs have common parts $V_1,\ldots,V_r$, and a graph $R$ with $V(R)=[r]$ such that the bipartite condition in the theorem holds between all $ij \in E(R)$, and each $ij$ has a dedicated set of colours $\Cols_{ij}$.

\subsubsection{Rainbow blow-up lemmas}

There are by now many blow-up lemmas for various settings which have been applied to many embedding problems.
Of particular relevance to this paper are \emph{rainbow blow-up lemmas} which apply to a single graph whose edges are coloured.
We say that an edge-coloured graph $G$ is \emph{$k$-bounded} if no colour appears on more than $k$ edges, and $G$ is \emph{locally $k$-bounded} if each vertex is incident to at most $k$ edges of the same colour.
Glock and Joos~\cite{SJ} proved a rainbow blow-up lemma for $o(n)$-bounded edge-colourings, which allows one to find a rainbow embedding of a given bounded degree graph $H$. Here, the number of colours is many times larger than $e(H)$. Later, Ehard, Glock, and Joos~\cite{EGJ} proved a similar lemma for locally $O(1)$-bounded colourings, which allows the number of colours to be $(1+o(1))e(H)$. 
Therefore it does not seem possible to use such blow-up lemmas for transversal embedding problems where we require exactly $e(H)$ colours.


\subsection{Applications to embedding in uniformly dense hypergraphs}\label{sec:quasi}

In this section we revisit the hypergraph perspective and discuss the consequences of our work to embeddings in $3$-uniform hypergraphs.
`Weak regularity' is the straightforward generalisation of Szemer\'edi regularity from ($2$-uniform hyper)graphs to $k$-uniform hypergraphs (hereafter referred to as \emph{$k$-graphs}), and the `weak regularity lemma' was proved by Chung~\cite{fanweak} in much the same way as the original~\cite{Szemeredi}. However, this lemma is not powerful enough to prove many of the analogues of graph results which one would like. In particular, there is no general \emph{counting lemma} which guarantees that the number of copies of a given small graph $F$ in a regularity partition is similar to what one would expect if pairs were replaced by uniform random hyperedges, with the same density. Thus the `strong regularity lemma' was developed, which uses, as the name suggests, a more complicated and much stronger notion of regularity, and does have an associated counting lemma; this lemma is much more applicable than its weaker counterpart.

It was shown by Conlon, H\`an, Person and Schacht~\cite{conlon2012weak} and independently Kohayakawa, Nagle, R\"odl and Schacht~\cite{kohayakawa2010weak} that weak regularity is in fact strong enough to give a counting lemma for \emph{linear} hypergraphs, which have the property that $|e \cap f| \leq 1$ for all distinct hyperedges $e,f$.
As a transversal copy of a (simple) graph inside a graph collection $\bm{G}$ is a linear subhypergraph of the associated $3$-graph $G^{(3)}$, weak regularity is an effective tool for transversal embedding problems.
And conversely, the tools developed in this paper are useful for embedding $3$-graphs.

We generalise the definition we gave in Section~\ref{sec:quasicol} for graphs.

\begin{defn}[Uniformly dense $k$-graph] Let $d,\eta>0$ and let $k \geq 2$ be an integer.
A $k$-graph is \emph{$(d,\eta)$-dense} if for all subsets $U_1,\ldots,U_k \subseteq V$, we have
$$
e(U_1,\ldots,U_k) \geq d|U_1|\ldots |U_k| - \eta n^k,
$$
where $e(U_1,\ldots,U_k)$ is the number of $k$-tuples 
$(x_1,\ldots,x_k) \in U_1\times \ldots \times U_k$
for which $\{x_1,\ldots,x_k\}$ is a hyperedge.
We informally refer to a $k$-graph $G$ on $n$ vertices that is $(d,\eta)$-dense for parameters $0<1/n \ll \eta \ll d \leq 1$ as \emph{uniformly dense}.
If there is $\aA \gg \eta$ for which additionally $d_G(v) \geq \aA n^{k-1}$ for all $v \in V(G)$, then we say that $G$ is \emph{super uniformly dense}.
\end{defn}

This is a type of quasirandomness, since a random $k$-graph of density at least $d$ is $(d,o(1))$-dense with high probability. Several papers studying uniform density use the term `quasirandom' instead.

Following a suggestion of Erd\H{o}s and S\'os~\cite{erdos1982ramsey}, a systematic treatment of extremal problems in uniformly dense hypergraphs has been started by R\"odl, Reiher and Schacht, in part due to the great difficulty of such problems in general hypergraphs. 
In~\cite{reiher2018hypergraphs}, they fully answered the zero Tur\'an density question for uniformly dense $3$-graphs, i.e.~for which $F$ is the following true? For all $d>0$, there exists $\eta>0$ such that every sufficiently large $(d,\eta)$-dense $3$-graph contains $F$ as a subhypergraph.
They showed that these $F$ are precisely those with the following property. There is an ordering $v_1,\ldots,v_r$ of $V(F)$ and a colouring of the set of pairs of vertices contained in edges by red, blue and green, so that whenever $v_iv_jv_k \in E(F)$ with $i<j<k$, we have that $v_iv_j$ is red, $v_iv_k$ is blue and $v_jv_k$ is green.
This set of $F$ includes linear $F$ and $3$-partite $F$ (Erd\H{o}s~\cite{erdos1964extremal} proved that a $k$-graph $F$ has Tur\'an density $0$ if and only if $F$ is $k$-partite), but is much richer than this: for example, the hypergraph obtained by removing one edge from the tight cycle on $5$ vertices is such a hypergraph.

The simple argument given in Section~\ref{sec:quasicol} shows that for $k=2$ and $\DD>0$, a sufficiently large uniformly dense $2$-graph contains a copy (in fact many copies) of $K_\DD$.
In contrast, for $k \geq 3$ there are very simple $k$-graphs $F$ which require a fairly large density $d$ to appear in any uniformly dense $k$-graph.
Indeed, there is a $(\frac{1}{4},o(1))$-dense $3$-graph in which $K^{(3)-}_4$, the (unique) $3$-graph with $4$ vertices and $3$ edges, does not appear. (The example is closely related to the one given in Section~\ref{sec:quasicol}.) 

In this paper, we are interested in which \emph{spanning} hypergraphs appear in any super uniformly dense hypergraph.
Note that we cannot remove `super' since uniform density does not preclude the existence of isolated vertices.

\begin{prob}
Let $\mc{H} = (H_\ell: \ell \in \mb{N})$ be a collection of $k$-graphs where $|V(H_\ell)|=: n_\ell \to \infty$.
For which $\mc{H}$ is the following true?
For all $d,\aA>0$, there exist $\eta>0$ and $\ell_0 \in \mb{N}$ such that for all integers $\ell>\ell_0$, every $(d,\eta)$-dense $k$-graph $G$ on $n_\ell$ vertices with $d_G(v) \geq \aA n_\ell^{k-1}$ for all $v \in V(G)$ contains $H_\ell$ as a subhypergraph.
\end{prob}

For $k=2$, the blow-up lemma of Koml\'os, S\'ark\"ozy and Szemer\'edi~\cite{komlos1997blow} implies that every uniformly dense graph contains as a subgraph any given spanning graph $H$ with bounded degree and sublinear bandwidth (as observed by Glock and Joos, see Theorem~9.3 in~\cite{ST}).

For general uniformities $k$, Lenz and Mubayi~\cite{lenz2016perfect} proved that uniformly dense $k$-graphs contain an $F$-factor for any given fixed size linear $F$. In fact they found non-linear $F$ for which uniform density is sufficient, and `almost-linear' $F$ for which it is not.
Ding, Han, Sun, Wang and Zhou~\cite{ding2022f} completed this work for $k=3$ by characterising those $F$ for which uniform density guarantees an $F$-factor,
and for general $k$ they also obtained a characterisation for $k$-partite $F$.
Lenz, Mycroft and Mubayi~\cite{lenz2016hamilton} showed that uniform density guarantees a loose cycle.
To the best of our knowledge, there are no further results on embedding spanning $F$ in uniformly dense hypergraphs of arbitrarily small positive density.

Our next main result generalises the result of~\cite{lenz2016hamilton} for $k=3$ by providing a new family of spanning $H$ which can be found in super uniformly dense $3$-graphs.
An \emph{expanded graph} is obtained from a $2$-graph by replacing every edge $e=xy$ by some $3$-edges $xyc_1,\ldots,xyc_{t_e}$, where all the new vertices $c_1,\ldots,c_{t_e}$ are distinct, and the number $t_e$ of new edges/vertices for the edge $e$ depends on $e$. If every $t_e$ is equal to $t$ then we call this a \emph{$t$-expansion}. For example, a loose $3$-uniform cycle is an expanded cycle where each edge is replaced by one expanded edge, that is, a $1$-expansion of a cycle.
The $1$-expansion of a (simple) graph $H$ is linear and has $|V(H)|+|E(H)|$ vertices.

\begin{theo}\label{th:quasi3}
For all $\DD,d,\aA>0$ there are $\eta,\mu>0$ and $n_0 \in \mb{N}$ such that the following holds for all integers $n \geq n_0$.
Let $G$ be a $(d,\eta)$-dense
$3$-graph on $n$ vertices with $d_{G}(v) \geq \aA n^2$ for all $v \in V(G)$.
Then $G$ contains a copy of the $1$-expansion of any given $\mu$-separable graph $H$ with $\DD(H) \leq \DD$ and $|V(H)|+|E(H)|\leq n$.
\end{theo}

Since a large $2$-regular graph (a union of vertex-disjoint cycles) is $o(1)$-separable, the theorem implies that any super uniformly dense $3$-graph $G$ contains as a subhypergraph any $3$-graph consisting of vertex-disjoint loose cycles.

Next, we reformulate Theorem~\ref{th:bip} as a blow-up lemma for $3$-graphs, which may be of independent interest.

\begin{theo}[Simplified weak hypergraph blow-up lemma]\label{lm:stran}
For all $\DD,d,\dD>0$ there are $\eps,\mu>0$ and $n_0 \in \mb{N}$ such that the following holds for all integers $n \geq n_0$.
Let $G$ be a $3$-partite $3$-graph with parts
$V_1,V_2,V_3$ where every $|V_i|=:n_i$
and $n \leq n_1 \leq n_2 \leq n/\dD$ and $n_3 \geq \dD n$ such that $G$ is a \emph{weakly $(\eps,d)$-half-superregular triple}, that is,
\begin{itemize}
\item[(i)] for all $i \in [3]$ and $V_i' \subseteq V_i$ with $|V_i'| \geq \eps|V_i|$, we have
$
e_G(V_1',V_2',V_3') \geq d|V_1'||V_2'||V_3'|
$,
\item[(ii)] $d_G(v) \geq dn^2$ for all $v \in V(G)$.
\end{itemize}
Let $H$ be a $\mu$-separable bipartite $2$-graph with parts of size $n_1,n_2$, with $n_3$ edges and maximum degree $\DD$.
Then $G$ contains a copy of the $1$-expansion of $H$, where the images of the new vertices lie in $V_3$.
\end{theo}

Weak superregularity and uniform density of $3$-graphs and graph collections are closely connected.
Observe that if $G$ is $(d+\sqrt{\eta},\eta)$-dense, where $\eta \gg \eps,\dD$, and $V_1,V_2,V_3$ is any partition with sizes as in the theorem, then the $3$-partite subhypergraph $G'$ of $G$ induced on these parts satisfies~(i). If $d_G(v) \geq \aA n^2$ for all $v \in V$, and in addition the vertex partition is uniformly random, then with 
probability tending to $1$ as $n \to \infty$, $G'$ is weakly $(\eps,\min\{d,\aA/2\})$-half-superregular, say. 
Furthermore, if one takes any partition $V(G)=V \cup \Cols$ into parts which are not too small, then the graph collection $\bm{G}$ consisting of graphs $G_c$ for $c \in \Cols$ with $V(G_c)=V$ and $E(G_c)=\{xy: x,y \in V \text{ and }xyc \in E(G)\}$ is $(d,\eta)$-dense. Again, if $d_G(v) \geq \aA n^2$, a random such partition gives rise to a super uniformly dense graph collection.

 As an example, any $3$-graph $G$ satisfying the hypotheses of the theorem with $n_1=n_2$ and $n_3=n_1+n_2$ contains a copy of a loose Hamilton cycle.
The full statement of a $3$-graph version of our transversal blow-up lemma (Theorem~\ref{th:blowupintro}) is stated in Theorem~\ref{th:blowup3}.

A much more general hypergraph blow-up lemma was established by Keevash \cite{Keevash} which applies to $k$-graphs satisfying a strong regularity property, which is designed to be used with the strong regularity lemma; this blow-up lemma applies to embed any bounded degree $k$-graph $H$.
Nevertheless, it would be interesting to determine which graphs can be embedded into a weakly superregular triple. We discuss this further in Section~\ref{sec:conclude}.

There are many other notions of quasirandomness of hypergraphs, often related to the number of edges in/between subsets, that have been studied in the literature. We refer the reader to~\cite{aigner2017quasirandomness,lenz2015poset,towsner2017sigma} for a detailed comparison of such notions; it turns out that equivalent notions of quasirandomness in graphs have inequivalent hypergraph analogues. The embedding of spanning structures has been studied in these various settings, for example tight Hamilton cycles in $3$-graphs with a stronger quasirandomness condition than the one in this paper, in~\cite{aigner2021tight,araujo2020localised,gan2022hamiltonicity}, 
and graphs of sublinear bandwidth in `locally dense' $2$-graphs, which is weaker than uniformly dense, in~\cite{ST}, and in `inseparable' $2$-graphs in~\cite{ebsen2020embedding}.
While in this paper we ask which $H$ are subhypergraphs of any quasirandom hypergraph $G$ of positive density (where for us, `quasirandom' means `uniformly dense'),
the results above mainly concern the generalisation of this question: given $H$, what minimum degree in a quasirandom hypergraph $G$ is sufficient to imply the existence of $H$? 
There are various notions of degree in hypergraphs and the results mentioned in this paragraph consider several of them.
Quasirandomness conditions that apply to sparse (hyper)graphs have also been well studied. Recently, H\`an, Han and Morris~\cite{han2022factors} extended the results of~\cite{lenz2016perfect,lenz2016hamilton} to the sparse regime.

\subsection{The statement of our main result}

The full statement of our transversal blow-up lemma is as follows.
Its proof will follow from Theorem~\ref{th:blowup} which is a very similar statement written in the notation defined in Section~\ref{sec:reg}.

\begin{theo}[Transversal blow-up lemma]\label{th:blowupintro}
For all $\nu,d,\dD,\DD,r>0$ where $r \geq 2$ is an integer, there exist $\eps,\mu,\aA>0$ and $m_0 \in \mb{N}$ such that the following holds for all integers $m \geq m_0$.

\noindent
Suppose that $\bm{G}=(G_c:c \in \Cols)$ is a graph collection with the following properties.
\begin{itemize}
    \item There is a graph $R$ with vertex set $[r]$ and a partition $\Cols = \bigcup_{e \in E(R)}\Cols_e$ where $|\Cols_e| \geq \dD m$ for all $e \in E(R)$;
\item for all $ij \in E(R)$ and $c \in \Cols_{ij}$, $G_c$ is bipartite with parts $V_i,V_j$,
and $V$ is a vertex set of size $n$ with partition $V=V_1 \cup \ldots \cup V_r$, where $m \leq |V_i| \leq m/\dD$ for each $i \in [r]$;
\item for all $ij \in E(R)$,
\begin{itemize}
\item for all $V_h' \subseteq V_h$ for $h=i,j$ and $\Cols'_{ij} \subseteq \Cols_{ij}$
with $|V_h'|\geq \eps|V_h|$ for $h=i,j$ and $|\Cols'_{ij}| \geq \eps |\Cols_{ij}|$, we have that
$$
\sum_{c \in \Cols'_{ij}}e_{G_c}(V_i',V_j') \geq d|\Cols'_{ij}||V_i'||V_j'|;
$$
\item for $h=i,j$ and every $v \in V_h$ we have $\sum_{c \in \Cols_{ij}}d_{G_c}(v) \geq d|\Cols_{ij}|m$
and for every $c \in \Cols_{ij}$, we have $e(G_c) \geq dm^2$.
\end{itemize}
\end{itemize}
Suppose that $H$ is a graph with the following properties.
\begin{itemize}
\item $\DD(H) \leq \DD$; 
    \item $H$ is $\mu$-separable;
\item $H$ has vertex partition $A_1 \cup \ldots \cup A_r$ such that $|A_i|=|V_i|$ for all $i \in [r]$ and for every $xy \in E(H)$ there is $ij \in E(R)$ such that $x \in A_i$ and $y \in A_j$;
\item $e(H[A_i,A_j])=|\Cols_{ij}|$ for all $ij \in E(R)$;
\item for each $i \in [r]$, there is a set $U_i \subseteq A_i$ with $|U_i| \leq \aA m$ and for each $x \in U_i$, a set $T_x \subseteq V_i$ with $|T_x| \geq \nu m$.
\end{itemize}
Then there is a transversal embedding of $H$ inside $\bm{G}$ such that for every $i \in [r]$, every $x \in U_i$ is embedded inside $T_x$.
\end{theo}

\subsection{Notation and organisation}

\textbf{Notation.} For reals $x,a,b$, we write $x=a\pm b$ if we have $a-b\leq x\leq a+b$. For any two constants $\alpha,\beta \in (0,1)$, we write $\alpha \ll \beta$ if there exists a function $\alpha_0=\alpha_0(\beta)$ such that the subsequent arguments hold for all $0<\alpha\leq \alpha_0$. When we write multiple constants in a hierarchy, we mean that they are chosen from right to left. For any positive integer $a$, let $[a]:=\{1,2,\ldots,a\}$.
Given a set $X$ and a positive integer $b$, write $\binom{X}{b}$ to denote the set of $b$-element subsets of $X$.

Let $k \geq 2$ be an integer and let $G=(V,E)$ be any $k$-graph, so each edge consists of $k$ vertices. We use $V(G):=V$ to denote its vertex set and $E(G):=E$ to denote its edge set. Let $v(G):=|V(G)|$ and $e(G):=|E(G)|$ be their sizes. For any vertex $v\in V(G)$, let $N_G(v)$ be the \emph{neighbourhood} of $v$, that is, the set of $(k-1)$-tuples of $v$ that are incident to $v$ and let $d_G(v):=|N_G(v)|$ be the \emph{degree} of $v$. Sometimes we use $d(v)$ for short if $G$ is obvious from the text. 
We write $\dD(G) := \min_{v \in V(G)}d_G(v)$ and $\DD(G) := \max_{v \in V(G)}d_G(v)$ for the minimum and maximum degree of $G$.
For each vertex $v\in V(G)$ and subset $S \subseteq \binom{V(G)}{k-1}$, let $N_G(v,S):=N_G(v)\cap S$ and $d_G(v,S):=|N_G(v,S)|$. 
For any vertex set $U\subseteq V(G)$, let $G[U]$ be the induced hypergraph of $G$ on $U$, i.e.~the hypergraph with vertex set $U$ and all the edges of $G$ with vertices only in $U$. Let $G-U:=G[V(G)\sm U]$. 
For a subhypergraph $H$ of $G$, let $G\sm H$ be the graph with vertex set $V(G)$ and edge set $E(G)\sm E(H)$.
For any (not necessarily disjoint) vertex sets $U_1,\ldots,U_k \subseteq V(G)$, we write $G[U_1,\ldots,U_k]$ for the $k$-graph with vertex set $U_1 \cup \ldots \cup U_k$ and edge set $E_G(U_1,\ldots,U_k)$, which is the set of edges of $G$ with one vertex in $U_i$ for each $i \in [k]$.
Let $e_G(U_1,\ldots,U_k)$ be the number of $k$-tuples $(x_1,\ldots,x_k) \in U_1\times \ldots \times U_k$ for which $\{x_1,\ldots,x_k\} \in E_G(U_1,\ldots,U_k)$
(note that when these sets are not disjoint, edges may be counted more than once).

We say that a statement about a graph on $n$ vertices holds \emph{with high probability}
if it holds with probability tending to $1$ as $n\to\infty$.
We use script letters (e.g.~$\ms{C}$) to denote sets of colours and bold letters (e.g.~$\bm{G}$) to denote graph collections.
It will be convenient to represent a graph collection in three equivalent ways:
\begin{itemize}
\item a collection $\bm{G}=(G_c: c \in \Cols)$ of graphs on the same vertex set $V$. We call $\Cols$ the \emph{colour set} of $\bm{G}$;
\item an edge-coloured graph $\mc{G}$ on vertex set $V$ with edge set $\bigcup_{c \in \Cols}G_c$, where $xy$ has a multiset of colours consisting of those $c \in \Cols$ for which $xy \in G_c$.
We say that $\mc{G}$ is the \emph{edge-coloured graph of $\bm{G}$} (but rarely use this representation);
\item a $3$-graph $G^{(3)}$ on vertex set $V \cup \Cols$ with edges $xyc$ whenever $xy \in G_c$. We say that $G^{(3)}$ is the \emph{$3$-graph of $\bm{G}$}.
\end{itemize}
Given two $2$-graphs $H$ and $G$, a \emph{graph homomorphism (from $H$ to $G$)} is a map $\phi: V(H) \to V(G)$ such that $\phi(x)\phi(y) \in E(G)$ whenever $xy \in E(H)$. An injective graph homomorphism is an \emph{embedding (of $H$ into $G$)}.
Given a graph $H$ and a graph collection $\bm{G}=(G_c: c \in \Cols)$ on vertex set $V$, a \emph{transversal embedding (of $H$ into $\bm{G}$)} is a pair $\tau:V(H) \to V$ and $\sS:E(H) \to \Cols$ of injective maps such that $\tau(x)\tau(y) \in E(G_{\sS(xy)})$ whenever $xy \in E(H)$.


\medskip
\noindent
\textbf{Organisation.}
In Section~\ref{sec:reg}, we define regularity for graph collections, state a regularity lemma for graph collections, define the `template' graph collections to which our theory applies, and prove some of the basic tools one uses when applying the regularity method, such as a slicing lemma and a degree inheritance lemma.
Section~\ref{sec:embed} contains some embedding lemmas which are the main ingredients of the proof of our transversal blow-up lemma, Theorem~\ref{th:blowup}, but which are also useful tools when applying the method.
In Section~\ref{sec:blowup}, we prove Theorem~\ref{th:blowup}
(which readily implies Theorem~\ref{th:blowupintro})
and in Section~\ref{sec:dense}, we prove Theorems~\ref{th:quasi} and~\ref{th:quasi3} on embeddings in uniformly dense graph collections and $3$-graphs.
Section~\ref{sec:conclude} contains some concluding remarks.

\section{The regularity-blow-up method for transversals}\label{sec:reg}

\subsection{Regularity}

Our notion of regularity for a graph collection $\bm{G}$ is essentially weak regularity for the $3$-graph $G^{(3)}$ of $\bm{G}$. We now define several related notions of weak regularity for $k$-graphs, which we will use for $k=2,3$. We usually omit \emph{weak} and \emph{weakly} since we will not use any stronger type of regularity.


Let $G$ be a $k$-partite $k$-graph with classes $V_1,\ldots,V_k$,
which we also denote as $(V_1,\ldots,V_k)_G$.
We define the \emph{density} of $G$ to be
$$
d_G(V_1,\ldots,V_k) := \frac{e(G)}{|V_1|\ldots|V_k|}.
$$
Given $\eps>0$, we say that $(V_1,\ldots,V_k)_G$ is 
\begin{itemize}
    \item \emph{(weakly) $\eps$-regular} if for every subhypergraph $(V_1',\ldots,V_k')_G$ with $V_i' \subseteq V_i$ and $|V_i'| \geq \eps|V_i|$ for all $i \in [k]$ we have
$$
|d_G(V_1',\ldots,V_k') - d_G(V_1,\ldots,V_k)|< \eps;
$$
\item \emph{(weakly) $(\eps,d)$-regular} if additionally $d_G(V_1,\ldots,V_k)\geq d$;
\item \emph{(weakly) $(\eps,d$)-superregular} if it is (weakly) $(\eps,d)$-regular and additionally $d_G(x)\geq d(|V_1|\ldots|V_k|)/|V_i|$ for all $i \in [k]$ and $x \in V_i$;
\item \emph{(weakly) $(\eps,d)$-half-superregular} if for every subhypergraph $(V_1',\ldots,V_k')_G$ with $V_i' \subseteq V_i$ and $|V_i'| \geq \eps|V_i|$ for all $i \in [k]$ we have $d_G(V_1',\ldots,V_k')\geq d$ and $d_G(x)\geq d(|V_1|\ldots|V_k|)/|V_i|$ for all $i\in[k]$ and $x \in V_i$.
\end{itemize}

Our main results will use `half-superregular' given its close connection to uniform density. The $2$-graph blow-up lemma~\cite{KSS} uses a similar notion, but nowadays `superregular' is generally used instead. This makes little material difference:
by definition, an $(\eps,d$)-superregular hypergraph is $(\eps,d-\eps$)-half-superregular,
and as shown by R\"odl and Ruci\'nski in the course of their alternative proof of the blow-up lemma~\cite{RR}, a half-superregular $2$-graph contains a spanning superregular subgraph, with weaker parameters. This proof generalises easily to $k$-graphs, 
and therefore we postpone it to the appendix.

\begin{lemma}\label{lm:hsrtosr}
Let $0<1/n\ll \eps \ll \eps' \ll d,\dD,1/k \leq 1$ where $k \ge 2$ is an integer. Suppose that $G$ is a $k$-partite $k$-graph with parts $V_1,\ldots,V_k$ where $\dD n\leq |V_i|\leq n/\dD$ for all $i \in [k]$. If $G$ is $(\eps,d)$-half-superregular, then
$G$ contains a spanning subhypergraph $G'$ that is $(\eps',d^2/2)$-superregular.
\end{lemma}

Regularity for a bipartite graph collection $\bm{G}$ is defined in terms of the $3$-graph $G^{(3)}$ of $\bm{G}$.

\begin{defn}[Regularity, semi- and half-superregularity, superregularity]
Suppose that $\bm{G}$ is a graph collection with colour set $\Cols$, where each $G_c$ is bipartite with parts $V_1,V_2$. 
Let $G^{(3)}$ be the $3$-graph of $\bm{G}$.
We say that
\begin{itemize}
\item $\bm{G}$ is \emph{$(\eps,d)$-regular} if $G^{(3)}$ is $(\eps,d)$-regular. That is, for all $V_i' \subseteq V_i$ with $|V_i'| \geq \eps|V_i|$ for $i=1,2$ and $\Cols' \subseteq \Cols$ with $|\Cols'| \geq \eps|\Cols|$, we have
$$
\left|\frac{\sum_{c \in \Cols'}e_{G_c}(V_1',V_2')}{|\Cols'||V_1'||V_2'|}-\frac{\sum_{c \in \Cols}e_{G_c}(V_1,V_2)}{|\Cols||V_1||V_2|}\right| < \eps.
$$
\item $\bm{G}$ is \emph{$(\eps,d)$-semi-superregular} if it is $(\eps,d)$-regular
and $d_{G^{(3)}}(v) = \sum_{c \in \Cols}d_{G_c}(v) \geq d|V_{3-i}||\Cols|$ for all $i \in [2]$ and $v \in V_i$.
\item $\bm{G}$ is \emph{$(\eps,d)$-superregular} if $G^{(3)}$ is $(\eps,d)$-superregular. That is, it is $(\eps,d)$-semi-superregular and $d_{G^{(3)}}(c) = e(G_c) \geq d|V_1||V_2|$ for all $c \in \Cols$.
\item $\bm{G}$ is \emph{$(\eps,d)$-half-superregular} if $G^{(3)}$ is $(\eps,d)$-half-superregular. That is, for all $V_i' \subseteq V_i$ with $|V_i'| \geq \eps|V_i|$ for $i=1,2$ and $\Cols' \subseteq \Cols$ with $|\Cols'| \geq \eps|\Cols|$, we have $\sum_{c \in \Cols'}e_{G_c}(V_1',V_2') \geq d|\Cols'||V_1'||V_2'|$ and $\sum_{c \in \Cols}d_{G_c}(v) \geq d|V_{3-i}||\Cols|$ for all $i=1,2$ and $v \in V_i$, and $e(G_c) \geq d|V_1||V_2|$ for all $c \in \Cols$.
\end{itemize}
\end{defn}

Note that, if every $G_c$ with $c \in \Cols$ is the same, then $\bm{G}$ is $(\eps,d)$-regular if and only if $G_c$ is $(\eps,d)$-regular; and $\bm{G}$ is $(\eps,d)$-superregular if and only if $G_c$ is $(\eps,d)$-superregular.
The superregularity of $\bm{G}$ does not imply a minimum degree condition for any graph $G_c$ in the collection, and indeed they could all have isolated vertices. 

The following simple lemma shows that, in an $(\eps,d)$-regular graph collection, most vertices -- \emph{typical} vertices -- have large total degree (the sum of degrees over all colours) and typical colours have many edges.

\begin{lemma}[Typical vertices and colours]\label{lm:standard}
Let $0<\eps \ll d \leq 1$, and let $\bm{G}$ be an $(\eps,d)$-regular graph collection with colour set $\Cols$, where each $G_c$ is bipartite with parts $V_1,V_2$. Then the following hold:
\begin{itemize}
\item[(i)] for every $i \in [2]$ and all but at most $\eps|V_i|$ vertices $v \in V_i$ we have $\sum_{c \in \Cols}d_{G_c}(v) \geq (d-\eps)|V_{3-i}||\Cols|$;
\item[(ii)] for all but at most $\eps|\Cols|$ colours $c \in \Cols$ we have $e(G_c) \geq (d-\eps)|V_1||V_2|$.
\end{itemize}
\end{lemma}

\begin{proof}
For~(i), let $V_1'$ be the set of vertices in $V_1$ without this property and suppose for a contradiction that $|V_1'|>\eps|V_1|$. Then $(\eps,d)$-regularity implies that
$d_{G^{(3)}}(V_1',V_2,\Cols)>d-\eps$, but the definition of $V_1'$ implies that $e_{G^{(3)}}(V_1',V_2,\Cols)\leq|V_1'|(d-\eps)|V_2||\Cols|$ and hence $d_{G^{(3)}}(V_1',V_2,\Cols) \leq d-\eps$, a contradiction.
By symmetry, the rest of~(i) and~(ii) are identical.
\end{proof}

The following lemma is a standard tool, written in our graph collection notation.

\begin{lemma}[Slicing lemma]\label{lm:slice}
Let $0<1/n \ll \eps \ll \eps' \ll \aA \ll d \leq 1$, and let $\bm{G}$ be a graph collection with colour set $\Cols$, where each $G_c$ is bipartite with parts $V_1,V_2$ each of size at least $n$, and let $V_i' \subseteq V_i$ for $i \in [2]$ and $\Cols' \subseteq \Cols$. Let $\bm{G}'=(G_c[V_1',V_2']:c \in \Cols')$.
\begin{itemize}
\item[(i)] Suppose that $\bm{G}$ is $(\eps,d)$-regular.
Suppose $|V_i'| \geq \aA|V_i|$ for $i \in [2]$ and $|\Cols'| \geq \aA|\Cols|$. Then $\bm{G}'$ is $(\eps/\aA,d/2)$-regular.
\item[(ii)] Suppose that $\bm{G}$ is $(\eps,d)$-superregular. Suppose $|V_i'| \geq (1-\aA)|V_i|$ for $i \in [2]$ and $|\Cols'|>(1-\aA)|\Cols|$. Then $\bm{G}'$ is $(2\eps,d/2)$-superregular.
\item[(iii)] Suppose that $\bm{G}$ is $(\eps,d)$-superregular. Given $n_1\geq \aA |V_1|,n_2 \geq \aA |V_2|$ and $h \geq \aA |\Cols|$, if $V_i'\subseteq V_i$ is a uniform random subset of size $n_i$ for each $i=1,2$
and $\Cols'\subseteq \Cols$ is a uniform random subset of size $h$, then with high probability,
$\bm{G}'$ is $(\eps/\aA,d^2/16)$-superregular.
\end{itemize}
\end{lemma}

\begin{proof}
For (i), firstly we have $d_{G^{(3)}}(V_1^{\prime},V_2^{\prime},\Cols')\geq d-\eps\geq d/2$. Secondly, given any subset $V_i^{\prime\prime}\subseteq V_i'$ with $|V_i^{\prime\prime}|\geq \eps|V_i^{\prime}|/\aA\geq \eps |V_i|$ where $i=1,2$ and any subset $\Cols^{\prime\prime}\subseteq \Cols^{\prime}$ with $|\Cols^{\prime\prime}|\geq \eps |\Cols^{\prime}|/\aA\geq \eps |\Cols|$, by regularity we have
\begin{align*}
&|d_{G^{(3)}}(V_1^{\prime\prime},V_2^{\prime\prime},\Cols^{\prime\prime})-d_{G^{(3)}}(V_1^{\prime},V_2^{\prime},\Cols^{\prime})|\\
&\leq
|d_{G^{(3)}}(V_1^{\prime\prime},V_2^{\prime\prime},\Cols^{\prime\prime})-d_{G^{(3)}}(V_1,V_2,\Cols)|
+|d_{G^{(3)}}(V_1,V_2,\Cols)-d_{G^{(3)}}(V_1^{\prime},V_2^{\prime},\Cols^{\prime})|\leq 2\varepsilon \leq \eps/\aA.
\end{align*}
For (ii), note that $\bm{G}'$ is $(\eps/(1-\aA),d/2)$-regular and thus $(2\eps,d/2)$-regular by (i). Let $G_c':=G_c[V_1',V_2']$ for each $c \in \Cols$. For each $i=1,2$ and $v\in V_{i}'$, we also have $\sum_{c\in \Cols'}d_{G_c'}(v)\geq d|V_{3-i}||\Cols|-2\aA|V_{3-i}||\Cols|\geq d|V_{3-i}'||\Cols'|/2$. Combining this with $e(G_c')\geq d|V_1||V_2|-2\aA|V_1||V_2|\geq d|V_1||V_2|/2$ for each $c\in \Cols'$, we get that $\bm{G}'$ is $(2\eps,d/2)$-superregular.

For (iii), similarly, we always have that $\bm{G}'$ is $(\eps/\aA,d/2)$-regular by~(i). Let $G_c':=G_c[V_1',V_2']$ for all $c \in \Cols$. We only need to show that with high probability, for each $i=1,2$ and $v\in V_i'$, we have $\sum_{c\in \Cols'}d_{G_c'}(v)\geq d^2|V_{3-i}'||\Cols'|/16$ and for each $c\in \Cols'$, we have $e(G_c')\geq d^2|V_1||V_2|/16$. Suppose that $\Cols'$ has been chosen. Since $\bm{G}$ is $(\eps,d)$-superregular, for each $c\in \Cols'$, we have $e(G_c)\geq d|V_1||V_2|$. Let $V_1^c:=\{v\in V_1 : d_{G_c}(v)\geq d|V_2|/2\}$. Then we get $d|V_1||V_2|\leq |V_1^c||V_2|+(|V_1|-|V_1^c|)d|V_2|/2$ and thus $|V_1^c|\geq d|V_1|/(2-d)\geq d|V_1|/2$. 
A Chernoff bound implies that, with probability $1-e^{-\OO(n)}$, we have $|V_1'\cap V_1^c|\geq d n_1/4$ and $|V_2'\cap N_{G_c}(v)|\geq d n_2/4$ for each $v\in V_1^c$.
Therefore, by a union bound, with high probability, we have $e(G_c')\geq d^2n_1n_2/16$ for each $c\in \Cols'$. Similarly, with high probability, for each $i=1,2$ and $v\in V_i'$,  we have $\sum_{c\in \Cols'}d_{G_c'}(v)\geq d^2n_{3-i}h/16$. It follows that $\bm{G}'$ is $(\eps/\aA,d^2/16)$-superregular with high probability.
\end{proof}


\subsection{The regularity lemma for graph collections}

We use the following version of the regularity lemma for graph collections, which is obtained by applying the degree version of the weak regularity lemma (Lemma~\ref{lm:weakreg}) to the $3$-graph $G^{(3)}$ of $\bm{G}$ and cleaning up the clusters so that vertex clusters and colour clusters are separate.
We postpone the derivation to the appendix.

\begin{lemma}[Regularity lemma for graph collections]\label{lm:weakregcol}
For all integers $L_0 \geq 1$ and every $\eps,\dD>0$, there is an $n_0=n_0(\eps,\dD,L_0)$ such that
for every $d \in [0,1)$ and
every graph collection $\bm{G}=(G_c: c \in \Cols)$ on vertex set $V$ of size $n \geq n_0$ with $\dD n \leq |\Cols| \leq n/\dD$, there exists a partition of $V$ into $V_0,V_1,\ldots,V_L$, of $\Cols$ into $\Cols_0,\Cols_1,\ldots,\Cols_M$ and a spanning subgraph $G'_c$ of $G_c$ for each $c \in \Cols$ such
that the following properties hold:
\begin{enumerate}[(i)]
\item $L_0 \leq L,M \leq n_0$ and $|V_0|+|\Cols_0| \leq \eps n$;
\item $|V_1|=\ldots=|V_L|=|\Cols_1|=\ldots = |\Cols_M| =: m$;
\item $\sum_{c \in \Cols}d_{G'_c}(v) > \sum_{c \in \Cols}d_{G_c}(v)-(3d/\dD^2+\eps)n^2$ for all $v \in V$ and
$e(G'_c) > e(G_c)-(3d/\dD^2+\eps)n^2$ for all $c \in \Cols$;
\item if, for $c \in \Cols$, the graph $G'_c$ has an edge with both vertices in a single cluster $V_i$ for some $i \in[L]$, then $c \in \Cols_0$;
\item for all triples $(\{h,i\},j) \in \binom{[L]}{2} \times [M]$, we have that either $G'_c[V_h,V_i]=\emptyset$ for all $c \in \Cols_j$, or
$\bm{G}'_{hi,j} := (G'_c[V_h,V_i]: c \in \Cols_j)$ is $(\eps,d)$-regular.
\end{enumerate}
\end{lemma}

The sets $V_i$ are called \emph{vertex clusters} and the sets $\Cols_j$ are called \emph{colour clusters}, while $V_0$ and $\Cols_0$ are the exceptional vertex and colours sets respectively.

\begin{defn}[Reduced graph collection
]
Given a graph collection $\bm{G}=(G_c: c \in \Cols)$ on $V$ and parameters $\eps>0, d \in [0,1)$ and $L_0 \geq 1$, the \emph{reduced graph collection} $\bm{R}=\bm{R}(\eps,d,L_0)$, \emph{reduced $3$-graph} $R^{(3)}=R(\eps,d,L_0)$ and the \emph{reduced edge-coloured graph} $\mc{R}=\mc{R}(\eps,d,L_0)$ of $\bm{G}$ are defined as follows.
Apply Lemma~\ref{lm:weakregcol} to $\bm{G}$ with parameters $\eps,d,L_0$
to obtain $\bm{G}'$ and a partition $V_0,\ldots,V_L$ of $V$ and $\Cols_0,\ldots,\Cols_M$ of $\Cols$ where
$V_0$, $\Cols_0$ are the exceptional sets and $V_1,\ldots,V_L$ are the vertex clusters
and $\Cols_1,\ldots,\Cols_M$ are the colour clusters.
Then
$\bm{R}=(R_1,\ldots,R_M)$ is a graph collection of $M$ graphs each on the same vertex set $[L]$,
where, for $(\{h,i\},j) \in \binom{[L]}{2} \times [M]$, we have $hi \in R_j$ whenever $\bm{G}'_{hi,j}$ is $(\eps,d)$-regular.

Also, $R^{(3)}$ is the $3$-graph of $\bm{R}$ and $\mc{R}$ is the reduced edge-coloured graph.
\end{defn}

The next lemma (related to Lemma 5.5 in~\cite{KOT})
states that clusters inherit a minimum degree bound in the reduced graph from $\bm{G}$.

\begin{lemma}[Degree inheritance]\label{lm:inherit}
Suppose $p>0$, $L_0 \geq 1$ and $0 < 1/n \ll \eps \leq d \ll \dD,\gG,p \leq 1$.
Let $\bm{G}=(G_c: c \in \Cols)$ be a graph collection on a vertex set $V$ of size $n$ with $\dD(G_c) \geq (p+\gG)n$ for all $c \in \Cols$ and $\dD n \leq |\Cols| \leq n/\dD$.
Let $\bm{R}=\bm{R}(\eps,d,L_0)$ be the reduced graph collection of $\bm{G}$ on $L$ vertices with $M$ graphs.
Then
\begin{enumerate}[(i)]
\item for every $i \in [L]$ there are at least $(1-d^{1/4})M$ colours $j \in [M]$ for which $d_{R_j}(i) \geq (p+\gG/2)L$;
\item for every $j \in [M]$ there are at least $(1-d^{1/4})L$ vertices $i \in [L]$ for which $d_{R_j}(i) \geq (p+\gG/2)L$.
\end{enumerate}
\end{lemma}

\begin{proof}
To prove~(i), note that for all $v \in V\sm V_0$ we have
$$
\sum_{c \in \Cols}d_{G'_c-V_0}(v) \geq \sum_{c \in \Cols}d_{G_c}(v)-(3d/\dD^2+\eps)n^2-\eps|\Cols|s \geq \sum_{c \in \Cols}d_{G_c}(v)-4dn^2/\dD^2.
$$
Let $\ms{D}_v$ be the collection of colours $c$ in $\Cols\sm\Cols_0$ for which $d_{G'_c-V_0}(v) \geq d_{G_c}(v)-\sqrt{d}n$. Then
$$
\sum_{c \in \Cols}d_{G'_c-V_0}(v) 
\leq \sum_{c \in \Cols}d_{G_c}(v)-|\Cols\sm \ms{D}_v|\sqrt{d}n
$$
and therefore $|\Cols\sm \ms{D}_v| \leq 4dn^2/(\dD^2\sqrt{d}n) \leq d^{1/3}n/2$ by $d\ll \dD$,
so $|\ms{D}_v| \geq |\Cols| -d^{1/3}n/2-\eps n \geq |\Cols|-d^{1/3}n$.
We have $mM \leq |\Cols| \leq mM+\eps n$ and $mL \leq n \leq mL+\eps n$.
Thus the number of clusters $\Cols_j$ containing at least one colour of $\ms{D}_v$ is at least
$$
|\ms{D}_v|/m \geq M-d^{1/3}n/m \geq M-d^{1/3}L/(1-\eps) \geq M-2d^{1/3}M/\dD \geq (1-d^{1/4})M.
$$
Now let $i \in [L]$ and $v \in V_i$.
For each cluster $\Cols_j$ as above, choose an arbitrary colour $c_j \in \Cols_j \cap \ms{D}_v$.
Then the number of clusters $V_h$ containing some $u \in N_{G'_{c_j}}(v)$ is at least
$$
\frac{d_{G_{c_j}}(v)-\sqrt{d}n}{m} \geq \frac{(p+\gG-\sqrt{d})n}{m} \geq (p+\gG/2)L.
$$
But then~Lemma~\ref{lm:weakregcol}(v) implies that $i$ is adjacent to each such $V_h$ in $R_j$.
So for every $i \in [L]$, $d_{R_j}(i) \geq (p+\gG/2)L$ for at least $(1-d^{1/4})M$ colours $j$.
The proof of~(ii) is similar and we omit it.
\end{proof}

\subsection{Templates}

We define the notion of a `template', which is essentially a reduced graph in the transversal setting. We will use these as templates for embedding, in the same way that reduced graphs are used for embedding into a single graph.

\begin{defn}[Templates]
Let $0 < 1/m \leq \eps \leq d,\dD \leq 1$ be parameters and let $r \in \mb{N}$.
Suppose that
\begin{itemize}
\item $R$ is an $r$-vertex graph, with vertex set $[r]$ unless otherwise specified,
\item $\mc{V}=\{V_1,\ldots,V_r\}$ is a set of $r$ disjoint vertex sets with $m \leq |V_j| \leq m/\dD$ for all $j \in [r]$, whose union we denote by $V$,
\item $\Cols=\bigcup_{e \in E(R)}\Cols_e$ is a colour set where $|\Cols_e| \geq \dD m$ for all $e \in E(R)$,
\item $\bm{G}$ is a graph collection with colour set $\Cols$ where for each $c \in \Cols$, the graph $G_c$ is the union of bipartite graphs $G_c^{e}$ where $G_c^{e}$ has parts $V_i,V_j$, over all $e=ij$ for which $c \in \Cols_e$.
For each $e \in E(R)$, let $\bm{G}^e=(G_c^e: c \in \Cols_e)$.
\end{itemize}
We say that $\mc{F}=(\mc{V},\Cols,\bm{G})$ is an \emph{$R$-template with parameters $(m,\eps,d,\dD)$}
if for every $e \in E(R)$, $\bm{G}^e$ is $(\eps,d)$-regular.
If we replace regular with semi-superregular, it is a \emph{semi-super $R$-template};
if we replace regular with superregular, it is a \emph{super $R$-template};
if we replace regular with half-superregular, it is a \emph{half-super $R$-template}.
If $\Cols=\bigcup_{e \in E(R)}\Cols_e$ is a partition, we say that the template is \emph{rainbow}.

A \emph{transversal embedding of a graph $H$ inside $\mc{F}$} is a copy of $H$ with vertices in $V$ such that for every edge $e$ there is a distinct $c \in \Cols$ such that $e \in G_c$.
That is, there exist injections $\tau: V(H) \to V$ and $\sS: E(H) \to \Cols$ where $\tau(x)\tau(y) \in G_{\sS(xy)}$ for all $xy \in E(H)$.
\end{defn}

Note that the partition $\Cols=\bigcup_{e \in E(R)}\Cols_e$ is suppressed in the notation.
Given a template $\mc{F}$, we explicitly use the notation in the definition unless otherwise specified.
Observe that for an $R$-template $(\mc{V},\Cols,\bm{G})$ with parameters $(m,\eps,d,\dD)$ and $v(R)=r$,
$$
rm \leq |V| \leq rm/\dD.
$$
Given parameters $\bB,\eps',d',\dD'>0$ with $\bB \leq 1$, $\eps' \geq \eps$, $d' \leq d$ and $\dD' \leq \bB\dD$, any $(m,\eps,d,\dD)$ template is also a $(\bB m,\eps',d',\dD')$ template.
We will often take \emph{subtemplates} of templates, meaning that the new vertex clusters, colour clusters and graphs $G_c$ are subsets/subgraphs of the originals.
Some convenient notation for this is as follows: if $\mc{F} = (\mc{V},\Cols,\bm{G})$ is a template,
$\Cols' \subseteq \Cols$ and $\mc{V}' = \{V_1',\ldots,V_r'\}$ where $V_j' \subseteq V_j$ for all $j \in [r]$, we say that $\mc{F}' = (\mc{V}',\Cols',\bm{G}')$ is the \emph{subtemplate of $\mc{F}$ induced by $\mc{V}',\mc{\Cols'}$} when each $\Cols_e'=\Cols_e \cap \Cols'$, and $(G')^e_c:=G_c[V_i',V_j']$ is defined for each $c \in \Cols'_e$.
We also say that $\mc{F}'$ is \emph{obtained by deleting $\Cols \sm \Cols'$ and $V \sm V'$}.

The following straightforward lemma shows that removing a small fraction of colours and vertices from a template produces a subtemplate with slightly weaker parameters, which remains super if the original template was super.

\begin{lemma}[Template slicing]\label{lm:tempslice}
Let $0 <1/m \ll \eps \ll \eps' \ll \aA \ll d,\dD,1/r,1/k \leq 1$ where $r \geq 2$ is an integer.
Let $R$ be an $r$-vertex graph and
let $\mc{F}=(\mc{V},\Cols,\bm{G})$ be an $R$-template with parameters $(m,\eps,d,\dD)$.
Let $\Cols' \subseteq \Cols$, $V_i' \subseteq V_i$ for all $i \in [r]$
and let $\mc{F}' = (\mc{V}',\Cols',\bm{G}')$ be the subtemplate of $\mc{F}$ induced by $\mc{V}',\Cols'$.
\begin{enumerate}
\item[(i)] If every $|\Cols'_e| \geq \aA |\Cols_e|/k$ and $\aA|V_i| \leq |V_i'| \leq k\aA |V_i|$,
then $\mc{F}'$ is a template with parameters $(\aA m,\eps/\aA,d/2,\dD/k)$.
\item[(ii)] If every $|\Cols_e \sm \Cols'_e| \leq \aA m$ and $|V_i \sm V_i'| \leq \aA m$,
then $\mc{F}'$ is a template with parameters $(m/2,2\eps,d/2,\dD/2)$.
Moreover, if $\mc{F}$ is super, then $\mc{F}'$ is super.
\item[(iii)] Given $\aA|V_i| \leq n_i \leq k\aA|V_i|$ for all $i \in [r]$ and $h_e \geq \aA|\Cols_e|/k$ for all $e \in E(R)$, if $V_i'$ is a uniform random subset of $V_i$ of size $n_i$
and $\Cols'_e$ is a uniform random subset of $\Cols_e$ of size $h_e$ and $\mc{F}$ is super,
then with high probability, $\mc{F}'$ is a super template with parameters $(\aA m,\eps/\aA,d^2/16,\dD/k)$.
\item[(iv)] Suppose $\mc{F}$ is half-super. For all $c \in \Cols$, there is $G_c' \subseteq G_c$ such that, defining $\bm{G}':=(G_c':c \in \Cols)$, the template $(\mc{V},\Cols,\bm{G}')$ is super with parameters $(m,\eps',d^2/2,\dD)$.
\end{enumerate}
\end{lemma}

\begin{proof}
For~(i), we have $\aA m \leq |V_i'| \leq \aA m/(\dD/k)$
and $|\Cols_e'| \geq \aA\dD m/k$.
Also, by Lemma~\ref{lm:slice}(i), for each $e\in E(R)$, $(\bm{G}')^e$ is $(\eps/\aA,d/2)$-regular. Thus $\mc{F}'$ is a template with parameters $(\aA m,\eps/\aA,d/2,\dD/k)$.

For (ii), showing that $\mc{F}'$ is a template with the given parameters is similar to~(i). Note that if $\mc{F}$ is super, then $(\bm{G}')^e$ is $(2\eps,d/2)$-superregular for each $e\in E(R)$ by Lemma \ref{lm:slice}(ii) and thus $\mc{F}'$ is a super $R$-template with parameters $(m/2,2\eps,d/2,\dD/2)$.

For (iii), if $\mc{F}$ is super, then $\bm{G}^e$ is $(\eps,d)$-superregular for each $e\in E(R)$. Thus by Lemma \ref{lm:slice}(iii), with high probability, $(\bm{G}')^e$ is $(\eps/\aA,d^2/16)$-superregular for each $e\in E(R)$. Therefore, by~(i), $\mc{F}'$ is a super template with parameters $(\aA m,\eps/\aA,d^2/16,\dD/k)$ with high probability.

Part~(iv) follows immediately from Lemma~\ref{lm:hsrtosr}.
\end{proof}


Edges which lie in many graphs $G_c$ are particularly useful for embedding and thus we define the (simple, uncoloured) \emph{thick graph} consisting of all such edges.

\begin{defn}[Thick graph of a template]
Given $\lL>0$ and an $R$-template $\mc{F}=(\mc{V},\Cols,\bm{G})$, let $T^\lL_{\mc{F}}$ be the simple $2$-graph
with vertex set $V$ such that $xy \in E(T^\lL_{\mc{F}})$ whenever $xy \in G_c$ for at least $\lL|\Cols_{ij}|$ colours $c \in \Cols_{ij}$ where $x \in V_i$, $y \in V_j$ and $ij \in E(R)$. We call $T^\lL_{\mc{F}}$ the \emph{$\lL$-thick graph of $\mc{F}$}.
\end{defn}

The following proposition states that the thick graph of a bipartite semi-super template has a spanning half-superregular subgraph. This is very useful for embedding as one can use the usual blow-up lemma to embed a bounded degree graph into the thick graph, and then greedily assign colours.

\begin{prop}\label{lm:thick}
Let $0<1/m \ll \eps \ll \lL \ll d,\dD,1/r \leq 1$
where $r \geq 2$ is an integer. 
Let $R$ be a graph on $r$ vertices
and let $\mc{F}=(\mc{V},\Cols,\bm{G})$ be a semi-super $R$-template with parameters $(m,\eps,d,\dD)$. Then for all $ij \in E(R)$, $T^\lL_{\mc{F}}[V_i,V_j]$ is  $(\eps,d/2)$-half-superregular. 
\end{prop}

\begin{proof}
Let $ij \in E(R)$ and let $U_i\subseteq V_i$ and $U_j\subseteq V_j$ be any two subsets such that $|U_i|\geq \eps |V_i|$ and $|U_j|\geq \eps |V_j|$. Let $\ell_{ij} := e(T^\lL_{\mc{F}}[U_i,U_j])$. By regularity we have $|d_{G^{(3)}}(U_i,U_j,\Cols_{ij})-d_{G^{(3)}}(V_i,V_j,\Cols_{ij})|\leq \eps$ and thus we get
$$(d-\eps)|U_i||U_j||\Cols_{ij}|\leq e_{G^{(3)}}(U_i,U_j,\Cols_{ij}) \leq \ell_{ij}|\Cols_{ij}|+\lL(|U_i||U_j|-\ell_{ij})|\Cols_{ij}|.$$
Therefore, we have $\ell_{ij}/(|U_i||U_j|)\geq (d-\eps-\lL)/(1-\lL)\geq d/2$.
Similarly, for any $x\in V_i$, let $\ell_x := d_{T^\lL_{\mc{F}}}(x,V_j)$. By semi-superregularity, we have $\sum_{c\in \Cols_{ij}}d_{G_c}(x)\geq d|V_j||\Cols_{ij}|$. Then we get $$d|V_j||\Cols_{ij}|\leq \ell_x|\Cols_{ij}|+\lL(|V_j|-\ell_x)|\Cols_{ij}|$$ and thus $\ell_x/|V_j|\geq (d-\lL)/(1-\lL)\geq d/2$. The half-superregularity then follows. 
\end{proof}

\section{Embedding lemmas}\label{sec:embed}

In this section we state a series of embedding lemmas which we will combine to prove our transversal blow-up lemma, stated at the end of the section. These lemmas are also useful in their own right in applications.

Each lemma is of the following type: we are given an $R$-template and a bounded degree graph $H$ whose partition matches the template. For a small number of vertices $y$ in $H$ we are given large \emph{target sets} $T_y$ where $y$ must be embedded.
The output of the lemma is a transversal embedding of $H$, which consists of an embedding $\tau$ of vertices and an embedding $\sS$ of colours, so $\tau(x)\tau(y) \in G_{\sS(xy)}$ for all $xy \in E(H)$, and such that $\tau(y) \in T_y$ whenever there is a target set.

The first such lemma applies to embed a small graph $H$ into a template. In fact we only embed some of $H$ while finding large \emph{candidate sets} for the rest of $H$. This means that we set aside some colours so that, for each unembedded vertex $y$
and each of its embedded neighbours $x$, there is a distinct colour $\sS(xy)$, and a large vertex set so that if the image of $y$ is chosen in this set, then we can extend our embedding to a transversal embedding that uses the specified colours.
That is, for every $z$ in this set we have $\tau(x)z \in G_{\sS(xy)}$.
The proof embeds vertices one by one, at each step fixing the colours that will be used to future neighbours, and is similar to the `partial embedding lemma', an uncoloured version, in~\cite{BST3}.

\begin{lemma}[Embedding lemma with target and candidate sets]\label{lm:partial}
Let $0 < 1/m \ll \eps,\gG \ll \nu' \ll \nu,d,\dD,1/\DD,1/r \leq 1$
where $r \geq 2$ is an integer.
\begin{itemize}
\item Let $R$ be a graph on vertex set $[r]$ such that
$\mc{F} = (\mc{V},\Cols,\bm{G})$ is an $R$-template with parameters $(m,\eps,d,\dD)$.
\item Let $H$ be a graph with
$\DD(H) \leq \DD$ for which there is a graph homomorphism
$\phi:V(H) \to V(R)$ such that $|\phi^{-1}(j)| \leq \gG m$ for all $j \in [r]$,
and suppose there is a partition $V(H) = X \cup Y$ where $E(H[Y])=\emptyset$.
\item For every $w \in V(H)$, suppose
there is a set $T_w \subseteq V_{\phi(w)}$ with $|T_w| \geq \nu m$.
\end{itemize}
\noindent
Then there are injective maps $\tau: X \to V$ and $\sS:E(H) \to \Cols$ such that
\begin{enumerate}[(i)]
\item $\tau(x) \in T_x$ for all $x \in X$;
\item $\sS(xy) \in \Cols_{\phi(x)\phi(y)}$ for all $xy \in E(H)$, and if $xx' \in E(H[X])$, then $\tau(x)\tau(x') \in E(G_{\sS(xx')})$;
\item for all $y \in Y$ there exists $C_y \subseteq V_{\phi(y)} \sm \tau(X)$
such that $C_y \subseteq \bigcap_{x \in N_H(y) \cap X}N_{G_{\sS(xy)}}(\tau(x)) \cap T_y$ and $|C_y| \geq \nu' m$.
\end{enumerate}
\end{lemma}

\begin{proof}
Fix an ordering of $V(H)$ where all vertices of $X$ come before any vertex of $Y$, and for each $x \in X$, let $N^<(x)$ be the set of neighbours of $x$ in $H$ which appear before $x$ in the ordering, and $N^>(x)$ the set of those which appear after.
At step ($x$) (with $x \in X$), we will choose $\tau(x)$
and $\sS(xy)$ for all $y \in N^>(x)$.
Initially, define candidate sets $C_w:=T_w$ for all $w \in V(H)$ and $C_{xy}:=\Cols_{\phi(x)\phi(y)}$ for all $xy \in E(H)$.
The following comprises step ($x$), which we perform for each $x \in X$ in order.
\begin{enumerate}[label=({$x$,}\arabic*),ref=\arabic*]
\item\label{del1} For all $y \in N^>(x)$, delete all vertices $v \in C_x$ with $\sum_{c \in C_{xy}}|N_{G_c}(v) \cap C_y| < (d-\eps)|C_{xy}||C_y|$.
\item\label{choosetau} Choose $\tau(x) \in C_x$.
\item\label{deltau} For all $u \in V(H)$, delete $\tau(x)$ from $C_u$.
\item\label{choosesig} For each $y \in N^>(x)$ in order:
\begin{enumerate}[label=({$x$,$y$,}\ref{choosesig}.\arabic*),ref=\ref{choosesig}.\arabic*]
\item\label{d1} delete all colours $c \in C_{xy}$ with $|N_{G_c}(\tau(x)) \cap C_y| < d|C_y|/2$,
\item\label{d2} choose $\sS(xy) \in C_{xy}$,
\item\label{d3} for all $uv \in E(H)$, delete $\sS(xy)$ from $C_{uv}$,
\item\label{d4} delete all vertices $v \in C_y$ with $v \notin N_{G_{\sS(xy)}}(\tau(x))$.
\end{enumerate}
\end{enumerate}
We claim that, at the end of the process, $|C_x|,|C_{xy}| \geq \nu'm$ for all $x \in V(H)$ and $xy \in E(H)$.
First, let us see why the claim implies the lemma.
Since candidate sets are never empty and every edge in $H$ is incident to $X$, $\tau: X \to V$ and $\sS:E(H) \to \Cols$ are defined, and by step~(\ref{deltau}) and~(\ref{d3}), they are both injections.
For~(i), we choose $\tau(x)$ from $C_x$ which is always a (non-empty) subset of $T_x$.
For~(ii), we choose $\sS(xy)$ from $C_{xy}$ which is always a (non-empty) subset of $\Cols_{\phi(x)\phi(y)}$.
If $xx' \in E(H)$ where $x'$ appears after $x$ in the ordering, then we choose $\tau(x')$ from $C_{x'}$ and by step~($x$,$x'$,\ref{d4}) we have $\tau(x') \in N_{G_{\sS(xx')}}(\tau(x))$.
For~(iii), given the claim and the fact that $Y$ comes after $X$ in the ordering, it suffices to show that, for all $y \in Y$ and $x \in N^<(y)$, we have $C_y \subseteq N_{G_{\sS(xy)}}(\tau(x))$.
Noting that $x \in X$, this is a consequence of step~($x$,$y$,\ref{d4}).

It remains to prove the claim.
We suppose that it is true up until step ($x$).
The vertex candidate set $C_x$ can only shrink at steps~($x$,\ref{del1}), (\ref{deltau}) and at step ($u$,$x$,\ref{d4}) for $u \in N^<(x)$.
Proposition~\ref{lm:standard} and the fact that, currently, $|C_x| \geq \nu'm > \eps|V_{\phi(x)}|$, imply that at step~($x$,\ref{del1}), at most $\DD\eps|C_x|$ vertices are deleted.
At step~(\ref{deltau}), at most $|X|$ vertices are deleted.
At step~($u$,$x$,\ref{d4}), 
the colour $\sS(ux)$ was chosen from $C_{ux}$ which has $|N_{G_{\sS(ux)}}(\tau(u)) \cap C_x| \geq d|C_x|/2$ due to step~($u$,$x$,\ref{d1}). Thus $C_x$ shrinks by a factor of at most $d/2$ at this step. Therefore
$$
|C_x| \geq ((d/2)^\DD-\DD\eps)|T_x|-|X| \geq ((d/2)^\DD-\DD\eps)\nu - r\gG)m \geq \nu(d/3)^\DD m > \nu'm.
$$
The colour candidate set $C_{xy}$, with $x$ before $y$ in the ordering, can only shrink at steps~($x$,$y$,\ref{d1}) and (\ref{d3}).
At most $e(H)$ colours are lost at step~(\ref{d3}).
At step~($x$,$y$,\ref{d1}), we chose $\tau(x)$ from $C_x$ which, immediately after~($x$,\ref{del1}), satisfies $\sum_{c \in C_{xy}}|N_{G_c}(\tau(x)) \cap C_y| \geq (d-\eps)|C_{xy}||C_y|$. 
Writing $A \subseteq C_{xy}$ for the subset of colours which are not deleted at step~($x$,$y$,\ref{d1}), we have $(d-\eps)|C_{xy}||C_y|\leq |A||C_y|+(d/2)(|C_{xy}|-|A|)|C_y|$ and thus $|A| \geq (d/2-\eps)|C_{xy}|$.
Therefore
$$
|C_{xy}| \geq (d/2-\eps)|\Cols_{\phi(x)\phi(y)}|-e(H) \geq (d/2-\eps)\dD m-\DD r\gG m \geq \dD d m/3  \geq \nu'm.
$$
This completes the proof of the claim, and hence of the lemma.
\end{proof}

The next ingredient is a version of the above lemma where we are embedding a small graph $H$ such that a small fraction of its vertices have target sets, but additionally we now have a very small set of colours which must be used in the embedding (together with any other colours).
These prescribed colours will be used on an induced matching $M$ in $H$. This matching, together with its neighbours, will be embedded greedily and then Lemma~\ref{lm:partial} will apply to extend the embedding to the whole of $H$.

\begin{lemma}[Embedding lemma with targets and prescribed colours]\label{lm:partialcol}
Let $0 < 1/m \ll \eps \ll \aA,\gG \ll \lL_1,\lL_2 \ll \nu,d,\dD,1/\DD,1/r \leq 1$
where $r \geq 2$ is an integer.
\begin{itemize}
\item Let $R$ be a graph on vertex set $[r]$ and let $(\mc{V},\Cols,\bm{G})$ be an $R$-template with parameters $(m,\eps,d,\dD)$.
\item Let $H$ be a graph with $\DD(H) \leq \DD$ for which there is a graph homomorphism
$\phi:V(H) \to V(R)$ such that $|\phi^{-1}(j)| \leq \lL_1 m$ for all $j \in [r]$
and $e(H[\phi^{-1}(i),\phi^{-1}(j)]) \geq \lL_2 m$ for all $ij \in E(R)$.
\item Suppose there is a set $W \subseteq V(H)$ with $|W| \leq \aA m$, such that for all $w \in W$
there is a set $T_w \subseteq V_{\phi(w)}$ with $|T_w| \geq \nu m$.
\item For each $e \in E(R)$, let $\ms{D}_e \subseteq \Cols_e$ be a set of at most $\gG m$ colours
and let $\ms{D}=\bigcup_{e \in E(R)}\ms{D}_e$, and suppose that $e(G_c[V_i,V_j]) \geq d|V_i||V_j|$ for all
$c \in \ms{D}_{ij}$ and $ij \in E(R)$, and $|\Cols_e \sm \ms{D}| \geq d|\Cols_e|$ for all $e \in E(R)$.
\end{itemize}
Then there are injective maps $\tau: V(H) \to V$ and $\sS:E(H) \to \Cols$ such that
\begin{enumerate}[(i)]
\item $\tau(x) \in V_{\phi(x)}$ for all $x \in V(H)$;
\item $\tau(w) \in T_w$ for all $w \in W$;
\item for all $xx' \in E(H)$ we have $\tau(x)\tau(x') \in E(G_{\sS(xx')})$;
\item $\ms{D} \subseteq \sS(E(H))$.
\end{enumerate}
\end{lemma}

\begin{proof}
Let $e_1,\ldots,e_s$ be an arbitrary ordering of $E(R)$.
Let $\ms{D}_{e_i}' := \ms{D}_{e_i} \sm (\ms{D}_{e_1} \cup \ldots \cup \ms{D}_{e_{i-1}})$ for all $i \in [s]$, so that the $\ms{D}_e'$ over $e \in E(R)$ are pairwise disjoint, and their union equals $\ms{D}$.
We claim that for each $e_i=jj' \in E(R)$, we can choose a matching $M_{jj'}\subseteq H[\phi^{-1}(j),\phi^{-1}(j')]$ such that, writing $M := \bigcup_{e \in E(R)}M_e$, we have
\begin{itemize}
\item $V(M) \cap W=\emptyset$ and 
$M$ is an induced matching in $H$ such that for every $y \in V(H) \sm V(M)$, there is $xx' \in E(M)$ such that $N_H(y,V(M)) \subseteq \{x,x'\}$;
\item $|M_e|=|\ms{D}_e'|$ for each $e\in E(R)$.
\end{itemize}
We can do this greedily, as follows.
Suppose we have found $M_{e_1},\ldots,M_{e_{i-1}}$ with the required properties, for some $i \geq 1$.
We show how to find $M_{e_i}$, where we write $e_i:=jj'$.
Vizing's theorem states that every graph $J$ with maximum degree $\DD$ can be properly edge-coloured with at most $\DD+1$ colours. Thus $J$ contains a matching of size $\lceil e(J)/(\DD+1)\rceil$.
Thus $H[V_j,V_{j'}]$ contains a matching of size $\lceil\lL_2 m/(\DD+1)\rceil$.
Now, from this matching, delete $W$ and any vertex at distance at most two to any vertex in a previously found matching.
There are at most $r-1$ neighbours $h$ of $j$ in $R$ for which we have defined $M_{jh}$, and for each one we delete at most $(1+\DD+\DD^2)v(M_{jh}) \leq 4\DD^2\gG m$ such vertices.
The total number of deleted vertices is at most $|W|+8r\DD^2\gG m$, which leaves a matching $M'$ of size at least $\lL_2 m/(4\DD)$.
Now we will greedily choose a suitable submatching. 
Suppose we have chosen a subset $E(M^q) := \{x_1y_1,\ldots,x_{q-1}y_{q-1}\} \subseteq E(M')$ where $1 \leq q \leq |\ms{D}_{jj'}'|$ and the distance in $H\sm M^q$ between any pair of vertices in $V(M^q)$ is at least three.
To choose $x_qy_q$, we delete all vertices in $V(M')$ at distance in $H\sm M^q$ at most two. The number of deleted vertices is at most $(1+\DD+\DD^2)2q \leq 4\DD^2\gG m < \lL_2 m/(4\DD)$.
Thus we can always choose $x_qy_q$ from the remaining vertices of $M'$.
We claim that $M^{q+1} := M^q\cup\{x_qy_q\}$ is a suitable extension of $M^q$.
Indeed, $x_q,y_q$ have no neighbours in $M^q$ so $M^{q+1}$ is induced, and if $x_q$ shared a neighbour $z$ with some other $y \in M^{q}$, then it would be at distance two from $y$, a contradiction.
Thus we can obtain $M_{e_i}$ and hence $M_{e_1},\ldots,M_{e_s}$ satisfying both required properties.

Let $X := V(M)$ and $Y := N_H(X) \sm X$.
We will embed $M$ into the template first, defining injections $\tau : X \to V$ and $\sS : E(M) \cup E(H[X,Y]) \to \ms{D}$ so that $\sS(M_e)=\ms{D}_e'$ for all $e \in E(R)$.
We want to choose good images for the vertices $x \in X$ so that there are many choices for each neighbour $y \in N_H(x) \sm X$. 

Now we will embed $M$ greedily, 
using the fact that $e(G_c) \geq d|V_j||V_{j'}|$ for all $c \in \ms{D}_{jj'}$ and $\gG \ll d$.
For each $p\leq e(M)$, let 
$i \leq s$ be such that $e(M_{e_1})+\ldots e(M_{e_{i-1}}) \leq p < e(M_{e_1})+\ldots + e(M_{e_i})$.
Let $\wt{M}_p := M_{e_1} \cup \ldots \cup M_{e_{i-1}} \cup M'_{e_i}$ where $e(\wt{M}_p)=p$ and
we have fixed an ordering of $E(M_{e_i})$ and $M'_{e_i}$ is the (possibly empty) matching which consists of an initial segment.
Let $X_p := V(\wt{M}_p)$ and $Y_p := N_H(X_p) \sm X_p$.
Suppose we have found injections $\tau :  X_p \to V$ and $\sS: E(\wt{M}_p) \cup E(H[X_p,Y_p]) \to \ms{C}$ such that
\begin{itemize}
\item $\tau(x) \in V_{\phi(x)}$ for all $x \in X_p$;
\item $\tau(x)\tau(y) \in G_{\sS(xy)}$ for all $xy \in E(\wt{M}_p)$;
\item $\sS(E(M_{e_h})) = \ms{D}_{e_h}'$ for all $h < i$ and $\sS(E(M_{e_i}')) \subseteq \ms{D}_{e_i}'$;
\item for all $y \in Y_p$ there exists $C_y \subseteq V_{\phi(y)}$
such that $C_y \subseteq \bigcap_{x \in N_H(y) \cap X_p}N_{G_{\sS(xy)}}(\tau(x))$, and $|C_y| \geq d|V_{\phi(y)}|/6$.
\end{itemize}
We want to extend $\tau$ and $\sS$ by embedding an edge $xx'$ in $M_{e_i} \sm M_{e_i}'$ with a colour $c^* \in \ms{D}_{e_i}'$ 
and choosing colours and candidate sets for its unembedded neighbours. So fix any such $c^*$ and let $\sS(xx') := c^*$.
Let $N(x),N(x')$ be the set of neighbours of $x,x'$ respectively in $H \sm M$.
The first property of $M$ implies that $N(x),N(y),V(M)$ are pairwise disjoint for all $y \in V(M) \sm \{x,x'\}$, and similarly for $x'$.
However, we could have $N(x) \cap N(x') \neq \emptyset$
(for example if $H$ is a union of vertex-disjoint triangles).

First, we will define $\tau(x)$ and colours and candidate sets for every vertex in $N(x)$.
Write $e_i=jj'$ and $U_h := V_h \sm X_p$ for $h=j,j'$.
We have $e(G_{c^*}[U_j,U_{j'}]) \geq d|V_j||V_{j'}| - 2\gG m(|V_j|+|V_{j'}|) \geq d|V_j||V_{j'}|/2$.
Let
$$
Z(x) := \{v \in U_j: d_{G_{c^*}}(v,U_{j'}) \geq d|U_{j'}|/4\}.
$$
A simple counting argument implies that $|Z(x)| \geq d|U_j|/4$.
We will choose $\tau(x) \in Z(x)$ and candidate sets $C_y$ for each $y \in N(x)$, as follows.
Order $N(x)$ as $y_1,\ldots,y_\ell$, where $\ell \leq \DD$, and let $a_1,\ldots,a_\ell$ be such that $\phi(y_i)=a_i$, so $a_i \in [r]$.
Suppose we have found, for $i \geq 0$,
\begin{itemize}
\item $c_i \in \Cols_{ja_i}$
such that $c^*,c_1,\ldots,c_i$ are all distinct and disjoint from $\sS(E(\wt{M}_p))$
and
\item a set $Z_i(x) \subseteq Z(x)$ with $|Z_i(x)| \geq (d/6)^i|Z(x)|$ and $d_{G_{c_h}}(z,V_{a_h}) \geq d|V_{a_h}|/6$ for all $h \in [i]$ and $z \in Z_i(x)$.
\end{itemize}
Now, $\bm{G}^{ja_{i+1}}[Z_i(x),V_{a_{i+1}}]$ is $(\sqrt{\eps},d/2)$-regular by Lemma~\ref{lm:slice}(i) (the slicing lemma),
so Lemma~\ref{lm:standard}(ii) implies there are at least $(1-\sqrt{\eps})|\Cols_{ja_{i+1}}|$ colours $c \in \Cols_{ja_{i+1}}$ with $e(G_c[Z_i(x),V_{a_{i+1}}]) \geq (d/2-\sqrt{\eps})|Z_i(x)||V_{a_{i+1}}| \geq d|Z_i(x)||V_{a_{i+1}}|/3$. Let $c_{i+1}$ be any such colour which does not lie in the set $\{c^*,c_1,\ldots,c_i\} \cup \sS(E(\wt{M}_p))$, which exists since
the number of available colours is at least $|\Cols_{ja_{i+1}} \sm \ms{D}| - 1 - \DD \geq d|\Cols_{ja_{i+1}}|/2 \geq d\dD m/2$.
By a simple counting argument, the number of vertices $z \in Z_i(x)$ with $d_{G_{c_{i+1}}}(z,V_{a_{i+1}}) \geq d|V_{a_{i+1}}|/6$ is at least $d|Z_i(x)|/6$, and we let the set of these vertices be $Z_{i+1}(x)$.
Thus we can complete the iteration and find $Z_\ell(x)$.
Now let $\tau(x)$ be an arbitrary vertex of $Z_\ell(x)$,
and for each $i \in [\ell]$, define the candidate set $C_{y_i} := N_{G_{c_i}}(\tau(x), V_{a_i})$ and $\sS(xy_i):=c_i$.

Next, we will define $\tau(x')$ and colours and candidate sets for every vertex in $N(x')$.
For each $y \in N(x) \cap N(x')$, the candidate set for $y$ will be a subset of $C_y$.
For this, let $Z(x') := N_{G_{c^*}}(\tau(x),U_{j'})$, so $|Z(x')| \geq d|U_{j'}|/4$ since $\tau(x) \in Z(x)$.
Write $N(x')=\{w_1,\ldots,w_k\}$ and let $b_1,\ldots,b_k$ be such that $\phi(w_i)=b_i$.
For each $w_i \in N(x)$, we already defined $C_{w_i}$;
for the other $w_i$, let $C_{w_i} := V_{b_i}$.
Using the fact that $\bm{G}^{j'b_i}[Z(x'),C_{w_i}]$ is $(\sqrt{\eps},d/2)$-regular, exactly the same argument for $N(x')$ as for $N(x)$ implies that we can find
\begin{itemize}
\item for each $i \in [k]$, $c_i' \in \Cols_{j'b_i}$ which are distinct and disjoint from $\{c^*,c_1,\ldots,c_\ell\} \cup \sS(E(\wt{M}_p))$ and
\item $Z_k(x') \subseteq Z(x')$ with $|Z_k(x')| \geq (d/6)^k|Z(x')|$ and $d_{G_{c_i'}}(z,C_{w_i}) \geq d|C_{w_i}|/6$ for all $i \in [k]$ and $z \in Z_k(x')$.
\end{itemize}
We let $\tau(x')$ be an arbitrary vertex of $Z_k(x')$
and for each $i \in [k]$, let $T_{w_i} := N_{G_{c_i'}}(\tau(x'),C_{w_i})$ and $\sS(x'w_i):=c_i'$.
Note that $\tau(x') \in Z(x') \subseteq N_{G_{c^*}}(\tau(x))$ so $\tau(x)\tau(x') \in G_{c^*} = G_{\sS(xx')}$.
For $y_i \in N(x) \sm N(x')$, $C_{y_i}$ has not changed since it was first defined, and we let $T_{y_i} := C_{y_i}$.
We have $|T_y| \geq d^2|V_{\phi(y)}|/36 \geq d^2m/36$ for all $y \in N(x) \cup N(x')$.

Thus we can complete the iteration and embed the whole of $M$.
This completes the required extension of $\tau$ and $\sS$, so we have obtained
$\tau: X \to V$ and $\sS: E(M) \cup E(H[X,Y]) \to \Cols$, with $\sS(E(M))=\ms{D}$.

Finally, we extend the embedding to the whole of $H$ by applying Lemma~\ref{lm:partial}.
Indeed, let $\mc{F}'$ be the subtemplate of $\mc{F}$ induced by $\mc{V}',\mc{\Cols}'$,
where $\mc{V}' = \{V_1',\ldots,V_r'\}$ and $V_i' := V_i \sm \tau(X)$, and $\Cols'_e := \Cols_e \sm (\sS(E(M) \cup E(H[X,Y]))$. Lemma~\ref{lm:tempslice}(ii) implies that $\mc{F}'$ is a template with parameters
$(m/2,2\eps,d/2,\dD/2)$. For each $y \in Y$, 
we have $|T_y \cap V_{\phi(y)}'| \geq d^2|V_{\phi(y)}|/36-|\phi^{-1}(\phi(y))| \geq d^2m/36-\lL_1 m \geq d^2 m/37$.
Similarly for each $w \in W$, we have $|T_w \cap V_{\phi(w)}'| \geq (\nu-\lL_1)m \geq \nu m/2$.
For all vertices $y$ of $H$ for which $T_y$ is not defined, let $T_y := V_{\phi(y)}$.
Thus we can apply Lemma~\ref{lm:partial} with parameters $m/2,2\eps,2\lL_1,\min\{\nu,d^2/37\},d/2,\dD/2$ playing the roles of $m,\eps,\gG,\nu,d,\dD$, target sets $T_y$ and $Y=\emptyset$ (so no candidate sets),
to find the desired embedding.
Most of the required properties are immediate but we justify why~(iii) holds. If $xx' \in E(M)$ then we already showed that $\tau(x)\tau(x') \in G_{\sS(xx')}$.
If $xy \in E(H) \sm E(M)$ where $x \in X$, we have $y \in Y$, and then the choice of $M$ guaranteed that $N_H(y,X) \subseteq \{x,x'\}$ where $xx' \in E(M)$ and, if this neighbourhood equals $\{x\}$, we chose $\tau(y) \in T_y \subseteq N_{G_{\sS(xy)}}(\tau(x))$,
while if it equals $\{x,x'\}$ we chose $\tau(y) \in T_y$ and $T_y \subseteq N_{G_{\sS(xy)}}(\tau(x)) \cap N_{G_{\sS(x'y)}}(\tau(x'))$.
If $xy \in E(H)- X$, then this follows from Lemma~\ref{lm:partial}.
\end{proof}

Next we state the usual blow-up lemma, which we will use in the proof of our transversal blow-up lemma.
Note that the lemma is usually stated in terms of the stronger `superregular' condition rather than `half-superregular', but the same proof applies.
Alternatively, one can apply Lemma~\ref{lm:hsrtosr}.

\begin{theo}[Blow-up lemma~\cite{KSS}]\label{lm:blowup}
Let $0 < 1/m \ll \eps,\aA \ll \nu,d,\dD,1/\DD,1/r \leq 1$.
\begin{itemize}
\item Let $R$ be a graph on vertex set $[r]$.
\item Let $G$ be a graph with vertex classes $V_1,\ldots, V_r$ where $m\leq |V_i| \leq m/\dD$ for all $i \in [r]$,
such that $G[V_i,V_j]$ is $(\eps, d)$-half-superregular whenever $ij \in E(R)$.
\item Let $H$ be a graph with $\DD(H) \leq \DD$ for which there is a graph
homomorphism $\phi: V(H) \to V(R)$
such that $|\phi^{-1}(j)| \leq |V_j|$ for all $j \in [r]$.
\item For each $i \in [r]$, let $U_i \subseteq \phi^{-1}(i)$ with $|U_i| \leq \aA m$
and suppose there is a set $T_x \subseteq V_i$ with $|T_x| \geq \nu m$ for all $x \in U_i$.
\end{itemize}
Then there is an embedding of $H$ inside $G$ such that for every $i \in [r]$, every
$x \in U_i$ is embedded inside $T_x$.
\end{theo}


The final embedding lemma that we prove in this section applies to embed a \emph{spanning graph} $H$ which consists of small components and such that a small fraction of its vertices have target sets. 
The template $\mc{F}$ is required to be semi-super (every vertex has large total degree) and rainbow, since $H$ could have many edges in every pair, and additionally $\mc{F}$ must contain many more colours than required for a transversal embedding.

\begin{lemma}[Embedding lemma with extra colours]\label{lm:approx}
Let $0 < 1/m \ll \eps,\mu,\aA \ll \bB \ll \nu,d,\dD, 1/\DD,1/r \leq 1/2$ where $r$ is an integer.
\begin{itemize}
\item Let $R$ be an $r$-vertex graph and let $\mc{F}=(\mc{V},\Cols,\bm{G})$ be a semi-super rainbow $R$-template with parameters $(m,\eps,d,\dD)$.
\item Let $n:=|V|$ and
suppose that $H$ is an $n$-vertex graph with $\DD(H) \leq \DD$ which is the union of vertex-disjoint components of size at most $\mu n$, and there is a graph homomorphism $\phi$ of $H$ into $R$ such that $|\phi^{-1}(j)| = |V_j|$ for all $j \in [r]$ and $e(H[\phi^{-1}(i),\phi^{-1}(j)])\leq |\Cols_{ij}|-\bB m$ for all $ij \in E(R)$.
\item For each $i \in [r]$, let $U_i \subseteq \phi^{-1}(i)$ with $|U_i| \leq \aA m$ and suppose there is a set $T_x \subseteq V_i$ with $|T_x| \geq \nu m$ for each $x \in U_i$.
\end{itemize}
Then there is a transversal embedding of $H$ inside $\mc{F}$ such that for every $i \in [r]$, every $x \in U_i$ is embedded inside $T_x$.
\end{lemma}

\begin{proof}
Choose new parameters $\mu',\gG,\zZ,\lL$ satisfying $\eps,\mu,\aA \ll \mu' \ll \gG \ll \zZ \ll \lL \ll \bB$.
Note that $n$ and $m$ are similar in size, since
\begin{equation}\label{eq:mn}
rm \leq n \leq rm/\dD,
\end{equation}
so in particular $\dD n/r \leq m \leq |V_j| \leq m/\dD \leq n/(r\dD)$ for each $j$.

We have that $H$ is the vertex-disjoint union of connected components $H_1,\ldots,H_t$ each of size at most $\mu n$.
We will partition each $\phi^{-1}(j)$ into $s$ parts of size roughly $\gG m$ and one part of size roughly $\mu' m$ that respect components.
For each $h \in [t]$ and
$j \in [r]$, let $A_{hj} := V(H_h) \cap \phi^{-1}(j)$ and $a_{hj}:=|A_{hj}|$.
Let $t^*$ be the largest integer such that $b_{0j}:=|B_{0j}| \geq \mu'm$ for all $j \in [r]$, where
$B_{0j} := A_{t^*+1,j} \cup \ldots \cup A_{tj}$.
Obtain $s \in \mb{N}$ and $B_{ij} \subseteq \phi^{-1}(j)$ for $i \in [s]$, $j \in [r]$ iteratively as follows.
Let $\ell_0:=0$ and $a_{0j}:=0$ for all $j \in [r]$ and do the following for $i \geq 1$.
If
$$
|\phi^{-1}(j)|-(a_{1j}+\ldots + a_{\ell_{i-1}j})>q:=2(\DD+1)^{r-1}\gG m
$$
for all $j \in [r]$, let $B_{ij}:= A_{\ell_{i-1}+1,j}\cup\ldots \cup A_{\ell_ij}$ where $\ell_i$ is the smallest integer so that $b_{ij} := |B_{ij}|$ is at least $\gG m$ for all $j \in [r]$.
Otherwise, set $s:= i$, let $B_{sj}:=A_{\ell_{s-1}+1,j}\cup\ldots\cup A_{t^*j}$, and let $b_{sj}:=|B_{sj}|$.
This process always terminates since we remove at least $r\gamma m$ vertices each time, so defines a partition $V_j = B_{0j} \cup B_{1j}\cup\ldots \cup B_{sj}$ for each $j \in [r]$. We have
$$
b_{0j} + b_{1j}+\ldots+b_{sj}=|\phi^{-1}(j)| = |V_j|.
$$
We claim that
\begin{align}\label{eq:bij}
\nonumber \mu' m \leq b_{0j} &\leq 2(\DD+1)^{r-1}\mu' m
\quad\text{for all }j \in [r],\\
\nonumber \gG m \leq b_{ij} &\leq 2(\DD+1)^{2r-2}\gG m\quad\text{for all }i \in [s],j \in [r]\quad\text{and}\\
s &\leq 1/(\dD\gG).
\end{align}

To prove the claim, we first note the following.
Since every $H_h$ is a connected graph with $\DD(H_h) \leq \DD$,
\begin{equation}\label{eq:sizes}
|A_{hj}| \leq |V(H_h)| \leq (\DD+1)^{r-1}|A_{hj'}|\quad
\text{for every }j,j' \in [r] \text{ and }h \in [t].
\end{equation}
(The second inequality can be proved by induction of the number of parts of $H_h$.)

By construction, for every $j \in [r]$ we have 
$b_{0j} \geq \mu'm$.
There at least one $j' \in [r]$ for which $|B_{0j'}| \leq \mu'm+\mu n \leq 2\mu'm$, otherwise we would have chosen a larger $t^*$, and so
the required upper bound on every $b_{0j}$ follows from~(\ref{eq:sizes}).

For the second part, suppose first that $i \in [s-1]$.
Then every $b_{ij} \geq \gG m$ by construction.
Also there is some $j' \in [r]$ for which $b_{ij'} \leq \gG m + \mu n \leq 3\gG m/2$, 
otherwise $\ell_i$ would be smaller.
Thus $b_{ij} \leq 3(\DD+1)^{r-1}\gG m/2$ for all $j \in [r]$ by~(\ref{eq:sizes}).

Now consider $s=i$.
There is at least one $j^* \in [r]$ which caused the partitioning to stop
during the $s$-th step due to
$$
q \geq |\phi^{-1}(j^*)|-(a_{1j^*}+\ldots+a_{\ell_{s-1}j^*}) = |\phi^{-1}(j^*)|-(b_{1j^*}+\ldots+b_{(s-1)j^*})=b_{0j^*}+b_{sj^*}.
$$
Since the partitioning did not stop during step $s-1$, we similarly have
$b_{0j^*}+b_{(s-1)j^*}+b_{sj^*}>q$.
Thus $b_{sj^*} \geq q-2(\DD+1)^{r-1}\mu' m -3(\DD+1)^{r-1}\gG m/2 \geq \gG m$.
Altogether, $\gG m \leq b_{sj^*} \leq q$.

For any $j \in [r]$ which did not cause the partitioning to stop, we have
$b_{0j}+b_{sj} > q$ and hence $b_{sj} >q-2(\DD+1)^{r-1}\mu' m \geq \gG m$.
Moreover,~(\ref{eq:sizes}) implies that
$$
b_{0j}+b_{sj}=a_{\ell_{s-1}+1,j}+\ldots +a_{tj} \leq (\DD+1)^{r-1}(a_{\ell_{s-1}+1,j^*}+\ldots +a_{tj^*}) = (\DD+1)^{r-1}(b_{0j^*}+b_{sj^*}) \leq (\DD+1)^{r-1}q.
$$
Thus $\gG m \leq b_{sj} \leq (\DD+1)^{r-1}q$.

Thus for every $j \in [r]$ we have $\gG m \leq b_{sj} \leq (\DD+1)^{r-1}q$, and hence
$\gG m \leq b_{ij} \leq (\DD+1)^{r-1}q$ for all $i \in [s]$ and $j \in [r]$. 
For the third part, the lower bound implies that $s\gG m \leq |V_j| \leq m/\dD$ and so $d \leq \dD/\gG$,
completing the proof of the claim that~(\ref{eq:bij}) holds.
Another related estimate that will be useful later is
\begin{equation}\label{eq:b0j}
b_{0j}-sr\eps^{1/3} m > \mu' m-\left(\frac{r|V_j|}{\gG m}\eps^{1/3} m\right) > \mu' m - \frac{r\eps^{1/3}m}{\gG} \geq \mu' m/2.
\end{equation}

Let $B^i := \bigcup_{j \in [r]}B_{ij}$.
We have
\begin{align}\label{eq:HBi}
e(H[B^i]) &\leq \DD|B^i| \leq 2(\DD+1)^{2r-1}r\gG m\quad\text{ for all }i \in [s]\quad
\text{ and}\quad\\
\nonumber e(H[B^0]) &\leq \DD|B^0| \leq 2(\DD+1)^{r} r\mu' m.
\end{align}

In the next part of the proof, we will embed $H[B^1],\ldots,H[B^s]$ in turn.
For this, we first partition each $V_j$ into sets for each stage of the embedding, and find a buffer set of colours which will be used in the final part of the embedding.
Given an edge $xy$ of $G$, write $c(xy) := \{c \in \Cols: xy \in G_c\}$.
For each $jj' \in E(R)$ and $x \in V_j$, define
$
M_{jj'}(x) := \{y \in V_{j'}: |c(xy)| \geq d|\Cols_{jj'}|/2\}
$.
We claim that each such set satisfies
$$
|M_{jj'}(x)| \geq d |V_{j'}|/2.
$$
Indeed,
semi-superregularity implies that
$$
d|V_{j'}||\Cols_{jj'}| \leq \sum_{c \in \Cols_{jj'}}d_{G_c}(x,V_{j'}) = \sum_{y \in V_{j'}}|c(xy)| \leq |M_{jj'}(x)||\Cols_{jj'}| + |V_{j'}|d|\Cols_{jj'}|/2,
$$
as required.

For each $j \in [r]$, let $V^0_j \cup V^1_j \cup \ldots \cup V^s_j$ be a random
partition of $V_j$ into parts 
of size $b_{0j}-sr\eps^{1/3} m, b_{1j}+r\eps^{1/3} m,\ldots,b_{sj}+r\eps^{1/3} m$ respectively.
For every $i \in [s]$, $j \in [r]$ and $x \in B_{ij}$, we will embed $x$ into $V^i_j$, which has size slightly larger than necessary.
For each $e \in E(R)$, let $\Cols_e^0$ be a uniform random subset of $\Cols_e$ of size $\zZ|\Cols_e|$
and let $\wt{\Cols}_e := \Cols_e \sm \Cols_e^0$.

By a Chernoff bound, using~(\ref{eq:bij}) and~(\ref{eq:b0j}) to see that all parts are sufficiently large,
we may assume that the following hold:
\begin{enumerate}[label=(C\arabic*),ref=(C\arabic*)]
\item\label{parti} for all $jj' \in E(R)$, $x \in V_j$ and $y \in M_{jj'}(x)$, we have $|c(xy) \cap \Cols_{jj'}^0| \geq d|\Cols^0_{jj'}|/4$;
\item\label{partii}  for all $jj' \in E(R)$ and $x \in V_j$ we have $|M_{jj'}(x) \cap V_{j'}^0| \geq d|V^0_{j'}|/4$;
\item\label{partiii}  for all $0 \leq i \leq s$ and $x \in V(B^i)$ we have $|T_x'| \geq \nu |V^i_{\phi(x)}|/2$, where $T_x' := T_x \cap V_{\phi(x)}^i$.
\end{enumerate}
Indeed, each of these lower bounds is at most half the expectation of the quantity in question.


We iteratively construct the transversal embedding of $H[B^1 \cup \ldots \cup B^s]$.
Suppose that we have obtained an embedding $g_{i-1}$ for some $i \geq 1$,
such that
\begin{enumerate}[label=(\roman*)]
\item $g_{i-1}$ is a transversal embedding of $H[W_{i-1}]$, where $W_{i-1} := B^1 \cup \ldots \cup B^{i-1}$;
\item $g_{i-1}(x) \in V^{i'}_j$ for every $x \in B_{i'j} \subseteq W_{i-1}$;
\item for every $jj' \in E(R)$ and every embedded pair $x,y$ of vertices where $xy \in E(H[V_j,V_{j'}])$, the colour of $xy$ is in $\wt{\Cols}_{jj'}$ and these colours are distinct over the embedding; and
\item if $x \in U_j \cap W_{i-1}$, then $g_{i-1}(x) \in T_x$ for every $j \in [r]$.
\end{enumerate}
That is, $g_{i-1}$ consists of injections $\tau:W_{i-1} \to V$ and $\sS: E(H[W_{i-1}]) \to \Cols$
such that $\tau(x)\tau(y) \in G_{\sS(xy)}$, with the stated properties.

Now we will extend $g_{i-1}$ to a transversal embedding $g_{i}$ of $W_{i} := W_{i-1} \cup B^i$ with the analogous properties for $i$.
First observe that the vertices we are about to embed, that is, $B^i = W_{i}\sm W_{i-1}$, have no previously embedded neighbours.
Therefore it suffices to find a transversal embedding of $B^i$ when $B_{ij}$ is embedded into $V^i_j$ for each $j \in [r]$ which uses unused colours.

Obtain $\Cols_e^i$ from $\wt{\Cols}_e$ by deleting any colour we have used in the embedding $g_{i-1}$.
Every colour cluster in the original template contains extra colours,
so for every $jj' \in E(R)$ , we have
\begin{equation}\label{eq:Cijj'}
|\Cols^i_{jj'}| \geq |\Cols_{jj'}| - \sum_{1 \leq \ell \leq i-1}e(H[B_{\ell j},B_{\ell j'}])-\zZ |\Cols_{jj'}| 
\geq (1-\zZ)|\Cols_{jj'}| - e(H[B^i]) \geq \bB m/2.
\end{equation}
For each $j \in [r]$ and $j' \in N_R(j)$, let
$$
V^{i,j',\rm bad}_j := \{v \in V^i_j : \textstyle\sum_{c \in \Cols_{jj'}^i}d_{G_c}(v,V^i_{j'}) < 2d|V^i_{j'}||\Cols_{jj'}^i|/3\}.
$$
Lemma~\ref{lm:slice}(i) implies that $\bm{G}^i_{jj'} := (G_c[V^i_j,V^i_{j'}]: c \in \Cols_{jj'}^i)$ is $(\sqrt{\eps},d/2)$-regular and thus, by Lemma~\ref{lm:standard}(i) and~(\ref{eq:bij}), we have that $|V^{i,j',\rm bad}_j| < \sqrt{\eps} |V^i_j| < \eps^{1/3}m$.
For each $i \in [s]$ and $j \in [r]$, obtain a subset $Y^i_j$ of $V^i_j$ by removing $V^{i,j',\rm bad}_j$ for each $j' \in N_R(j)$ and enough additional arbitrary vertices so that $|Y^i_j|=|V^i_j|-r\eps^{1/3} m=b_{ij}$.
Let $\mc{Y}^i := \{Y^i_1,\ldots,Y^i_r\}$ and let $\Cols^i:= \bigcup_{e \in E(R)}\Cols^i_e$.

We claim that the subtemplate $\mc{F}_i := (\mc{Y}^{i},\Cols^i,\bm{G}^i)$ of $\mc{F}$ induced by $\mc{Y}^i$, $\Cols^i$ is a semi-super rainbow $R$-template with parameters $(\gG m,\sqrt{\eps},d/2,1/(2(\DD+1)^{2r-2}))$.
That vertex clusters have suitable sizes follows from~(\ref{eq:bij}), and that colour clusters have suitable sizes follows from~(\ref{eq:Cijj'}).
We have seen that $\bm{G}^i_{jj'}$ is $(\sqrt{\eps},d/2)$-regular.
Moreover, for all $jj' \in E(R)$ and $v \in Y^i_j$, since $v \notin V^{i,j',\rm bad}_j$, we have
$$
\sum_{c \in \Cols^i_{jj'}}d_{G_c}(v,Y^i_{j'})
\geq 2d|V^i_{j'}||\Cols^i_{jj'}|/3-\eps^{1/3}m|\Cols^i_{jj'}| > d|V^i_{j'}||\Cols^i_{jj'}|/2,
$$
so $\bm{G}^i_{jj'}$ is $(\sqrt{\eps},d/2)$-semi-superregular.
This completes the proof of the claim.

Therefore, for each $jj' \in E(R)$, we can apply Proposition~\ref{lm:thick} to 
see that the $\lL$-thick graph $T^\lL_{\mc{F}_i}$ is such that $T^i_{jj'} := T^\lL_{\mc{F}_i}[Y^i_j,Y^i_{j'}]$ is $(\sqrt{\eps},d/4)$-half-superregular for all $jj' \in E(R)$.
Apply Theorem~\ref{lm:blowup} (the blow-up lemma) with target sets $T_x'$
and parameters $\gG m,\sqrt{\eps},\sqrt{\aA},\nu/2,d/4$ playing the roles of $m,\eps,\aA,\nu,d$
to embed $H[B ^i]$ into $T^i_{jj'}$
such that for every $j \in [r]$, every $x \in U_j \cap B^i$ is mapped to $T_x' \subseteq V^i_j$, which is possible since, by~\ref{partiii}, $|T_x'| \geq \nu\gG m/2$ and $|U_j \cap B^i| \leq \aA m$.
Since for each edge of $T^i_{jj'}$, the number of unused colour graphs it lies in is
$$
\lL|\Cols^i_{jj'}| \stackrel{(\ref{eq:Cijj'})}{>} \lL\bB m/2 > 3(\DD+1)^{2r-1}r\gG m \stackrel{(\ref{eq:HBi})}{>} e(H[B^i]),
$$
we can greedily assign colours to the embedding.
This completes the construction of $g_i$ which satisfies~(i)--(iv).

Thus we can obtain $g_s$ with the required properties.
At the end of this process,
the unembedded part of $H$ is $H[B^0]$,
and the set of vertices of $V_j$ which are not an image of any embedded vertex of $H$
is precisely $Z_j := V^0_j \cup \bigcup_{i \in [s]}(V^i_j \sm Y^i_j)$, which has size $b_{0j}$.
We will use the colours in $\Cols^0:= \bigcup_{e \in E(R)}\Cols^0_e$ to embed each $B_{0j}$ into $Z_j$.

We claim that the subtemplate $\mc{F}_0 := (\mc{Z},\Cols^0,\bm{G}_0)$ of $\mc{F}$ induced by $\mc{Z}:=(Z_j: j \in [r]),\Cols^0$ is a semi-super rainbow template with parameters $(\mu' m,\sqrt{\eps},d^2/16,1/(2(\DD+1)^{r-1}))$.
Indeed,~(\ref{eq:b0j}) implies that vertex clusters have sizes satisfying $\mu' m \leq b_{0j} = |Z_j| \leq 2(\DD+1)^{r-1}\mu' m$.
Colour clusters have size $|\Cols^0_{jj'}| = \zZ|\Cols_{jj'}| \geq \zZ\dD m > \mu' m$.
Lemma~\ref{lm:slice}(i) implies that $\bm{G}^0_{jj'}:=(G_c[V^0_j,V^0_{j'}]: c \in \Cols^0_{jj'})$ is $(\sqrt{\eps},d/2)$-regular for every $jj' \in E(R)$.
Additionally, for all $x \in Z_j \subseteq V_j$,
\begin{align*}
\sum_{c \in \Cols^0_{jj'}}d_{G_c}(x,Z_{j'}) 
&= \sum_{y \in V^0_{j'}}|c(xy) \cap \Cols^0_{jj'}| \geq \sum_{y \in M_{jj'}(x) \cap V^0_{j'}}|c(xy) \cap \Cols^0_{jj'}| \stackrel{\ref{parti},\ref{partii}}{\geq} d |V^0_{j'}|/4 \cdot d|\Cols^0_{jj'}|/4\\&=d^2|V^0_{j'}||\Cols^0_{jj'}|/16,
\end{align*}
so the template is semi-super.

As before, we can apply Proposition~\ref{lm:thick} to see that the $\lL$-thick graph $T^\lL_{\mc{F}_0}$ is such that $T^0_{jj'} := T^\lL_{\mc{F}^0}[Z^0_j,Z^0_{j'}]$ 
is $(\sqrt{\eps},d/4)$-half-superregular for each $jj' \in E(R)$.
We have $\aA m < \sqrt{\aA}|Z_j|$,
and~\ref{partiii} implies that $|T'_x| \geq \nu|V^0_{\phi(x)}|/2 \geq \nu|Z_{\phi(x)}|/3$ for all $x \in V(B^0)$ with a target set.
Thus we can apply Theorem~\ref{lm:blowup} (the blow-up lemma)
with target sets $T_x'$ and parameters $\mu'm,\sqrt{\eps},\sqrt{\aA},\nu/3,d/4$
playing the roles of $m,\eps,\aA,\nu,d$
to embed $H[B^0]$ into $T^0_{jj'}$
such that for every $j\in [r]$, every $x \in U_j$ is mapped to $T_x' \subseteq V^0_j$, which is possible by~\ref{partiii}.
The number of colours on each edge of $T^0_{jj'}$ is at least
$$
\lL|\Cols^0_{jj'}| = \lL\zZ|\Cols_{jj'}| \geq \lL\zZ\dD m > 3(\DD+1)^{r} r\mu' m \stackrel{(\ref{eq:HBi})}{>}e(H[B^0]).
$$
Thus we can again greedily assign colours of $\Cols^0_{jj'}$ to obtain a transversal embedding of $H[B^0]$ using colours untouched by the previous embedding, which completes the transversal embedding of $H$.
\end{proof}

Finally we state a transversal blow-up lemma, which we prove in the next section by combining the embedding lemmas of this section.
Its statement is slightly stronger than Theorem~\ref{th:blowupintro} since the template is super rather than half-super, but the proof of Theorem~\ref{th:blowupintro} is a simple matter of applying Lemma~\ref{lm:tempslice}(iv) (essentially Lemma~\ref{lm:hsrtosr}) first.
The details can be found at the end of the next section.

\begin{theo}[Transversal blow-up lemma]\label{th:blowup}
Let $0 < 1/m \ll \eps,\mu,\aA \ll \nu,d,\dD,1/\DD,1/r \leq 1/2$ where $r$ is an integer.
\begin{itemize}
\item Let $R$ be a simple graph with vertex set $[r]$ and let $\mc{F}=(\mc{V},\Cols,\bm{G})$ be a rainbow super $R$-template with parameters $(m,\eps,d,\dD)$.
\item Suppose that $H$ is a $\mu$-separable graph with $\DD(H) \leq \DD$ and there is a graph homomorphism $\phi$ of $H$   into $R$ such that $|\phi^{-1}(j)| = |V_j|$ for all $j \in [r]$ and $e(H[\phi^{-1}(i),\phi^{-1}(j)])=|\Cols_{ij}|$ for all $ij \in E(R)$.
\item For each $i \in [r]$, let $U_i \subseteq \phi^{-1}(i)$ with $|U_i| \leq \aA m$ and suppose there is a set $T_x \subseteq V_i$ with $|T_x| \geq \nu m$ for each $x \in U_i$.
\end{itemize}
Then there is a transversal embedding of $H$ inside $\mc{F}$ such that for every $i \in [r]$, every $x \in U_i$ is embedded inside $T_x$.
\end{theo}


\section{Proof of the transversal blow-up lemma}\label{sec:blowup}

\subsection{Sketch of the proof}

The proof is a combination of the usual blow-up lemma applied to the thick graph, together with the rainbow partial embedding lemmas in the previous section
and the colour absorbing approach pioneered by Montgomery, M\"{u}yesser and Pehova in~\cite{MMP}.

We begin by outlining the steps of the proof in~\cite{MMP}, which is a common sequence of steps for embeddings using absorption, and which we shall also use.
The transversal bandwidth theorem in~\cite{CIKL} also follows this general outline, as well as using the (usual) blow-up lemma as a key tool.
Let $H=H_{\con} \cup H_{\abs} \cup H_{\app} \cup H_{\col} \cup H_{\vx}$ be a suitable partition, which will be carefully chosen.
These parts will be embedded in turn, each time using new colours to extend the transversal embedding.
\begin{itemize}
\item[Step 0.] \emph{Embed a small `connecting graph'}. Embed $H_{\con}$ whose vertex set, as guaranteed by separability, is such that its removal disconnects $H$ into very small components (Lemma~\ref{lm:partial}).
\item[Step 1.] \emph{Find a colour absorber}. Embed $H_{\abs}$ into the edge-coloured graph $\mc{G}$ of $G$,
and find disjoint sets $\ms{A},\ms{B} \subseteq \Cols$ of colours such that, given any $e(H_{\abs})-|\ms{A}|$ colours in $\ms{B}$,
we can choose the colours of the edges of $H_{\abs}$ using exactly those colours and the colours in $\ms{A}$.
\item[Step 2.] \emph{Use most of the colours outside $\ms{A} \cup \ms{B}$}. Embed $H_{\app}$ using most of the colours outside $\ms{A} \cup \ms{B}$ (Lemma~\ref{lm:approx}).
\item[Step 3.] \emph{Use the remaining colours outside $\ms{A} \cup \ms{B}$}. Embed $H_{\col}$ using every unused colour in $\Cols \sm (\ms{A} \cup \ms{B})$ as well as some colours of $\ms{B}$ (Lemma~\ref{lm:partialcol}).
\item[Step 4.] \emph{Embed the remaining vertices of $H$ using colours in $\ms{B}$}. Embed $H_{\vx}$ using colours of $\ms{B}$ (Lemma~\ref{lm:approx}).
\item[Step 5.] \emph{Use the colour absorber}. The colours for $H_{\abs}$ will consist of $\ms{A}$ along with the unused colours in $\ms{B}$.
\end{itemize}
We recall that in~\cite{MMP}, these steps were used to find transversal embeddings of spanning trees, and $F$-factors. When embedding an $F$-factor where $F$ is a small graph, one can embed each copy of $F$ in turn, ensuring that they are disjoint. Similarly, to embed a tree, one can embed small subtrees one by one, ensuring that the single vertex of the current tree that was already embedded matches up.
The new difficulty in our setting (and in~\cite{CIKL}) is that we would like to embed graphs which are more highly-connected.
Nevertheless, they have the property that a small fraction of edges can be removed to produce a graph which consists of small (but still linear) connected components.

One difficulty is that we do not have a minimum degree condition. Every vertex has large total degree and every colour graph is dense, but there could be a small proportion of vertices which do not have any edges of a given colour.

For Step~1, we use
the following lemma from~\cite{MMP} which is their key tool for colour absorption. Here we present a slightly weaker version with modified parameters for simplicity.
\begin{lemma}[\cite{MMP}]\label{lm:mmp}
Let $0 < \lL_1 \ll \lL_2 \ll \lL_3<1$ and let $\ell,m,n$ be integers with $\ell=\lL_1 m$ and $1\leq m \leq \lL_3^2n/8$. Suppose that $G=(U,\Cols)$ is a bipartite graph with $|U|=m$ and $|\Cols|=n$ such that $d(v,\Cols)\geq \lL_3 n$ for each $v\in U$. Then there exist disjoint subsets $\ms{A}, \ms{B} \subseteq \Cols$ with $|\ms{A}|=m-\ell$ and $|\ms{B}|= \lL_2 n$ such that, for every subset $\ms{B}^0$ of $\ms{B}$ with size $\ell$, we can find a perfect matching between $U$ and $\ms{A} \cup \ms{B}^0$.
\end{lemma}


\subsection{Proof of Theorem~\ref{th:blowup}}

\begin{proof}[Proof of Theorem~\ref{th:blowup}]
Without loss of generality, we may assume that $\eps \ll \mu \ll \aA$.
We further define constants $\lL_1,p_{\abs},p_{\col},p_{\vx},\lL_3,\nu'$ so that altogether
$$
0<1/m \ll \eps \ll \mu \ll \aA \ll \lL_1 \ll p_{\abs} \ll p_{\col} \ll p_{\vx} \ll \lL_3 \ll \nu' \ll \nu,d,\dD, 1/\DD,1/r \leq 1/2.
$$
Let $p_{\app}:=1-(p_{\vx}+p_{\abs}+p_{\col})$.
Let $R$ be a graph on vertex set $[r]$, $\mc{F}=(\mc{V},\Cols,\bm{G})$ a rainbow super template with parameters $(m,\eps,d,\dD)$ and $H$ a graph as in the statement. 
Let $n := |V|$ be the number of vertices in the template, which equals $v(H)$.
Note that
\begin{gather}
\nonumber r m \leq n \leq rm/\dD, \quad\text{and}\\
\label{eq:vxpart} m \leq |\phi^{-1}(i)| = |V_i| \leq m/\dD \ \text{ for all }i \in [r],
\quad\text{and}\quad
\dD m \leq |\Cols_e| \leq \DD m/\dD \ \text{ for all }e \in E(R),
\end{gather}
where the final assertion follows from the fact that every $|\Cols_{ij}| = e(H[\phi^{-1}(i),\phi^{-1}(j)]) \leq \DD |\phi^{-1}(i)| \leq \DD m/\dD$.

\medskip
\noindent
{\bf Preparation of $H$.}
First we will choose a partition $H=H_{\con}\cup H_{\abs}\cup H_{\app}\cup H_{\vx}\cup H_{\col}$
such that each part $H_\circ$ has roughly a given number $p_\circ n$ of vertices, and moreover
$e(H_\circ[\phi^{-1}(i),\phi^{-1}(j)]) \approx p_\circ e(H[\phi^{-1}(i),\phi^{-1}(j)])$ for all $ij \in E(R)$ and $\circ\in \{\abs,\app,\vx,\col\}$.
Since $H$ is $\mu$-separable, there is a set $X$ of size at most $\mu n$ such that $H-X$
consists of disjoint components $H_1,\ldots, H_t$, each of size at most $\mu n$.
For all $\ell \in [t]$ and $ij \in E(R)$, let $H^\ell_{ij} := H_\ell[\phi^{-1}(i),\phi^{-1}(j)]$
and let $n^\ell_j := |V(H_\ell) \cap \phi^{-1}(j)|$ and $h^\ell_{ij} := e(H^\ell_{ij})$.
Let $H_\con := H[X]$.
Note that 
\begin{equation}\label{eq:VHcon}
|X|=|V(H_\con)| \leq \mu n \leq r\mu m/\dD \leq \sqrt{\mu}m.
\end{equation}
Independently, for each $\ell\in [t]$, add $H^{\ell}:=\bigcup_{ij \in E(R)}H^\ell_{ij}$ to $H_\circ$ with probability $p_\circ$ for $\circ\in \{\abs, \app, \vx, \col\}$.
Then, for all $i \in [r]$ we have $\mathbb{E}(|V(H_\circ)\cap \phi^{-1}(i)|)=p_\circ|V_i| \pm |X| \geq (p_\circ-\sqrt{\mu}) m$, and for all $ij \in E(R)$ we have
\begin{align*}
\mathbb{E}(e(H_\circ[\phi^{-1}(i),\phi^{-1}(j)]))&=p_\circ\sum_{\ell\in[t]} h^{\ell}_{ij} = p_{\circ}e(H[\phi^{-1}(i),\phi^{-1}(j)])\pm\DD|X| =p_{\circ}|\Cols_{ij}| \pm\DD\sqrt{\mu}m\\
&\geq (p_{\circ}\dD-\DD\sqrt{\mu}) m.
\end{align*}
Now, $t \geq (1-\mu)/\mu \geq 1/(2\mu)$ and $e(H^\ell_{ij}) \leq \DD n^\ell_{i} \leq \DD \mu n$ for all $\ell \in [t]$ and $ij \in E(R)$, 
and each of these expectations is at least $\dD p_{\abs} m/2$, so
a Chernoff bound implies that these values are within a multiplicative factor of $(1\pm\frac{1}{3})$ of their expectations
with probability at least $$
1-4\left(r+\binom{r}{2}\right)\exp\left(\frac{-(\dD p_{\abs} m/2)^2/9}{\frac{1}{2\mu}(\DD\mu n)^2}\right)
\geq 1 - 4r^2\exp\left(\frac{-p_{\abs}^2\dD^4}{18r\DD\mu}\right) \geq \frac{1}{2}.
$$
Thus we may assume that, for all $\circ \in \{\abs,\app,\vx,\col\}$,
\begin{align}
\label{eq:nij} n^\circ_i &:= |V(H_\circ) \cap \phi^{-1}(i)| =(1\pm\tfrac{1}{2})p_\circ|V_i| \in
[p_\circ m/2,3p_\circ m/(2\dD)]\quad\text{for all }i \in V(R);\\
\label{eq:hij} h^{\circ}_{ij} &:= e(H_{\circ}[\phi^{-1}(i),\phi^{-1}(j)]) =
(1\pm\tfrac{1}{2})p_\circ |\Cols_{ij}| \in
[p_{\circ}\dD m/2, 3\DD p_{\circ}m/(2\dD)]\quad\text{for all }ij \in E(R).
\end{align}
We write $H^\circ_{ij} := H_\circ[\phi^{-1}(i),\phi^{-1}(j)]$ for all $ij \in E(R)$. 
Every vertex $x \in V(H) \sm V(H_\con)$ lies in some $H_\circ$ where $\circ\in \{\abs, \app, \vx, \col\}$, and $x$ can only have neighbours inside $H_\circ$ or $V(H_\con)$.
Note here that for each $\circ \in \{\abs,\app,\col,\vx\}$, $H_\circ$ is the union of vertex-disjoint components of size at most $\mu n \leq (2\mu/p_\circ)|V(H_\circ)| \leq \sqrt{\mu}|V(H_{\circ})|$, by~(\ref{eq:nij}). 

To find a transversal embedding of $H$, we need to define two injective maps $\tau: V(H)\rightarrow V$ and $\sS: E(H)\rightarrow \Cols$
where $\tau(x)\tau(y) \in G_{\sS(xy)}$ for all $xy \in E(H)$.
For this, we will define $\tau(x)$ for every vertex in $H_\con$, followed by every vertex in $H_\abs$, $H_\app$, $H_\col$, $H_\vx$.
When defining $\tau$ for vertices in $H_\circ$, we simultaneously define $\sS$ for all future incident edges
except for $\circ=\abs$, for which colours are defined at the end.
Whenever we define $\tau(y)$ for some $y$ we always choose $\tau(y) \in V_{\phi(y)}$ and if $y$ has a target set $T_y$, we ensure $\tau(y) \in T_y$. We also always choose $\sS(xy) \in \Cols_{\phi(x)\phi(y)}$.

\medskip
\noindent
{\bf Step 0.} We first embed $H_{\con}$ into $\mc{F}$ by applying Lemma~\ref{lm:partial}, as follows.
Recall that $X=V(H_\con)$, and define $Y := N_H(X) \sm X$.
Let $H'_{\con}$ be the graph with vertex set $X \cup Y$ and edge set $E(H[X \cup Y]) \sm E(H[Y])$.
Then~(\ref{eq:VHcon}) implies that
\begin{equation}\label{eq:V'Hcon}
|V(H'_\con)| \leq (\DD+1)|V(H_\con)| \stackrel{(\ref{eq:VHcon})}{\leq} 2\DD \sqrt{\mu} m\quad\text{and hence}
\quad e(H'_\con) \leq \mu^{1/3}m.
\end{equation}
For each $w \in V(H'_\con)$ where $T_w$ is not defined, let $T_w := V_{\phi(w)}$.
Lemma~\ref{lm:partial} applied with $2\DD\sqrt{\mu}$ playing the role of $\gG$ (and the other parameters the same) implies that there are injective maps $\tau: X \to V$ and $\sS: E(H'_\con) \to \Cols$ with 
$\tau(x)\tau(x') \in G_{\sS(xx')}$ for all $xx' \in E(H_\con)$, and
so that $\tau(x) \in T_x$ for all $x \in X$, and so that there are candidate sets $C_y$ for each $y \in Y$ where $C_y \subseteq V_{\phi(y)} \sm \tau(X)$ such that $C_y \subseteq \bigcap_{x \in N_H(y) \cap X}N_{G_{\sS(xy)}}(\tau(x)) \cap T_y$ and $|C_y| \geq \nu' m$.

We update the target sets, by defining $T^1_y := C_y \subseteq T_y$ for $y \in Y$ where this set is defined, and
$T_y^1 := T_y$ if $y \notin V(H_{\con}')$ and this set is defined. This updates target sets $T_y^1$ for all unembedded vertices $y$, and all target sets have size at least $\nu'm$.
Note that for all $y \in Y$, if we choose $\tau(y) \in T_y^1$, then $\tau(x)\tau(y) \in G_{\sS(xy)}$ for all $x \in N_H(y) \cap X$, so the colour $\sS(xy)$ will be used as desired.
For each $i \in [r]$, let $U^1_i$ be the set of vertices $y$ in $\phi^{-1}(i)$ with a target set $T_y^1$.
We have
\begin{equation}\label{eq:target}
|U^1_i| \leq |U_i|+|Y| \leq \aA m+2\DD\sqrt{\mu} m \leq 2\aA m.
\end{equation}
This set will only shrink during the rest of the proof as vertices with target sets are embedded.
Some target sets of vertices in $U^1_i$ will also shrink, but will remain large enough.

Let $n^{\con}_i := |V(H_\con) \cap \phi^{-1}(i)|$ for all $i \in [r]$
and $h^\con_{ij} := e(H'_\con[\phi^{-1}(i),\phi^{-1}(j)])$ for all $ij \in E(R)$.
We have
\begin{align}
\nonumber |V_i|&=n^\con_i+n^\abs_i+n^\app_i+n^\vx_i+n^\col_i\quad\text{for all }i \in [r]\quad\text{and}\\
|\Cols_{ij}|&=e(H[\phi^{-1}(i),\phi^{-1}(j)])=h^\con_{ij}+h^\abs_{ij}+h^\app_{ij}+h^\vx_{ij}+h^\col_{ij}\quad\text{for all }ij \in E(R).\label{eq:sum}
\end{align}
Let $\mathcal{F}' = (\mc{V}',\Cols',\bm{G}')$ be the subtemplate of $\mc{F}$ obtained by deleting the vertices $\tau(X)$ and the colours $\sS(E(H_{\con}'))$ from $\mathcal{F}$, so, defining for all $i \in [r]$ and $e \in E(R)$
\begin{gather*}
V'_i := V_i \sm \tau(X), \quad \Cols'_e := \Cols_e \sm \sS(E(H'_\con)),\\
\text{we have }\quad(1-\mu^{1/4})|\Cols_e| \stackrel{(\ref{eq:vxpart}),(\ref{eq:V'Hcon})}{\leq} |\Cols_e'|\stackrel{(\ref{eq:sum})}{=}h^\abs_e+h^\app_e+h^\col_e+h^\vx_e.
\end{gather*}
By Lemma~\ref{lm:tempslice}(ii), $\mc{F}'$ is a rainbow super $R$-template with parameters $(m/2,2\eps,d/2,\dD/2)$.

\medskip
\noindent
{\bf Step 1.}
Next we embed (the vertices but not the colours of) $H_{\abs}$ into $\mathcal{F}'$.
For each $i \in [r]$,
let $V^{\abs}_i \subseteq V_i'$ be a uniform random subset of size $n^\abs_i$.
Lemma~\ref{lm:tempslice}(iii) applied with $\mc{F}',p_{\abs}/3,6$ playing the roles of $\mc{F},\aA,k$ implies that the subtemplate $\mc{F}_\abs$ of $\mc{F}'$ induced by $(V^\abs_i: i \in [r])$ is a rainbow super $R$-template with parameters $(p_{\abs} m/3,\sqrt{\eps},d^2/64,\dD/6)$. Furthermore, a Chernoff bound implies that we may assume
$|T^1_y \cap V_{\phi(y)}^{\abs}| \geq 2\nu' n_i^\abs/3 \geq \nu'|V_i^\abs|/3$ for all $y \in V(H_\abs)$ with a target set $T_y^1$.

Let $V_i'' := V_i'\sm V^\abs_i$.
We have
\begin{equation}\label{eq:Vi''}
|V_i''| \geq |V_i|-2\DD\sqrt{\mu} m-n^\abs_i \stackrel{(\ref{eq:nij})}{\geq} (1-2p_{\abs})|V_i| \geq m/2.
\end{equation}
Lemma~\ref{lm:tempslice}(ii) applied with $\mc{F},2p_\abs$ playing the roles of $\mc{F},\aA$ implies that the subtemplate $\mc{F}''$ induced by $V_i''$ is a rainbow super $R$-template with parameters $(m/2,2\eps,d/2,\dD/2)$,
and a Chernoff bound implies that we may assume
\begin{equation}\label{eq:T1y}
|T^1_y \cap V_{\phi(y)}''| \geq 2(\nu'm-n^\abs_i)/3 \geq \nu' m/2
\end{equation}
for all $y$ with a target set $T^1_y$.

We will embed $H_{\abs}$ into the $\lL_3$-thick graph $T^{\lL_3}_{\mathcal{F}_\abs}$. By Proposition~\ref{lm:thick}, for every $ij \in E(R)$, $T^{\lL_3}_{\mathcal{F}_\abs}[V_i^{\abs},V_j^{\abs}]$ is $(\sqrt{\eps},d^2/128)$-half-superregular. Apply Theorem~\ref{lm:blowup} (the blow-up lemma) with target sets $T^1_y\cap V_{\phi(y)}^{\abs}$ and parameters $\sqrt{\eps},\sqrt{\aA},\sqrt{\nu'},d^2/128$ playing the roles of $\eps,\aA,\nu,d$ to find an embedding $\tau$ of $H_{\abs}$ into $T^{\lL_3}_{\mc{F}_{\abs}}$ such that $\tau(x) \in V_{\phi(x)}$ for all $x \in V(H_\abs)$ and $\tau(y)\in T_y^1$ for every $y \in U^1_i$ with $i \in [r]$.
Note that we haven't yet defined $\sS(e)$ for any $e\in E(H_{\abs})$. However, by our definition, for every such $e=xy$ we have $\tau(x)\tau(y)\in E(G_c)$ for at least $\lL_3 |\Cols_e'|$ colours $c\in \Cols_e'$.
For each $e \in E(R)$, let $G_{\abs, e}$ be the auxiliary bipartite graph with vertex classes $Z_e:= \{\tau(x)\tau(y): xy \in E(H_\abs^e)\}$ and $\Cols'_e$, where $\{\tau(x)\tau(y),c\}$ is an edge whenever $\tau(x)\tau(y) \in E(G_c)$.
Then~(\ref{eq:hij}) implies that $|Z_e|=h^\abs_e \leq 3p_\abs |\Cols_e|/2 \leq 2p_\abs|\Cols_e'|$,
and by construction, every $\tau(x)\tau(y)$ in $Z_e$ has degree at least $\lL_3|\Cols_e'|$ in $G_{\abs,e}$.
For each $e \in E(R)$, define constants $\ell_e$ and $p_e$ via
\begin{eqnarray}
\label{eq:lp} \ell_e &:=&\lL_1 h^\abs_e\quad\text{and}\\
\nonumber p_e|\Cols'_e| &:=& \textstyle (1-p_\abs)|\Cols_e'|-(h^\abs_e-\ell_e)-h^\app_e = h^\vx_e+h^\col_e-p_\abs|\Cols_e'|+\lL_1h^\abs_e \\
&\stackrel{(\ref{eq:hij})}{=}&  (1\pm\tfrac{1}{4})h^\vx_e \in [p_{\vx}|\Cols'_e|/3,3p_{\vx}|\Cols'_e|].
\end{eqnarray}
Thus for each $e \in E(R)$ we can apply Lemma~\ref{lm:mmp} with $G_{\abs,e},Z_e,\Cols'_e,\lL_1,p_e,\lL_3$
playing the roles of $G,U,\Cols,\lL_1,\lL_2,\lL_3$ to obtain disjoint sets $\ms{A}_e, \ms{B}_e\subseteq \Cols_e'$ such that
\begin{enumerate}[label=(Q\arabic*),ref=(Q\arabic*)]
\item\label{partI} $|\ms{A}_e|=h^\abs_e -\ell_e \stackrel{(\ref{eq:hij})}{\leq} 2p_{\abs}|\Cols_e'|$;
\item\label{partII} $|\ms{B}_e|=p_e|\Cols_e'|$;
\item\label{partIII} for any set $\ms{B}^0_e\subseteq \ms{B}_e$ of size $\ell_e$, there exists a colouring using colours from $\ms{A}_e\cup \ms{B}^0_e$ that makes the embedding $\tau$ of $H^{\abs}_e$ rainbow. That is, there is a bijection $\sS_e: E(H^\abs_e) \to \ms{A}_e \cup \ms{B}_e^0$ such that $\tau(x)\tau(y) \in G_{\sS_e(xy)}$ for all $xy \in E(H^\abs_e)$.
\end{enumerate}

\medskip
\noindent
{\bf Preparation for Steps~2--4.}
Recall that the template $\mc{F}''=(\mc{V}'',\Cols',\bm{G}'')$ has vertex clusters $\mc{V}'' := (V_i'' := V_i'\sm V^\abs_i: i \in [r])$, colour clusters $(\Cols'_e: e \in E(R))$ and graphs $(\bm{G}'')^{ij} = (G_c[V_i'',V_j'']: c \in \Cols'_{ij})$, and $\mc{F}''$ is a rainbow super $R$-template with parameters $(m/2,2\eps,d/2,\dD/2)$.
Let $n^\colvx_i := n^\col_i+n^\vx_i$ for each $i \in [r]$ and let $p_\colvx := p_\col+p_\vx$.
During the rest of the proof, for each $\diamond \in \{\app,\colvx\}$ we will identify pairwise disjoint vertex sets $V^\diamond_i$ in the vertex clusters $V_i''$, $i \in [r]$, so that
$$
|V^\diamond_i|=n^\diamond_i,\quad\text{and}\quad
V_i'' = V^\app_i \cup V^\colvx_i,\quad\text{where}\quad
\mc{V}^\diamond := \{V^\diamond_1,\ldots,V^\diamond_r\} \ \forall \diamond\in\{\app,\colvx\}.
$$
We will choose $\tau(x) \in V^\app_{\phi(x)}$
for all $x \in V(H_\app)$ and $\tau(x) \in V^\colvx_{\phi(x)}$ for all $x \in V(H_\col) \cup V(H_\vx)$.
We will identify pairwise disjoint colour sets $\Cols^\circ$, so that
$$
\Cols^\circ = \bigcup_{e \in E(R)}\Cols^\circ_e \ \forall \circ\in\{\app,\col,\vx\},\quad
\Cols_e' = \bigcup_{\circ \in \{\app,\col,\vx\}}\Cols^\circ_e
$$
and for each $xy \in E(H_\circ)$, we will ensure $\sS(xy) \in G_c$ for some $c \in \Cols^\circ_{\phi(x)\phi(y)}$.
Given such sets of vertices and colours for $\circ \in \{\app,\col,\vx\}$, we define $\mc{F}_\circ := (\mc{V}^{\circ'},\Cols^\circ,\bm{G}^\circ)$ to be the subtemplate of $\mc{F}''$ induced by $\mc{V}^{\circ'},\Cols^\circ$, where $\app'=\app$ and $\vx'=\col'=\colvx$.
Note that these templates are always rainbow $R$-templates with pairwise disjoint sets of colours, but $\mc{F}_\col$ and $\mc{F}_\vx$ share the same vertex clusters.

We will now define the vertex sets $V^{\circ'}_i$ for $\circ' \in \{\app,\colvx\}$,
but colour sets will be defined sequentially, given the colours used in each step.
For each $i \in [r]$, do the following.
For each $j \in N_R(i)$,
recall that $(G_c[V_i'',V_j'']: c \in \Cols'_{ij})$ is $(2\eps,d/2)$-superregular.
By~(\ref{eq:lp}) and \ref{partII}, we have $|\ms{B}_{ij}| \geq p_{\vx}|\Cols'_{ij}|/2$, so by Lemma~\ref{lm:slice}(i), $(G_c[V_i'',V_j'']:c \in \ms{B}_{ij})$ is $(4\eps/p_{\vx},d/4)$-regular and hence $(\sqrt{\eps}/r,d/4)$-regular.
By Lemma~\ref{lm:standard}(i), there is a set $\ov{V}^{j}_i \subseteq V''_i$ with $|\ov{V}^{j}_i| \leq \sqrt{\eps}|V_i''|/r$ such that
\begin{equation}\label{eq:Bij}
\sum_{c \in \ms{B}_{ij}}d_{G_c}(v) \geq d|V_j''||\ms{B}_{ij}|/5 \quad\text{for all }v \in V_i'' \sm \ov{V}^{j}_i.
\end{equation}
Let $\ov{V}_i := \bigcup_{j \in N_R(i)}\ov{V}^j_i$, so $|\ov{V}_i| \leq \sqrt{\eps}|V_i''|$,
and let $V^\colvx_i$ be a uniform random subset of $V_i'' \sm \ov{V}_i$ of size $n^\colvx_i$,
and let $V^\app_i := V_i'' \sm V^\colvx_i$. So $|V^\app_i| =n^\app_i$.
Let $T_y^2 := T_y^1 \cap V^{\circ'(y)}_{\phi(y)}$ for all $y \in U^1_i$ and $i \in [r]$ (i.e.~those $y$ for which $T_y^1$ has been defined), where $\circ'(y)=\app$ if $y \in V(H_\app)$ and $\circ'(y)=\colvx$ if $y \in V(H_\col) \cup V(H_\vx)$.

For all $c \in \Cols'_{ij}$, the superregularity of our graph collections implies that
\begin{align*}
\mb{E}(e(G_c[V^\colvx_i,V^\colvx_j]))&\geq(n^\colvx_i/|V_i''|)(n^\colvx_j/|V_j''|) (e(G_c[V_i'',V_j''])-\sqrt{\eps}|V_i''||V_j''|)\\
&\geq d n^\colvx_i n^\colvx_j/3\stackrel{(\ref{eq:nij})}{\geq} d p_{\vx}^2\dD^2 m^2/12.
\end{align*}
For all $i \in [r]$ and $y \in U^1_i$ we have
$$
\mb{E}(|T^2_y|) \geq (n^{\circ'(y)}_{i}/|V''_{i}|)(|T^1_y \cap V''_{i}|-\sqrt{\eps}|V''_{i}|) \stackrel{(\ref{eq:nij}),(\ref{eq:T1y}),(\ref{eq:Vi''})}{\geq} p_{\circ'(y)}\nu'm/5.
$$
Further, for all $ij \in E(R)$ and $v \in V^\colvx_i$, since $v \notin \ov{V}^j_i$ we have
$$
\textstyle\mb{E}\left(\sum_{c \in \ms{B}_{ij}}d_{G_c}(v,V^\colvx_j)\right)\geq(n^\colvx_j/|V_j''|)\sum_{c \in \ms{B}_{ij}}(d_{G_c}(v,V_j'')-\sqrt{\eps}|V_j''|) \stackrel{(\ref{eq:Bij})}{\geq} d n^\colvx_j|\ms{B}_{ij}|/11.
$$

By Chernoff bounds, we may assume that each of the above quantities are close to their expectations, so
\begin{enumerate}[label=(C\arabic*),ref=(C\arabic*)]
\item\label{c1} $|T_y^2| \geq p_{\circ'(y)}\nu'm/6$ for all $y \in U^1_i$ and $i \in [r]$ (i.e.~all those $y$ which have a target set);
\item\label{c2} $e(G_c[V^\colvx_i,V^\colvx_j]) \geq dn^\colvx_i n^\colvx_j/13$ for all $ij \in E(R)$ and $c \in \Cols'_{ij}$;
\item\label{c3} $\sum_{c \in \ms{B}_{ij}}d_{G_c}(v,V^\colvx_j) \geq dn^\colvx_j|\ms{B}_{ij}|/12$ for all $ij \in E(R)$ and $v \in V^\colvx_i$.
\end{enumerate}
Let $m' := p_{\vx} m/8$. Now~(\ref{eq:target}) and~\ref{c1} imply that, for all $i \in [r]$, 
the number $|U^1_i|$ of vertices $y \in \phi^{-1}(i)$ with a target set $T^2_y$ satisfy
\begin{equation}\label{eq:targets2}
|U^1_i| \leq 2\aA m \leq \sqrt{\aA}m',\quad\text{and}\quad
|T^2_y| \geq p_{\vx}\nu'm/6 > \nu'm'.
\end{equation}
For all $i \in [r]$ we have
$|V^\colvx_i|=n^\colvx_i=n^\col_i+n^\vx_i $, so
\begin{equation}\label{eq:Vleft}
p_{\vx}|V_i''|/3 \stackrel{(\ref{eq:Vi''})}{\leq} p_{\vx}|V_i|/2 \stackrel{(\ref{eq:nij})}{\leq} |V^\colvx_i| \stackrel{(\ref{eq:nij})}{\leq} 2p_{\vx}|V_i| \stackrel{(\ref{eq:Vi''})}{\leq} 3p_{\vx}|V_i''|.
\end{equation}



We claim that the following properties about templates $\mc{F}_\circ$ with colour sets $\Cols^\circ_e$ for $\circ \in \{\app,\col,\vx\}$ hold.
\begin{enumerate}[label=(F\arabic*),ref=(F\arabic*)]
\item\label{f1} Let
$$
\Cols^\app_e := \ms{C}'_e \sm (\ms{A}_e \cup \ms{B}_e)\text{ for all }e \in E(R).
$$
Then $\mc{F}_{\app}$, with vertex clusters $(V^\app_i: i \in [r])$, is a rainbow super $R$-template with parameters $(m/4,4\eps,d/4,\dD/4)$.
\item\label{f2} Suppose
$$
\Cols^\col_e = \ms{B}_e \cup \ms{D}_e\text{ where }\ms{D}_e \subseteq \Cols^\app\text{ for all }e \in E(R).
$$
Then
$\mc{F}_{\col}$, with vertex clusters $(V^\colvx_i: i \in [r])$, is a rainbow $R$-template with parameters $(m',\sqrt{\eps},d/4,\dD/24)$ with the property that $e(G_c[V^\colvx_i,V^\colvx_j]) \geq d|V^\colvx_i||V^\colvx_j|/22$ for all $c \in \ms{D}_e$ and $ij \in E(R)$.
\item\label{f3} Suppose
$$
\Cols^\vx_e \subseteq \ms{B}_e\text{ and }|\ms{B}_e \sm \Cols^\vx_e| \leq 2p_{\col} |\Cols_e'|\text{ for all }e \in E(R).
$$
Then $\mc{F}_{\vx}$, with colour clusters $(V^\colvx_i: i \in [r])$, is a rainbow super $R$-template with parameters $(m',\sqrt{\eps},d/13,\dD/24)$.
\end{enumerate}
These properties follow from Lemma~\ref{lm:tempslice} applied to $\mc{F}_\circ$ as a subtemplate of $\mc{F}''=(\mc{V}'',\Cols',\bm{G}'')$ (which has parameters $(m/2,2\eps,d/2,\dD/2)$), as follows.
First,~\ref{f1} holds by Lemma~\ref{lm:tempslice}(ii) since $\mc{F}_\app$ is a small perturbation of $\mc{F}''$. Indeed,~\ref{partI} and~\ref{partII} implies that
$
|\Cols'_e \sm \Cols^\app_e| \leq (2p_{\abs}+3p_{\vx})|\Cols'_e| \leq 4p_{\vx}|\Cols'_e| \leq \sqrt{p_\vx}m/2
$,
and
$|V_i'' \sm V^\app_i| = |V^\colvx_i| \leq 3p_{\vx}|V_i''| \leq \sqrt{p_\vx}m/2$ by~(\ref{eq:Vleft}),
so we can apply the lemma with $2\eps,\sqrt{p_{\vx}},d/2$ playing the roles of $\eps,\aA,d$
to obtain~\ref{f1}.

For~\ref{f2},~\ref{partII} and~(\ref{eq:lp}) imply that
$|\Cols^\col_e| \geq |\ms{B}_e| = p_e|\Cols'_e| > p_{\vx}|\Cols_e'|/4$.
Together with~(\ref{eq:Vleft}), this means that we can
apply Lemma~\ref{lm:tempslice}(i) with parameters $m/2,2\eps,p_{\vx}/4,d/2,\dD/2,12$ playing the roles of $m,\eps,\aA,d,\dD,k$ to see that
$\mc{F}_\col$ is a rainbow $R$-template with parameters $(m',\sqrt{\eps},d/4,\dD/24)$. 
The second property follows from~\ref{c2} since $\ms{D}_e \subseteq \Cols^\app_e \subseteq \Cols'_e$.

For property~\ref{f3}, we have $|\Cols^\vx_e| \geq |\ms{B}_e|-2p_{\col}|\Cols'_e| \geq p_{\vx}|\Cols_e'|/4$ by~(\ref{eq:lp}). Together with~(\ref{eq:Vleft}), this means we can apply Lemma~\ref{lm:tempslice}(i)
with the same parameters as previously to see that
$\mc{F}_{\vx}$ is a rainbow $R$-template with the given parameters.
For superregularity, we have for all $e=ij \in E(R)$ and $v \in V^\colvx_i$ that
\begin{align*}
\sum_{c \in \Cols^\vx_{ij}}d_{G_c}(v,V^\colvx_j) &\geq \sum_{c \in \ms{B}_{ij}}d_{G_c}(v,V^\colvx_j)-2p_{\col}|\Cols'_{ij}||V_j''| \stackrel{(\ref{eq:vxpart}),(\ref{eq:Bij}),\ref{c3}}{\geq} dn^\colvx_j|\ms{B}_{ij}|/12-2p_{\col}\DD m^2/\dD^2\\
&\geq dn^\colvx_j|\Cols^{\rm vx}_{ij}|/13,
\end{align*}
and $e(G_c[V^\colvx_i,V^\colvx_j]) \geq dn^\colvx_i n^\colvx_j/13$ for all $c \in \Cols^\vx_{ij}$ by~\ref{c2}.
Thus~\ref{f1}--\ref{f3} all hold.

Now it is a matter of embedding each $H_\circ$ into its corresponding template,
using suitable (unused) colours at each step.
The image of $V(H_\vx)$ in each $V_i$ will be the remaining vertices of $V^\colvx_i$
after embedding $H_\col$,
which is most of this set since $H_\col$ is much smaller than $H_\vx$.

\medskip
\noindent
{\bf Step 2.}
We embed $H_{\app}$ into $\mathcal{F}_{\app}$ using Lemma~\ref{lm:approx} (embedding lemma with extra colours) with parameters 
$m/4,4\eps,\sqrt{\mu},\sqrt{\aA},2p_{\abs}\dD,\nu'/7,d/4,\dD/4$ playing the roles of 
$m,\eps,\mu,\aA,\bB,\nu,d,\dD$. 
For this, we recall that $H_\app$ is the union of components of size at most $\sqrt{\mu}|V(H_\app)|$, and there is a graph homomorphism from $H_\app$ into $R$.
By~\ref{c1} and~\ref{f1}, it now suffices to check that colour clusters have a suitable size to apply the lemma.
Indeed, for all $e \in E(R)$ we have 
\begin{eqnarray}
\nonumber |\Cols^\app_e|-h^\app_e &\stackrel{\ref{f1}}{=}& |\Cols'_e|-|\ms{A}_e|-|\ms{B}_e|-h^\app_e
\stackrel{\ref{partI},\ref{partII}}{=}|\Cols_e'|-(h^\abs_e-\ell_e)-p_e|\Cols_e'|-h_e^{\app} \stackrel{(\ref{eq:lp})}{=} p_{\abs}|\Cols_e'|\\
\label{eq:Capp} &\in& [p_{\abs}\dD m/2,p_{\abs}\DD m/\dD].
\end{eqnarray}
Thus we can apply Lemma~\ref{lm:approx} to obtain $\tau(V(H^\app))$ and $\sS(E(H^\app))$
where each $y$ with a target set $T^2_y$ is embedded inside it.

For each $e \in E(R)$, let $\ms{D}_e := \Cols^\app_e \sm \sS(E(H_\app))$ and let $\Cols^{\col}_e := \ms{B}_e \cup \ms{D}_e$.

\medskip
\noindent
{\bf Step 3.}
We embed $H_\col$ into $\mc{F}_\col$ by applying Lemma~\ref{lm:partialcol} (embedding lemma with target sets and prescribed colours) with $m',\sqrt{\eps},\sqrt{\aA},\sqrt{p_{\abs}},\sqrt{p_{\col}},p_{\col},\nu',d/4,\dD/24$ playing the roles of
$m,\eps,\aA,\gG,\lL_1,\lL_2,\nu,d,\dD$.
To see that this is possible, we have $|V(H_\col) \cap \phi^{-1}(i)|=n^\col_i \leq 2p_\col m/\dD \leq \sqrt{p_{\col}}m'$.
We also have
$$
e(H_\col[\phi^{-1}(i),\phi^{-1}(j)]) = h^\col_{ij} \stackrel{(\ref{eq:hij})}{\geq} p_\col\dD m/2 \geq p_{\col} m'.
$$
Equation (\ref{eq:targets2}) implies that the number of vertices with target sets and the size of these target sets are suitable.
By~(\ref{eq:Capp}), for every $e \in E(R)$ we have
\begin{equation}\label{eq:De}
|\ms{D}_e|=|\Cols_e^\app|-h_e^\app \leq p_\abs\DD m/\dD \leq \sqrt{p_\abs}m'.
\end{equation}
By~\ref{f2}, it remains to check that $|\Cols^\col_e \sm \ms{D}| \geq d|\Cols^\col_e|/4$. Since the original template was rainbow, $\Cols^\col_e \sm \ms{D} = \Cols^\col_e \sm \ms{D}_e = \ms{B}_e$.
Equation~(\ref{eq:De}) implies that $|\ms{B}_e| \geq |\Cols^\col_e| - \sqrt{p_\abs}m' \geq |\Cols^\col_e|/2 \geq d|\Cols^\col_e|/4$, as required.
Thus we can apply Lemma~\ref{lm:partialcol} to obtain $\tau(V(H_\col))$ and $\sS(E(H_\col))$
where each $y$ with a target set $T^2_y$ (i.e.~those in $U^1_{\phi(y)}$) is embedded inside it.

Let $\Cols^\vx_e := \ms{B}_e \sm \sS(E(H_\col))$
and $V^\vx_i := V^\colvx_i \sm \tau(V(H_\col))$, so $V^\vx_i = n^\vx_i$.

\medskip
\noindent
{\bf Step 4.}
Let $\mc{F}_\vx'$ be the subtemplate of $\mc{F}_\vx$ induced by $(V^\vx_i: i \in [r])$.
It is a small perturbation:
$|V^\colvx_i \sm V^\vx_i| = n^\col_i \leq 2p_{\col}m/\dD \leq p_{\vx} m'$,
so Lemma~\ref{lm:tempslice}(ii) with $m',\sqrt{\eps},p_{\vx},d/13,\dD/24$
playing the roles of $m,\eps,\aA,d,\dD$ implies that $\mc{F}'_\vx$ is a rainbow super template with parameters $(m'/2,2\sqrt{\eps},d/26,\dD/48)$.
Let $T^3_y := T^2_y \cap V^\vx_i$ for all $i \in [r]$ and $y \in V^\colvx_i \cap U^1_i$.

We embed $H_\vx$ into $\mc{F}'_\vx$ by applying Lemma~\ref{lm:approx}
with parameters 
$\sqrt{\mu},\sqrt{\aA},\lL_1^2,\nu'/3$
playing the roles of $\mu,\aA,\bB,\nu$ (and template parameters as above).
For this, we recall that $H_\vx$ is the union of vertex disjoint components of size at most $\sqrt{\mu}|V(H_\vx)|$ and there is a graph homomorphism from $H_\vx$ into $R$.
By~(\ref{eq:targets2}), there are suitably few vertices with target sets $T^3_y$, and all target sets are suitably large.
By~\ref{f3}, it now suffices to check that every
$|\Cols^\vx_e| - h^\vx_e$ is large.
We have
\begin{align*}
|\Cols^\vx_e|-h^\vx_e &= |\ms{B}_e|-h^\col_e-h^\vx_e \stackrel{\ref{f1}}{=} |\Cols_e'|-|\ms{A}_e|-h^\app_e-h^\col_e-h^\vx_e
=h^\abs_e-|\ms{A}_e|=\ell_e=\lL_1 h^\abs_e\\
&\geq \lL_1p_{\abs}\dD m/2 \geq \lL_1^2 m'/2.
\end{align*}
Thus we can apply Lemma~\ref{lm:approx}
to obtain $\tau(V(H_\vx))$ and $\sS(E(H_\vx))$
where each $y$ with a target set $T^2_y$ is embedded inside it.

Let $\Cols^\abs_e := \ms{A}_e \cup (\Cols^\vx_e \sm \sS(E(H_\vx)))$.

\medskip
\noindent
{\bf Step 5.}
Since $\Cols_e' = \Cols^\app_e \cup \Cols^\col_e \cup \Cols^\vx_e \cup \Cols^\abs_e$ is a disjoint union,
and $\ms{A}_e \subseteq \Cols^\abs_e \subseteq \ms{A}_e \cup \ms{B}_e$, the set $\Cols^\abs_e \cap \ms{B}_e$ must have exactly the right size:
$$
|\Cols^\abs_e \cap \ms{B}_e| = |\Cols'_e|-h^\app_e-h^\col_e-h^\vx_e-|\ms{A}_e|=|\Cols'_e|-(|\Cols_e'|-h^\abs_e)-|\ms{A}_e|=\ell_e.
$$
Thus, by~\ref{partIII}, for each $e \in E(R)$ we can find $\sS_e$ with image $\sS_e(E(H_\abs))=\Cols^\abs_e$.
Extending $\sS$ by all of these $\sS_e$ completes the transversal embedding.
\end{proof}

\subsection{Proof of Theorem~\ref{th:blowupintro}}

\begin{proof}[Proof of Theorem~\ref{th:blowupintro}]
We may assume without loss of generality that $1/m \ll \eps,\mu,\aA \ll \nu,d,\dD,1/\DD,1/r \leq 1/2$.
Choose an additional constant $\eps'$ with $\eps \ll \eps' \ll \nu,d,\dD,1/\DD,1/r$ such that the conclusion of Lemma~\ref{lm:hsrtosr} holds with $\eps,\eps',d,\dD,3$ playing the roles of $\eps,\eps',d,\dD,k$.
We may further assume that the conclusion of Theorem~\ref{th:blowup} holds with $\eps',d^2/2$ playing the roles of $\eps,d$ and the other parameters unchanged.

Let $\bm{G},\Cols,R,\mc{V}:=(V_1,\ldots,V_r)$ be as in the statement of Theorem~\ref{th:blowupintro}. Let $\phi:V(H)\to V(R)$ be such that $\phi(x)=j$ whenever $x \in A_j$. So $\phi$ is a graph homomorphism. 
Our hypothesis is that $\mc{F}=(\mc{V},\Cols,\bm{G})$ is a rainbow half-super $R$-template with parameters $(m,\eps,d,\dD)$.
Lemma~\ref{lm:tempslice}(iv) implies that for all $c \in \Cols$ there is $G_c' \subseteq G_c$ such that, writing $\bm{G}' := (G_c':c \in \Cols)$, the template $\mc{F}'=(\mc{V},\Cols,\bm{G}')$ is super with parameters $(m,\eps',d^2/2,\dD)$.

Apply Theorem~\ref{th:blowup} to $\mc{F'}$ to obtain the required embedding.
\end{proof}

\section{Proof of Theorems~\ref{th:quasi} and~\ref{th:quasi3}}\label{sec:dense}

In this section we prove two applications of our transversal blow-up lemma, to super uniformly dense graph collections (Theorem~\ref{th:quasi}) and super uniformly dense $3$-graphs (Theorem~\ref{th:quasi3}).
The latter is an easy consequence of the former.

\begin{proof}[Proof of Theorem~\ref{th:quasi}]
Let $\DD,d,\dD,\aA>0$ be given and let $r := \DD+1$.
Choose additional constants such that
$$
0<1/n_0 \ll \eta,\mu \ll \eps \ll \dD_1 \ll \dD_2 \ll \ldots \ll \dD_{\binom{r}{2}+1} \ll \nu' \ll \aA \ll \dD,d,1/\DD,
$$
where we are assuming without loss of generality that $\aA \ll \dD,d/1/\DD$, and
so that the conclusion of Lemma~\ref{lm:hsrtosr} holds with $3,\eta^{1/4},\eps,\aA^2/6$ playing the roles of $k,\eps,\eps',d$, 
and Lemma~\ref{lm:partial} holds with $n_0/r,\eps,\sqrt{\dD_1},\aA^4/72$ playing the roles of $m,\eps,\gG,d$
and Theorem~\ref{th:blowup} holds with
$\dD_2^2 n_0/2,\sqrt{\eps},2\mu,\dD_1^{1/3},\nu',\aA^9,\dD/2$ playing the roles of
$m,\eps,\mu,\aA,\nu,d,\dD$.
Let $n \geq n_0$ be an integer and let $\bm{G}=(G_c: c \in \Cols)$ be a graph collection on a vertex set $V$ of size $n$ and $H$ a graph on $n$ vertices satisfying the conditions of the theorem.

By assumption, $\dD n \leq e(H)=|\Cols| \leq \DD n$.
Let $m := \lfloor n/r\rfloor$.
The Hajnal-Szemer\'edi theorem implies that there is a partition $V(H) = A_1 \cup \ldots \cup A_r$ into parts of size $m$ and $m+1$ such that all edges go between different parts.
Let $V=V_1 \cup \ldots \cup V_r$ be a random partition of $V$ with $|V_i|=|A_i|$ for all $i \in [r]$, and let $\mc{V} := \{V_1,\ldots,V_r\}$. We claim that with high probability, the following hold for any $ij \in E(K_r)=\binom{[r]}{2}$:
\begin{enumerate}[label=(\roman*),ref=(\roman*)]
\item\label{cc1} for all $V_h' \subseteq V_h$ with $|V_h'| \geq \eta^{1/4}|V_h|$ for $h=i,j$ and $\Cols' \subseteq \Cols$ with $|\Cols'| \geq \eta^{1/4}|\Cols|$, we have $\sum_{c \in \Cols'}e_{G_c}(V_i',V_j') \geq d|\Cols'||V_i'||V_j'|/2$;
\item\label{cc2} for each labelling $\{g,h\}=\{i,j\}$ and $v \in V_g$, we have $\sum_{c \in \Cols}d_{G_c}(v,V_h) \geq \aA^2|V_h||\Cols|/6$;
\item\label{cc3} $e(G_c) \geq \aA^2|V_i||V_j|/6$ for all $c \in \Cols$.
\end{enumerate}
For~\ref{cc1}, we have
$$
\eta n^3 < (d/2) \cdot \eta^{3/4}\dD n^3/(2r)^2 < d\eta^{3/4}\dD nm^2/2
\leq d(\eta^{1/4})^3|\Cols||V_i||V_j|/2 \leq d|\Cols'||V_i'||V_j'|/2.
$$
The statement then follows directly from the definition of $(d,\eta)$-dense. 
To prove~\ref{cc2}, for each vertex $v \in V$, there are at least $\aA|\Cols|/2$ colours $c \in \Cols$ for which $d_{G_c}(v) \geq \aA n/2$, and a Chernoff-type bound implies that, with high probability $d_{G_c}(v,V_j) \geq \aA|V_j|/3$.
Part~\ref{cc3} is proved similarly by swapping colour and vertex.

Parts~\ref{cc1}--\ref{cc3} imply that the $3$-graph $G^{(3),ij}$ of $\bm{G}^{ij} := (G_c[V_i,V_j]: c \in \Cols)$ 
 is $(\eta^{1/4},\aA^2/6)$-half-superregular.
Lemma~\ref{lm:hsrtosr} and our choice of parameters implies that every $G^{(3),ij}$ contains a spanning subhypergraph $J^{(3),ij}$ which is $(\eps,\aA^4/72)$-superregular.
Thus $\mc{F} := (\mc{V},\Cols,\bm{J})$ is a super $K_r$-template with parameters $(m,\eps,\aA^4/72,\dD)$, where $\bm{J}^{ij}$ is the graph collection whose $3$-graph is $J^{(3),ij}$ 
and $\Cols_{ij} = \Cols$ for all $ij \in E(K_r)$.

Let $d_{ij}:=e(H[A_i,A_j])/n$ for all $ij \in E(K_r)$. We have $\sum_{ij} d_{ij} =e(H)/n \geq \dD$.
Thus there is at least one $ij$ such that $d_{ij} \geq \dD_{\binom{r}{2}+1}$.
By the pigeonhole principle, there is $\ell \in [\binom{r}{2}]$ such that for all $ij$,
either $d_{ij} \leq \dD_\ell$, or $d_{ij} \geq \dD_{\ell+1}$.
Let $P^{<} := \{ij \in \binom{[r]}{2}: d_{ij} \leq \dD_\ell\}$.
We will embed those sparse $H[A_i,A_j]$, with indices $ij$ in $P^{<}$, using Lemma~\ref{lm:partial} (Embedding lemma with target and candidate sets), while the remaining (dense) pairs will be embedded using Theorem~\ref{th:blowup} (Transversal blow-up lemma) with target sets from the initial embedding of sparse pairs.

Let $X$ be the union of non-isolated vertices in $H[A_i,A_j]$ over all $ij \in P^<$,
let $Y := N_H(X) \sm X$, and let $H^< := H[X \cup Y] \sm H[Y]$, whose edge set is precisely those edges of $H$ incident to $X$.
We have $v(H^<) \leq 2\binom{r}{2}\dD_\ell n \leq \sqrt{\dD_\ell}m$.
Apply Lemma~\ref{lm:partial} to $\mc{F}$ 
with $m,\eps,\sqrt{\dD_\ell},\nu',1,\aA^4/72,\dD,\DD,r$ playing the roles of $m,\eps,\gG,\nu',\nu,d,\dD,\DD,r$, and $T_w:=V_i$ for all $w \in V(H^<) \cap A_i$ and $i \in [r]$, to find injective maps $\tau:X \to V$ and $\sS: E(H^<) \to \Cols$ such that
$\tau(x)\tau(x') \in J_{\sS(xx')}$ for all $xx' \in E(H[X])$,
for all $i \in [r]$ we have
$\tau(x) \in V_i$ for all $x \in A_i$, and for all $y \in Y \cap A_i$ there exists $C_y \subseteq V_i \sm \tau(X)$ such that $C_y \subseteq \bigcap_{x \in N_{H^<}(y) \cap X}N_{G_{\sS(xy)}}(\tau(x))$ and $|C_y| \geq \nu' m$.

Let $\Cols':= \Cols \sm \sS(E(H^<))$ and let $\mc{V}':=\{V_1',\ldots,V_r'\}$ where $V_i':= V_i\sm\tau(X)$ for each $i\in [r]$. 
By Lemma~\ref{lm:tempslice}(ii) applied with $\aA = \DD\sqrt{\dD_\ell}$, the template $(\mc{V}',\Cols',\bm{J}')$ induced by $\mc{V}',\Cols'$ is super with parameters $(m/2,2\eps,\aA^4/144,\dD/2)$.
Let $H^> := H-X$.
Note that $H^>$ (with respect to its slightly smaller vertex set) is $2\mu$-separable.
Now we randomly partition $\Cols'$ into $\binom{r}{2}$ (some perhaps empty) parts $\Cols'=\bigcup_{ij\in E(K_r)} \Cols''_{ij}$ such that $|\Cols_{ij}''|=e(H^{>}[V_i',V_j'])$
for all $ij \in \binom{[r]}{2}$. 
This is possible since $|\Cols'|=|\Cols|-e(H^<)=e(H)-e(H^<)=e(H^>)$.
Let $\mc{F}'' := (\mc{V}',\Cols'',\bm{J}')$, 
where each $\Cols''_{ij}$ is as above (so $\Cols'' = \Cols'$).
Each colour set is either large or empty, indeed, if $ij \in P^<$, then $\Cols_{ij}''=\emptyset$, 
but otherwise we have
$$
|\Cols_{ij}''| \geq e(H[A_i,A_j])-e(H^<) \geq \dD_{\ell+1}n-\DD\sqrt{\dD_\ell}m \geq \dD_{\ell+1}n/2 \geq \dD_{\ell+1}|\Cols'|/(2\DD) \geq \dD_{\ell+1}^2|\Cols'|.
$$
Lemma~\ref{lm:tempslice}(iii) applied with $\dD_{\ell+1}^2,1$ playing the roles of $\aA,k$ 
implies that $\mc{F}''$
is a rainbow super template with parameters $(\dD_{\ell+1}^2 m/2,\sqrt{\eps},\aA^9,\dD/2)$.
Apply Theorem~\ref{th:blowup} (transversal blow-up lemma) with target sets $C_y$ of size at least $\nu'm$ for the at most $\sqrt{\dD_{\ell}}m \leq \dD_\ell^{1/3}\dD_{\ell+1}^2 m/2$ vertices $y \in Y$,
and parameters
$\dD_{\ell+1}^2 m/2,\sqrt{\eps},2\mu,\dD_1^{1/3},\nu',\aA^9,\dD/2$ playing the roles of
$m,\eps,\mu,\aA,\nu,d,\dD$
to obtain a transversal embedding of $H^>$ inside $\mc{F}''$ such that every $y \in Y$ is embedded inside $C_y$.
Together with the embedding of $H^<$, this gives a transversal embedding of $H$.
\end{proof}

\begin{proof}[Proof of Theorem~\ref{th:quasi3}]
Let $\DD,d,\aA>0$ be given.
Without loss of generality, we may assume that $\aA \ll d,1/\DD$.
Choose constants $\eta',\mu,n_0>0$ such that the conclusion of Theorem~\ref{th:quasi} holds, applied with $\DD,d,1/5,\aA^3$ playing the roles of $\DD,d,\dD,\aA$.
Let $\eta := \eta'/(2\DD)$.

Let $n \geq n_0$ be an integer and let $G$ be a $(d,\eta)$-dense $3$-graph on $n$ vertices with $d_G(v) \geq \aA n^2$ for all $v \in V(G)$.

Let $H$ be a $\mu$-separable graph with $\DD(H) \leq \DD$ and $v(H)+e(H) \leq n$.
First, if $H$ is small, we will enlarge it, as follows.
If $e(H) > n/4-1$, let $H' := H$.
Otherwise, obtain $H'$ from $H$ by successively adding an edge between isolated vertices until $e(H)>n/4-1$.
Let $H'$ be the obtained graph restricted to non-isolated vertices.
We have $e(H') \leq n/4$, which implies that
$v(H')+e(H') \leq 4e(H') \leq n$.

In both cases, we have $\DD(H') \leq \DD$, and $H'$ is $\mu$-separable;
$e(H')>n/4-1$ and $v(H') \geq 3e(H')/\DD > 3(n-4)/(4\DD)>n/(2\DD)$;
and $v(H')+e(H') \leq n$.

Let $V,\Cols$ be disjoint subsets of $V(G)$ chosen uniformly at random subject to
$$
|V|=v(H') > n/(2\DD)\quad\text{and}\quad |\Cols|=e(H') > n/4-1.
$$
Standard Chernoff-type bounds imply that, with high probability,
for every $v \in V$ we have 
\begin{align*}
    e(G[\{v\},V,\Cols]) &\geq \aA^2|V||\Cols|/6 \geq \aA^2n^2/(49\DD) \geq \aA^3n^2\quad\text{for all }v \in V\\
    |\{xyc \in E(G): x,y \in V\}| &\geq \aA^2|V|^2/12 \geq \aA^2n^2/(48\DD^2) \geq \aA^3 n^2
    \quad\text{for all }c \in \Cols.
\end{align*}
 For each $c \in \Cols$, let $G_c$ be the $2$-graph with vertex set $V$ and edge set $\{xy: x,y \in V\text{ and } xyc \in E(G)\}$, and let $\bm{G}:=(G_c: c \in \Cols)$. Thus
$\sum_{c \in \Cols}d_{G_c}(v,V)=e(G[\{v\},V,\Cols]) \geq \aA^3n^2$, and for every $c \in \Cols$ we have $e(G_c) \geq \aA^3n^2$.
Since $G$ is $(d,\eta)$-dense and $|V| \geq n/(2\DD)$, we have that $\bm{G}$ is $(d,\eta')$-dense.
Theorem~\ref{th:quasi} applied with $\DD,d,1/5,\aA^3$ playing the roles of $\DD,d,\dD,\aA$ implies that $\bm{G}$
contains a transversal copy of $H'$ and thus a transversal copy of $H$, that is, there are injective maps $\tau: V(H) \to V$ and $\sS: E(H) \to \Cols$ so that $\tau(x)\tau(y) \in G_{\sS(xy)}$ for every $xy \in E(H)$.
Define $\rho: V(H) \cup E(H) \to V(G)$ by setting $\rho(x):=\tau(x)$ for $x \in V(H)$ and $\rho(e):=\sS(e)$ for $e \in E(H)$.
Then $\rho$ is injective and $\rho(x)\rho(y)\rho(xy) \in E(G)$ for all $xy \in E(H)$; that is, $\rho$ defines a copy of the $1$-expansion of $H$ in $G$, as required.
\end{proof}

\section{Concluding remarks}\label{sec:conclude}

In this paper, we have proved a transversal blow-up lemma that embeds separable graphs into graph collections, which can be used to apply the regularity-blow-up method to transversal embedding problems. We conclude with some remarks on future directions.

\medskip
\noindent
\textbf{Separability.}
In our proof of the transversal blow-up lemma, the separability condition is necessary because we need to divide the graph into linear-sized pieces and then use the blow-up lemma to embed them piece by piece. For this embedding process to work, the number of edges between different pieces should not be too large. We wonder if our transversal blow-up lemma can be generalised to embed any graph with bounded maximum degree. If such a version can be proven, it would directly generalise the original blow-up lemma for graphs (since a collection of identical superregular pairs is a superregular collection).

\medskip
\noindent
\textbf{Future applications to transversal embedding.}
In a subsequent paper, we will utilise the transversal blow-up technique developed in this paper in combination with the absorption method to provide a new proof of the transversal version of the approximate P{\'o}sa-Seymour conjecture, which was recently established in \cite{GHMPS}, and is a special case of the yet more recent main result of~\cite{CIKL}. Additionally, we will prove a stability result for transversal Hamilton cycles that has not been demonstrated before. 

However, using our method to obtain transversal versions of embedding results proved using the regularity blow-up method does not seem quite as straightforward as simply following the same proof. The regularity lemma produces some exceptional vertices which need to be incorporated into the structure built between vertex clusters: they are `absorbed' by clusters.
For transversal embedding of a spanning graph $H$, one needs to insert the exceptional vertices as well as some exceptional colours into the structure built between vertex clusters and colour clusters. Such an insertion may make some new colours exceptional.
Right at the end of the process, inserting vertices and colours simultaneously becomes difficult, when there are no more `spare colours'.
Thus, we need to construct an absorption set for the remaining vertices and colours prior to embedding $H$.

\medskip
\noindent
\textbf{Applications to hypergraph embedding.}
We state the full $3$-graph version of our transversal blow-up lemma. 

\begin{theo}[weak $3$-graph blow-up lemma]\label{th:blowup3}
Let $0 < 1/m \ll \eps,\mu,\aA \ll \nu,d,\dD,1/\DD,1/r \leq 1$.
\begin{itemize}
\item Let $R$ be a $2$-graph with vertex set $[r]$.
\item Let $G$ be a $3$-graph with parts $V_1,\ldots,V_r$ and $V_{ij}$ for $ij \in E(R)$ where $m \leq |V_i|\leq m/\dD$ for all $i \in [r]$,
and $|V_{ij}|\geq \dD m$ for all $ij \in [r]$.
Suppose that $G[V_i,V_j,V_{ij}]$ is weakly $(\eps,d)$-(half-)superregular for all $ij \in E(R)$.
\item Let $H$ be a $\mu$-separable $2$-graph with $\DD(H) \leq \DD$ for which there is a graph homomorphism $\phi: V(H) \to V(R)$ with $|\phi^{-1}(i)|=|V_i|$ for all $i \in [r]$ and $e(H[\phi^{-1}(i),\phi^{-1}(j)]) = |V_{ij}|$ for all $ij \in E(R)$.
\item Suppose that for each $i \in [r]$ there is a set $U_i \subseteq \phi^{-1}(i)$ with $|U_i| \leq \aA m$ and $T_x \subseteq V_{\phi(x)}$ with $|T_x| \geq \nu m$ for all $x \in U_i$.
\end{itemize}
Then $G$ contains a copy of the $1$-expansion of $H$, where each vertex $x \in V(H)$ is mapped to $V_{\phi(x)}$ and the new vertex for $xy \in E(H)$ is mapped to $V_{\phi(x)\phi(y)}$; and moreover, for every $i \in [r]$ every $x \in U_i$ is mapped to $T_x$.
\end{theo}

This is a reformulation of Theorem~\ref{th:blowupintro} since for each $ij \in E(R)$, the $3$-graph $G^{(3)}_{ij}$ of $(G_c: c \in \Cols_{ij})$ in Theorem~\ref{th:blowupintro} is weakly half-superregular, so we simply take $V_{ij} := \Cols_{ij}$.
It would be interesting to extend Theorem~\ref{th:quasi3} on embeddings in uniformly dense $3$-graphs beyond $1$-expansions of separable $2$-graphs.
It is not clear which $3$-graphs one should expect to be able to embed in a weakly superregular triple (or in a uniformly dense $3$-graph).
Even the case of $3$-partite $3$-graphs seems difficult; that is, to extend Theorem~\ref{lm:stran} (simplified weak hypergraph blow-up lemma).

\begin{prob}
Given $0<1/n \ll \eps \ll d \ll \dD \leq 1$, which $J$ are subhypergraphs of any weakly $(\eps,d)$-superregular triple $G[V_1,V_2,V_3]$, where $\dD n \leq |V_1|,|V_2|,|V_3| \leq n/\dD$?
\end{prob}

A $3$-partite $3$-graph $J$ is equivalent to an edge-coloured bipartite $2$-graph $\mc{J}$ (where each edge can receive multiple colours) using our usual identification of one part with a set of colours.
If $J$ has $\DD(J) \leq \DD$, then $\DD(\mc{J}) \leq \DD$ and the colouring is $\DD$-bounded.
(If $J$ is linear, then the edge-colouring is proper.)
Thus we would like to extend Theorem~\ref{th:bip} (simplified weak transversal blow-up lemma). to find an embedding of $\mc{J}$ with a given edge-colouring up to permutation of colours.
If every colour only appears on a single edge, that is, $J$ is an expansion of some $2$-graph $H$ where the $2$-edge $e$ is replaced by some bounded number $t_e \geq 1$ of $3$-edges (which is not linear if some $t_e>1$), it seems plausible that our techniques would work.
However, the colour absorption we use from~\cite{MMP} does not seem to be able to deal with colours playing multiple roles.

Let $F$ be a $3$-partite $3$-graph of fixed size.
Recall that an old result of Erd\H{o}s~\cite{erdos1964extremal} implies that any large $3$-graph of positive density contains a copy of $F$.
However, the results of~\cite{ding2022f} imply that any large uniformly dense $3$-graph $G$ (whose number of vertices is a multiple of $v(F)$ and where every vertex sees a positive fraction of pairs) contains an $F$-factor if and only if there is a vertex $v^* \in V(F)$ such that for any two edges $e,e'$ where $e$ contains $v^*$ and $e'$ does not,
$e$ and $e'$ share at most one vertex.
Probably, the characterisation for weakly superregular triples with balanced classes is the same.

Another instructive case is the tight Hamilton cycle (whose number of vertices is divisible by $6$). As a graph collection problem, this corresponds to finding a ($2$-graph) Hamilton cycle on $4n$ vertices with $2n$ colours, where, cyclically labelling the edges $e_1,\ldots,e_{4n}$, the consecutive edges $e_{2c-1},e_{2c},e_{2c+1}$ are all in $G_c$, for each $c \in [2n]$ where indices are taken modulo $4n$. A construction in~\cite{lenz2016perfect} (see also~\cite{lenz2016hamilton}) shows that a weakly superregular triple need not contain a tight Hamilton cycle, even if its density is close to $\frac{1}{8}$.

\begin{eg}
Let $V$ be a vertex set of size $6n$, and let $V=V_1 \cup V_2 \cup V_3$ be a partition, where $V_1,V_2,V_3$ have equal size $2n$, and let $X \subseteq V$ where $|X \cap V_1|$ is odd.
Independently for each distinct $i,j \in [3]$ add uniform random edges between parts $V_i,V_j$ with probability $\frac{1}{2}$, to obtain a $3$-partite $2$-graph $J$.
Now form a $3$-partite $3$-graph $G$ by, for each triple $abc \in A \times B \times C$, as follows:
\begin{itemize}
    \item add $abc$ to $G$ if $|\{a,b,c\} \cap X|$ is even and $\{a,b,c\}$ spans a triangle in $J$.
    \item add $abc$ to $G$ if $|\{a,b,c\} \cap X|$ is odd and $\{a,b,c\}$ spans an independent set in $J$.
    \end{itemize}
Suppose $v^1_1v^1_2v^1_3\ldots v^{2n}_1v^{2n}_2v^{2n}_3$ is a tight cycle $H$ in $J$, where without loss of generality, $v^j_i \in V_i$ for all $j$. For any $1 \leq j < 2n$, Since $v^j_1v^j_2v^j_3, v^j_2v^j_3v^{j+1}_1$ are edges of $H$, we have that
$|\{v^j_1,v^j_2,v^j_3\} \cap X|, |\{v^{j+1}_1,v^j_2,v^j_3\} \cap X|$ are both odd or both even. This implies that $v^j_1,v^{j+1}_1$ are either both in $X$ or neither in $X$. 
By considering the disjoint pairs $v_1^1v_1^2,v_1^3v_1^4,\ldots,v_1^{2n-1}v_1^{2n}$,
this is a contradiction to $|X \cap V_1|$ odd.

It can easily be checked using Chernoff bounds (see~\cite{lenz2016perfect}) that, if every $|X \cap V_i|$ has size about $n$, then with high probability $G$ is weakly $(\frac{1}{8},o(1))$-superregular.
\end{eg}

\bibliographystyle{plain}
\bibliography{Bibte}

\appendix

\section{Appendix}

\subsection{Weak hypergraph regularity lemmas}

We first introduce the following special form of the weak hypergraph regularity lemma for $k$-partite $k$-graphs. Its proof follows easily from the original lemma of Chung~\cite{fanweak}, which is very similar to the $2$-graph regularity lemma of Szemer\'edi~\cite{Szemeredi}.

\begin{lemma}\label{lm:3pweakreg}
For every $L_0,k\geq 1$ and every $\eps,\dD>0$, there is an $n_0>0$ such that for every $k$-partite $k$-graph $G$ on $n \geq n_0$ vertices with parts $V_1,\ldots,V_k$ where $\dD n\leq |V_i|\leq n/\dD$, there exists a refined partition $V_i=\bigcup_{j=1}^{t_i}V_{i,j}$ for each $i\in [k]$ such that the following properties hold:
\begin{enumerate}
\item[(i)] $L_0 \leq t:=t_1+\ldots+t_k \leq n_0$;
\item[(ii)] $\left| |V_{i,j}|-|V_{i',j'}|\right| \leq 1$ for any $i,i'\in[k]$, $j\in [t_i]$ and $j'\in [t_{i'}]$;
\item[(iii)] all but at most $\eps t^k$ $k$-tuples $(V_{1,j_1},\ldots,V_{k,j_k})$ are $\eps$-regular.
\end{enumerate}
\end{lemma}

The following weak hypergraph regularity lemma~\cite{KOT} is proved in the same way as the original
degree form of the regularity lemma (see~\cite{Townsend}), which in turn can be derived from the standard regularity lemma via some cleaning.

\begin{theo}[Degree form of the weak hypergraph regularity lemma]\label{lm:weakreg}
For all integers $L_0 \geq 1$ and every $\eps>0$, there is an $n_0=n_0(\eps,L_0)$ such that
for every $d \in [0,1)$ and for every $3$-graph $G=(V,E)$ on $n \geq n_0$ vertices there
exists a partition of $V$ into $V_0,V_1,\ldots,V_L$ and a spanning subhypergraph $G'$ of $G$ such
that the following properties hold:
\begin{enumerate}
\item[(i)] $L_0 \leq L \leq n_0$ and $|V_0| \leq \eps n$;
\item[(ii)] $|V_1|=\ldots=|V_L|=: m$;
\item[(iii)] $d_{G'}(v) > d_G(v)-(d+\eps)n^{2}$ for all $v \in V$;
\item[(iv)] every edge of $G'$ with more than one vertex in a single cluster $V_i$ for some $i \in[L]$
has at least one vertex in $V_0$;
\item[(v)] for all triples $\{h,i,j\} \in \binom{[L]}{3}$, we have that
$(V_h,V_i,V_j)_{G'}$ is either empty or $(\eps,d)$-regular.
\end{enumerate}
\end{theo}

The proof of Lemma~\ref{lm:weakregcol} (the regularity lemma for graph collections) is a routine but tedious consequence of this theorem applied to the $3$-graph $G^{(3)}$ of a graph collection $\bm{G}$.

\begin{proof}[Proof of Lemma~\ref{lm:weakregcol}]
By increasing $L_0$ and decreasing $\eps,\dD$ as necessary, we may assume without loss of generality that $0 < 1/L_0 \ll \eps \ll \dD \ll 1$.
Let $L_0' := 2L_0/\dD$.
Choose new constants $\eps',\aA$, which we may assume satisfy $0 < 1/L_0' \ll \eps' \ll \aA \ll \eps$.
Let $n_0'$ be obtained from Theorem~\ref{lm:weakreg} applied with parameters $\eps',L_0'$.
We may assume that $1/n_0' \ll 1/L_0'$ by increasing $n_0'$.
Altogether,
$$
0 < 1/n_0' \ll 1/L_0' \ll \eps' \ll \aA \ll \eps \ll \dD \ll 1.
$$
Let $n_0 := 2n_0'/\aA$.
Let $n \geq n_0$ be an integer and suppose that $\bm{G}$ is a graph collection on a vertex set $V$ of size $n$ and with colour set $\Cols$, where $\dD n \leq |\Cols| \leq n/\dD$.
Let $d \in [0,1)$ and let $G^{(3)}$ be the $3$-graph of $\bm{G}$ with vertex set $U := V \cup \Cols$.
Theorem~\ref{lm:weakreg} implies that there is a partition $U_0,U_1,\ldots,U_K$ of $U$
and a spanning subhypergraph $G^{(3)'}$ of $G^{(3)}$ such that
\begin{enumerate}[label=(\alph*),ref=(\alph*)]
\item\label{reg1} $L_0' \leq K \leq n_0'$ and $|U_0| \leq \eps'|U|$;
\item\label{reg2} $|U_1|=\ldots=|U_K|=: m'$;
\item\label{reg3} $d_{G^{(3)'}}(u) > d_{G^{(3)}}(u)-(2d+\eps')|U|^2$ for all $u \in U$;
\item\label{reg4} every edge of $G^{(3)'}$ with more than one vertex in a single cluster $U_i$ has at least one vertex in $U_0$;
\item\label{reg5} for all triples $\{h,i,j\} \in \binom{[K]}{3}$, we have that $(U_h,U_i,U_j)_{G^{(3)'}}$ is either empty or $(\eps',2d)$-regular.
\end{enumerate}
Partition each cluster $U_i$, $i \in [K]$, into $1/\aA$ subclusters of size at most $m$ so that all but at most two subclusters of $U_i$ have size exactly $m$ and the property that they lie entirely within $V$, or within $\Cols$.
If a subcluster does not have this property, add it to $U_0$.
The new exceptional set has size at most $|U_0|+2\aA m'K$, and we let $V_0$ be its intersection with $V$, and $\Cols_0$ its intersection with $\Cols$.
Relabel the subclusters so that those which are subsets of $V$ are $V_1,\ldots,V_L$ and those which are subsets of $\Cols$ are $\Cols_1,\ldots,\Cols_M$.
Let $G'_c$ be the graph with vertex set $V$ and edge set $\{xy: xyc \in G^{(3)'}\}$ for all $c \in \Cols$.
We claim that the properties of the lemma are satisfied.

For~(i), we have $m'K \leq |U|$ so
$$
|V_0|+|\Cols_0| \leq |U_0|+2\aA m'K \stackrel{\ref{reg1}}{\leq} (\eps'+2\aA)|U| \leq (\eps'+2\aA)(n+n/\dD) \leq \eps n.
$$
Also,
$L_0 \leq L_0' \leq L_0'/\aA \leq K/\aA \leq L+M \leq 2K/\aA \leq 2n_0'/\aA=n_0$,
proving the required upper bound for both $L$ and $M$.
Furthermore, $m(L+M) \leq n+|\Cols|$
so
$$
L \geq \frac{n-|V_0|}{m }\geq \frac{(n-|V_0|)(L+M)}{n+|\Cols|} \geq \frac{(n-|V_0|)L_0'}{n+n/\dD} \geq \dD L_0'/2 = L_0
$$
and similarly for $M$.
Part~(ii) follows by construction.

For~(iii), we have $\sum_{c \in \Cols}d_{G_c'}(v)=d_{G^{(3)'}}(v)$ for all $v \in V$, and $e(G_c')=d_{G^{(3)'}}(c)$ for all $c \in \Cols$,
and similarly for $G_c$. Further, $(2d+\eps')(n+|\Cols|)^2<(2d+\eps')(1+1/\dD)^2n^2\leq(3d/\dD^2+2\eps'/\dD^2)n^2 \leq (3d/\dD^2+\eps)n^2$. Thus~\ref{reg3} implies the required.

For~(iv), if $G'_c$ has an edge $xy$ with $x,y \in V_i$ for some $i \in [L]$, then $x,y \in U_{i'}$ for the cluster $U_{i'}$ containing $V_i$.  Since $xyc \in E(G^{(3)'})$,~\ref{reg4} implies that $c \in U_0$. Thus $c \in \Cols_0$, as required. 

For~(v), suppose that $(\{h,i\},j) \in \binom{[L]}{2}\times M$ and $\bm{G}'_{hi,j}$ is non-empty.
Let $h',i',j' \in [K]$ be such that $V_h \subseteq U_{h'}$, $V_i \subseteq U_{i'}$ and $\Cols_j \subseteq U_{j'}$.
Since $\bm{G}'_{hi,j}$ is non-empty, we have that $(U_{h'},U_{i'},U_{j'})_{G^{(3)'}}$ is non-empty.
So $h',i',j'$ are distinct by~\ref{reg4}. 
Part~\ref{reg5} implies that $(U_{h'},U_{i'},U_{j'})_{G^{(3)'}}$ is $(\eps',2d)$-regular.
Therefore, for each $c \in U_{j'}$, letting $J_c$ be the bipartite graph with partition $(U_{h'},U_{i'})$ and edge set $\{xy: x \in U_{h'}, y \in U_{i'}, xyc \in G^{(3)'}\}$, we have that $\bm{J} := (J_c: c \in U_{j'})$ is $(\eps',2d)$-regular.
Lemma~\ref{lm:slice}(i) (the slicing lemma) implies that $\bm{G}'_{hi,j}$ is $(\eps'/\aA,d)$-regular and thus $(\eps,d)$-regular.
This completes the proof.
\end{proof}

We conclude the appendix with a proof of Lemma~\ref{lm:hsrtosr}, which states that a half-superregular ($k$-graph) $k$-tuple contains a spanning superregular subhypergraph. The proof is the same as the one in~\cite{RR} for $k=2$.

\begin{proof}[Proof of Lemma~\ref{lm:hsrtosr}]
Choose a new parameter $\eps''$ such that $0<\eps \ll \eps'' \ll \eps'$.
Apply Lemma \ref{lm:3pweakreg} to $G$ with parameters $L_0 := 1, \eps'',\dD$ to obtain $n_0>0$. Increasing $n_0,n$ and decreasing $\eps$ if necessary, we may assume that $1/n \ll \eps \ll 1/n_0 \ll \eps''$.
Now let $G=(V_1,\ldots,V_k)_G$ be an $(\eps,d)$-half-superregular $k$-graph with $\dD n \leq |V_i| \leq n/\dD$ for all $i \in [k]$.
Lemma~\ref{lm:3pweakreg} implies that there is a refinement $V_i=\bigcup_{j=1}^{t_i}V_{i,j}$ for each $i\in [k]$ such that 
$t := t_1+\ldots+t_k \leq n_0$, every pair of subparts differ in size by at most one, and
all but at most $\eps'' t^3$ $k$-tuples $(V_{1,j_1},\ldots,V_{k,j_k})$ are $\eps''$-regular.

Obtain a spanning subhypergraph $G'$ of $G$ by, for each $\eps''$-regular triple $(X_1,\ldots,X_k)$ of subparts,
removing every $k$-edge in $(X_1,\ldots,X_k)$ independently at random with probability $1-d/d'$, where $d'$ is the density of $(X_1,\ldots,X_k)$. 
(Note that by $|X_i|\geq |V_i|/n_0 \geq \eps |V_i|$ for each $i\in [k]$, we have $d'\geq d$ due to half-superregularity.)
We claim that, with high probability,
\begin{itemize}
\item $d_{G'}(v)\geq d^2(|V_1|\ldots|V_k|)/(2|V_i|)$ for any $i\in[k]$ and $v\in V_i$,
\item any $\eps''$-regular $k$-tuple (of subparts) in $G$ is $2\eps''$-regular in $G'$, 
\item any $\eps''$-regular $k$-tuple (of subparts) in $G$ has density $d \pm 2\eps''$ in $G'$.
\end{itemize}
This follows from Chernoff bounds.

Now we claim that, given that the above hold, $G'$ is $(\eps',d^2/2)$-superregular. We only need to check that $G'$ is $(\eps',d^2/2)$-regular. Let $(A_1,\ldots,A_k)$ be any $k$-tuple where $A_i\subseteq V_i$ and $|A_i|\geq \eps'|V_i|$ for all $i\in [k]$. 
For each $i\in [k]$, we have partition $A_i=\bigcup_{j=1}^{t_i}(A_i\cap V_{i,j})$. We say a part $A_{i,j} := A_i\cap V_{i,j}$ is \emph{big} if $|A_i\cap V_{i,j}|\geq \eps'' |V_{i,j}|$ and otherwise it is \emph{small}. Note that in total we have at most $\eps'' (|V_1|+\ldots+|V_k|)\leq \eps'' kn/\dD$ vertices in small parts. 
Also note that we have at most $\eps'' t^k$ non-regular $k$-tuples that contain in total at most $\eps'' (n/\dD)^k$ edges. 
Thus
\begin{align*}
e_{G'}(A_1,\ldots,A_k) = \sum_{\substack{(A_{1,j_1},\ldots,A_{k,j_k}) \text{ all big and}\\
(V_{1,j_1},\ldots,V_{k,j_k})~\eps''\text{-regular}}}e_{G'}(A_{1,j_1},\ldots,A_{k,j_k})
\pm \left(2\eps'' kn/ \dD \cdot (n/\dD)^{k-1} + \eps'' (n/\dD)^k\right).
\end{align*}
For every summand index $(j_1,\ldots,j_k)$, we have $e_{G'}(A_{1,j_1},\ldots,A_{k,j_k}) = (d\pm 2\eps'')|A_{1,j_1}|\ldots|A_{k,j_k}|$ by $2\eps''$-regularity.
Thus
$$
e_{G'}(A_1,\ldots,A_k) = (d \pm 2\eps'')|A_1|\ldots|A_k| \pm 3\eps'' kn^k/\dD^k 
= (d \pm \eps')|A_1|\ldots|A_k|.
$$
This implies that $G'$ is $(\eps',d^2/2)$-superregular and thus we finish the proof.
\end{proof}

\end{document}